\documentclass{amsart}
\usepackage{verbatim}
\usepackage[textsize=scriptsize]{todonotes}
\usepackage{mathtools}
\usepackage{xcolor}
\usepackage{tikz-cd}
\usetikzlibrary{decorations.pathmorphing}

\usepackage{amsmath}
\usepackage{amssymb}
\usepackage{amsthm}
\usepackage{amscd}
\usepackage{enumerate}
\usepackage[pdfusetitle,unicode,hidelinks]{hyperref}
\usepackage{bbm}
\usepackage{etoolbox}

\usepackage[utf8]{inputenc}

\newtheorem{proposition}{Proposition}
\newtheorem{corollary}[proposition]{Corollary}
\newtheorem{lemma}[proposition]{Lemma}
\newtheorem{theorem}[proposition]{Theorem}
\newtheorem{conjecture}[proposition]{Conjecture}
\newtheorem*{conjecture*}{Conjecture}
\newtheorem*{theorem*}{Theorem}
\newtheorem*{corollary*}{Corollary}
\newtheorem*{proposition*}{Proposition}
\newtheorem*{lemma*}{Lemma}
\theoremstyle{definition}
\newtheorem{definition}[proposition]{Definition}

\newtheorem*{definition*}{Definition}
\newtheorem*{construction*}{Construction}
\theoremstyle{remark}
\newtheorem{remark}[proposition]{Remark}
\newtheorem{example}[proposition]{Example}

\newcommand{\id}{\operatorname{id}}
\newcommand{\Z}{\mathbb{Z}}

\newcommand{\Hh}{\mathbb{H}}

\let\scr=\mathcal
\let\bb=\mathbb
\newcommand{\Gm}{{\mathbb{G}_m}}
\newcommand{\Gmp}[1]{{\mathbb{G}_m^{\wedge #1}}}
\def\A{\bb A}
\def\P{\bb P}

\newcommand{\eff}{{\text{eff}}}
\newcommand{\veff}{{\text{veff}}}
\newcommand{\DM}{\mathcal{DM}}
\newcommand{\SH}{\mathcal{SH}}
 \newcommand{\HH}{\mathcal{H}}
\newcommand{\Spc}{\mathrm{Spc}}

\DeclareMathOperator*{\colim}{colim}

\let\lim=\relax
\DeclareMathOperator*{\lim}{lim}
\def\Map{\mathrm{Map}}

\def\PSh{\mathcal{P}}

\def\Spc{\mathcal{S}\mathrm{pc}{}}

\def\Fun{\mathrm{Fun}}
\def\Fr{\mathrm{Fr}}
\def\red{\mathrm{red}}
\def\Ab{\mathrm{Ab}}
\newcommand{\spk}{\mathop{\mathrm{Spec} \, k}\nolimits}

\newcommand{\wequi}{\simeq}
\newcommand{\Mod}{\text{-}\mathcal{M}od}
\def\adj{\leftrightarrows}

\DeclareRobustCommand{\ul}{\underline}

\DeclareRobustCommand{\wh}{\widehat}
\DeclareRobustCommand{\wt}{\widetilde}
\newcommand{\heart}{\heartsuit}
\newcommand{\tr}{\mathrm{tr}}

\def\op{\mathrm{op}}
\def\pr{\mathrm{pr}}
\newcommand{\HI}{\mathbf{HI}}

\let\cat=\mathrm
\def\Sm{{\cat{S}\mathrm{m}}}

\def\Nis{\mathrm{Nis}}
\def\Zar{\mathrm{Zar}}

\def\mot{\mathrm{mot}}

\numberwithin{proposition}{section}

\usepackage{etoolbox}
\newtoggle{draft}

\togglefalse{draft}

\iftoggle{draft} {
\usepackage[margin=1.5in]{geometry}

\newcommand{\NB}[1]{\todo[color=gray!40]{#1}}
\newcommand{\TODO}[1]{\todo[color=red]{#1}}
\newcommand{\tom}[1]{\todo[color=green]{#1}}
\newcommand{\mura}[1]{\todo[color=yellow]{#1}}
}{ 
\usepackage[margin=1in]{geometry}

\newcommand{\NB}[1]{}
\newcommand{\TODO}[1]{}
\newcommand{\tom}[1]{}
\newcommand{\mura}[1]{}
\renewcommand{\todo}[1]{}
}
\usepackage{hyperref}

\newcommand{\lra}[1]{\langle #1 \rangle}
\newcommand{\SHS}{\mathcal{SH}^{S^1}\!}
\newcommand{\qgood}{\text{ $q$-good}}
\newcommand{\Atr}{\mathrm{\A^1tr}}
\newcommand{\tw}{\mathrm{tw}}
\newcommand{\fr}{\mathrm{fr}}
\newcommand{\gp}{\mathrm{gp}}
\newcommand{\codim}{\mathrm{codim}}
\newcommand{\E}{\mathcal{E}}

\newcommand{\CH}{\mathrm{CH}}
\def\Cor{\mathrm{Corr}}
\def\h{\mathrm h}
\def\Pre{\mathrm{Pre}}
\newcommand{\sslash}{\mathbin{/\mkern-6mu/}}

\def\ph{\mathord-}

\title{Towards conservativity of $\Gm$-stabilization}
\author{Tom Bachmann}
\address{Department of Mathematics, Massachusetts Institute of Technology, Cambridge, MA, USA}
\email{tom.bachmann@zoho.com}
\author{Maria Yakerson}
\address{Fakultät Mathematik, Universität Regensburg, Regensburg, Germany}
\email{maria.yakerson@ur.de}
\urladdr{\url{https://www.muramatik.com}}
\thanks{M.Y.\ was supported by SFB/TR 45 ``Periods, moduli spaces and arithmetic of algebraic varieties"}

\begin{document}
\maketitle

\begin{abstract}
We study the interplay of the homotopy coniveau tower, the Rost-Schmid complex of a strictly homotopy invariant sheaf, and homotopy modules. For a strictly homotopy invariant sheaf $M$, smooth $k$-scheme $X$ and $q\ge0$ we construct a new cycle complex $C^*(X, M, q)$ and we prove that in favorable cases, $C^*(X,M,q)$ is equivalent to the homotopy coniveau tower $M^{(q)}(X)$. To do so we establish moving lemmas for the Rost-Schmid complex. As an application we deduce a cycle complex model for Milnor-Witt motivic cohomology. Furthermore we prove that if $M$ is a strictly homotopy invariant sheaf, then $M_{-2}$ is a homotopy module. Finally we conjecture that for $q>0$, $\ul{\pi}_0(M^{(q)})$ is a homotopy module, explain the significance of this conjecture for studying conservativity properties of the $\Gm$-stabilization functor $\SHS(k) \to \SH(k)$, and provide some evidence for the conjecture.
\end{abstract}

\tableofcontents

\section{Introduction}

\subsection*{Motivation}

Classically, a very helpful fact for understanding invariants of topological spaces is the following: the functor $\Spc_* \to \mathcal{D}(\mathrm{Ab})$ sending a space to its singular chain complex is conservative on simply connected spaces.\footnote{Here and below, whenever we call a functor conservative we always mean that the induced functor on homotopy categories is conservative.}
Using this fact, one reduces homotopy-theoretic questions  to questions in homological algebra, which are in general more accessible. 
This conservativity statement can be split in two: the functor $\Sigma^\infty \colon \Spc_* \to \SH$ is conservative on simply connected spaces, and the canonical  functor $\SH \to \mathcal{D}(\mathrm{Ab})$ is conservative on bounded below spectra.

We would like to address the analogous question in the motivic context. We work over a perfect base field $k$. Recall the category of pointed motivic spaces $\HH(k)_*$, given by pointed $\A^1$-invariant Nisnevich sheaves (of spaces) on the category of smooth $k$-schemes $\Sm_k$. By stabilizing $\HH(k)_*$, one obtains the category of motivic $S^1$-spectra $\SHS(k)$ \cite[Section 4]{morel-trieste}, and then the motivic stable homotopy category $\SH(k)$, by further stabilization with respect to $\Gm$. 
Work of the first author \cite{bachmann-hurewicz} investigates the problem of conservativity of the functor of taking the associated motive, i.e. the functor $\SH(k) \to \DM(k)$. As in the classical situation, the functor $\Sigma^\infty_{S^1}$ is conservative on simply connected motivic spaces \cite[Corollary 2.23]{wickelgren2014simplicial}. Hence the main remaining question is the following: up to which extent is the functor 
$\sigma^\infty := \Sigma^\infty_{\Gm} \colon \SHS(k) \to \SH(k)$ conservative? 
We believe that this functor is \emph{conservative after $\Gm$-suspension}, i.e. that $\Sigma^{\infty-1}_{\Gm} \colon \SHS(k)(1) \to \SH(k)$ is conservative on bounded below objects, where $\SHS(k)(n)$ for $n \geqslant 0$ denotes the localizing subcategory of 
$\SHS(k)$, generated by $\Sigma_{\Gm}^n \Sigma^\infty_{S^1} X_+$ for $X \in \Sm_k$. By a standard argument using $t$-structures (see e.g. \cite[Corollary 4]{bachmann-hurewicz}) this would be a corollary of the following statement.

\begin{conjecture}[See Conjecture \ref{conj:equivalence}] \label{conj:intro}
Let $n \geqslant 1$. Then the canonical functor \[ \Sigma^{\infty-n\heart}_{\Gm} \colon \SHS(k)(n)^\heart \to \SH(k)^{\eff\heart} \] is an equivalence of abelian categories.
\end{conjecture}
Here $\SH(k)^{\eff}$ is the localizing subcategory generated by the image of $\SHS(k)$ in $\SH(k)$ under $\Sigma^{\infty}_{\Gm}$, and the hearts are taken with respect to homotopy $t$-structures on these categories (note that $\SH(k)^{\eff\heart}$ is not equivalent to $\SH(k)^{\eff} \cap \SH(k)^{\heart}$). In this work, among other things, we show some evidence for this conjecture.

\begin{remark}
The functor $\Sigma^{\infty\heart}_{\Gm}\colon \SHS(k)^\heart \to \SH(k)^{\eff\heart}$ is not an equivalence (i.e. Conjecture \ref{conj:intro} is false with $n=0$), since its right adjoint $\omega^{\infty\heart}$ is not essentially surjective: there exists a strictly homotopy invariant sheaf that does not admit framed transfers (a minor modification of the arguments from \cite[Section 2]{levine2010slices} provides an example). It could be possible that $\Sigma^\infty_\Gm\colon \SHS(k)_{\ge 0} \to \SH(k)$ (or even its extension to unbounded spectra) is nonetheless conservative, but that seems somewhat unlikely\NB{does it?}.
\end{remark}

\begin{remark}
To the best of our knowledge, Conjecture \ref{conj:intro} is the first concrete suggestion for a ``Freudenthal $\Gm$-suspension theorem''.\footnote{Levine's result \cite[Theorem 7.4.2 implying Conjecture 7.4.1]{levine2008homotopy} is arguably a ``$\Gm$-(-1)-connected'' Freudenthal $\Gm$-suspension theorem: it states that for $E \in \SHS(k)(n)$ the map $E \to \Omega_\Gm \Sigma_\Gm E$ induces an isomorphism on slices $s_i$ for $i < n$. In this terminology Conjecture \ref{conj:intro} asks for a ``$S^1$-0-connected'' Freudenthal $\Gm$-suspension theorem.}
\end{remark}

\subsection*{Some recollections about motivic stable homotopy theory}
The heart of the $t$-structure on $\SHS(k)$ is given by $\SHS(k)^\heart \wequi \HI(k)$, the category of \emph{strictly homotopy invariant (Nisnevich) sheaves} \cite[Lemma 4.3.7]{morel-trieste}. Furthermore $\SH(k)^\heart \wequi \HI_*(k)$, the category of \emph{homotopy modules} \cite[Theorem 5.2.6]{morel-trieste}. An object in $\HI_*(k)$ is given by a sequence of strictly homotopy invariant sheaves $\{M_i\}_{i \in \Z}$ with isomorphisms $M_i \simeq (M_{i+1})_{-1}$, where $M_{-1}$ is the so called \emph{contraction} of a sheaf of abelian groups $M$, which models $\Omega_\Gm M$. The functor $\omega^\infty\colon \SH(k) \to \SHS(k)$ right adjoint to $\sigma^\infty$ induces a ``forgetful'' functor $\omega^\infty\colon \HI_*(k) \to \HI(k)$ which sends $\{M_i\}_{i \in \Z}$ to $M_0$. We say that a strictly homotopy invariant sheaf is a homotopy module if it belongs to the essential image of $\omega^\infty$.

For any $E \in \SHS(k)$, $X \in \Sm_k$ and $q \ge 0$, Levine defines a simplicial spectrum $E^{(q)}(X, \bullet)$ called the ($q$-th level of the) \emph{homotopy coniveau tower of $E$} (at $X$) \cite{levine2008homotopy}. This is a generalization of Bloch's cycle complex. Without recalling the entire construction of $E^{(q)}(X, \bullet)$ here, let us note that if $E = M \in \SHS(k)^\heart \wequi \HI(k)$ corresponds to a strictly homotopy invariant sheaf, then $E^{(q)}(X, n)$ is built out of complexes of the form $R\Gamma_Z(Y, M)$, where $Y$ is a smooth scheme, $Z \subset Y$ is closed of codimension $\ge q$, and $R\Gamma_Z$ denotes derived global sections with support. The starting point of this article is the observation that there is a very efficient complex which can be used to compute $R\Gamma_Z(Y, M)$, namely the Rost-Schmid complex (with support) $C^*_Z(Y, M)$ constructed by Morel \cite[Chapter 5]{A1-alg-top}. Defined by means of explicit transfer maps on contractions of $M$, the Rost-Schmid complex is canonically isomorphic to the Gersten complex. We recall features of the Rost-Schmid complex in Section \ref{sec:RS}.

\subsection*{The Bloch-Levine-Rost-Schmid complex}
For $M \in \HI(k)$, $X \in \Sm_k$ and $q \ge 0$, in Section \ref{sec:BLRS} we construct an explicit complex $C^*(X, M, q)$ and establish some functoriality properties. It is obtained by splicing together the Rost-Schmid complex $C^{* \ge q}(X, M)$ with a degreewise truncated version of the homotopy coniveau tower (see Definition \ref{def:BLRS} for details). Here is our first main result, which provides an explicit model for the homotopy coniveau tower of a homotopy module.

\begin{theorem}[Simplified version; see Theorem \ref{thm:comparison}] \label{thm:intro-1}
Let $k$ be a perfect field, $M$ a homotopy module and $X \in \Sm_k$ be affine with trivial canonical line bundle $\omega_X$. Then there is a canonical equivalence of spectra \[M^{(q)}(X) \wequi C^*(X, M, q),\]
where $M^{(q)}(X)$ denotes the geometric realization of the simplicial spectrum $M^{(q)}(X, \bullet)$, and the complex $C^*(X, M, q)$ is considered as an $H\Z$-module.
\end{theorem}

Let us point out that if $M = \ul{K}_q^M$, then $C^*(X, M, q)$ is Bloch's cycle complex, and in this case our theorem recovers the comparison of higher Chow groups and motivic cohomology for such $k$-schemes $X$. 
As a further application, we provide a cycle-complex presentation of Milnor-Witt motivic cohomology (also called ``generalized motivic cohomology"), defined in~\cite{calmes2014finite}. 
\begin{corollary}[see Corollary \ref{corr:MW-mot-coho}]
Let $X$ be a smooth scheme over a perfect field $k$ of characteristic $\ne 2$. Then for any $q \ge 0$, $i \in \Z$ there is a canonical isomorphism
\[ H^{q+i}(X, \tilde{\Z}(q)) \wequi \Hh^i_\Zar(X, C^*(\ph, \ul{K}_q^{MW}, q)). \]
\end{corollary}

The key ingredients for the proof of Theorem \ref{thm:intro-1}  are moving-lemma type results for the Rost-Schmid complex of a homotopy module, which we prove in Section \ref{sec:q-good-RS}. The proofs follow the familiar pattern established by Levine in \cite{levine2006chow}, suitably adapted for the Rost-Schmid complex, of course. Let us mention for the expert reader that the usual proof of the so-called ``hard moving lemma'' uses transfers. This is why we need to ask that $M$ should be a homotopy module. Moreover, transfers on homotopy modules depend on certain orientations; this is why we need to ask that $\omega_X$ should be trivial.

\subsection*{Contractions and homotopy modules}
Both the homotopy coniveau tower $M^{(q)}$ and the complex $C^*(X, M, q)$ make sense for $M \in \HI(k)$, not just $M \in \HI_*(k)$. It seems natural to ask if the equivalence $C^*(X, M, q) \wequi M^{(q)}(X)$ might extend to such more general $M$. Indeed one has the feeling that, in essence, $C^*(X, M, q)$ only depends on the $q$-fold contraction $M_{-q}$, and for $q>0$, $M_{-q}$ already has some kinds of transfers \cite[Chapter 4]{A1-alg-top}. Instead of re-doing the moving lemmas in this more complicated context, we make two observations: (1) the same proof as before shows that $C^*(X, M, q) \wequi M^{(q)}(X)$, provided that $M_{-q}$ is a homotopy module (even if $M$ is not); and (2) $M_{-q}$ is often a homotopy module, for $q > 0$.

We should justify claim (2); in fact this is the content of Section \ref{sec:generalized-transfers}. 
Using the recent developments in motivic infinite loop space theory, in particular \cite{EHKSY}, it is possible to identify the homotopy modules among strictly homotopy invariant sheaves. To explain this, recall the subcategory of effective motivic spectra $\SH(k)^\eff \subset \SH(k)$ \cite[Section 2]{voevodsky-slice-filtration}. 
This subcategory also has a $t$-structure (see~\cite[Proposition~4(4)]{tom2017slices}), with heart $\SH(k)^{\eff\heart} \wequi \HI_0(k)$ consisting of so-called ``effective homotopy modules'' (this is a full subcategory of $\HI_*(k)$). Moreover, using results of \cite{EHKSY} we prove that $\HI_0(k) \wequi \prod^\fr(k)$, where $\prod^\fr(k)$ denotes the category of (strictly) homotopy invariant sheaves with \emph{framed transfers} (see Theorem \ref{thm:identify-effective-heart}). In other words, in order to prove that $M_{-q}$ is a homotopy module, it suffices to exhibit a structure of framed transfers on it. Our second main result provides an infinite delooping for the two-fold contraction of a strictly homotopy invariant sheaf.
\begin{theorem}[Simplified version; see Example \ref{ex:M-1-transfers} and Theorem \ref{thm:M-2-delooping}] \label{thm:intro-2}
Let $char(k)=0$ and $M \in \HI(k)$. Then $M_{-2}$ is a homotopy module.
\end{theorem}
The above theorem is also an immediate consequence of Conjecture \ref{conj:intro}. Indeed given $M \in \HI(k)$ there exists $f_2^\heart M \in \SHS(k)(2)^\heart$ with $M_{-2} \wequi (f_2^\heart M)_{-2}$.
\NB{to check, evaluate both sides on a scheme and apply adjunctions}
By the conjecture, $f_2^\heart M$ is a homotopy module, and hence so is its two-fold contraction.

\subsection*{Effective covers and transfers}
In the final Section \ref{sec:Gm-stab} we re-interpret some parts of the above results in more abstract terms, closing the circle with the motivational section from the beginning of this introduction. The inclusion $i_q\colon \SHS(k)(q) \hookrightarrow \SHS(k)$ preserves colimits and hence has a right adjoint $r_q\colon \SHS(k) \to \SHS(k)(q)$. By the main theorem of \cite{levine2008homotopy}, at least if $k$ is infinite, the functor $E \mapsto i_q r_q E$ is equivalent to the functor $E \mapsto E^{(q)}$. In other words the homotopy coniveau tower is inextricably linked with $\Gm$-stabilization.

The functor $\sigma_q\colon \SHS(k) \to \SHS(k)(q), E \mapsto E \wedge \Gmp{q}$ has a right adjoint $\omega_q\colon \SHS(k)(q) \to \SHS(k), E \mapsto \Omega_\Gm^q i_q(E)$ which is easily seen to be monadic (in the $\infty$-categorical sense). The composite $\omega_q r_q$ is equivalent to $\Omega_\Gm^q$, and consequently, for $M \in \HI(k)$, it must be possible to think of the spectrum $i_q r_q M$ as some extra structure on $\Omega_\Gm^q M \wequi M_{-q}$. The equivalence $i_q r_q M(X) \wequi M^{(q)}(X) \wequi C^*(X, M, q)$ (for $X$ affine with $\omega_X \wequi \scr O_X$) reveals what this structure is. Indeed, $C^*(X, M, q)$ is built out of groups $H^q_Z(Y, M)$ for certain smooth schemes $Y$ and closed subschemes $Z \subset Y$ of codimension $q$; inspection of the Rost-Schmid complex shows that $H^q_Z(Y, M)$ only depends on $M_{-q}$. However, the boundary maps in $C^*(X, M, q)$ are built out of the following types of maps: given a closed immersion $i\colon Y' \to Y$ such that $Z' := i^{-1}(Z) \subset Y'$ still has codimension $\ge q$, there is a pullback $i^*\colon H^q_Z(Y, M) \to H^q_{Z'}(Y', M')$. The results of Section \ref{sec:BLRS} (in particular Remark \ref{rmk:reconstruction}) imply that this is precisely the extra structure needed to recover $M^{(q)}$ from $M_{-q}$.\footnote{Technically speaking, we also need the $\ul{GW}$-module structure on $M_{-q}$.}

On the other hand, we already have a significant amount of extra structure on $M_{-q}$ (at least for $q \ge 2$): it is a homotopy module. Could there really be \emph{further} extra structure? Our instinct would be to guess that this is not the case (see Remark~\ref{rmk: helpless}). Upon further reflection this is equivalent to the following statement, rephrasing Conjecture \ref{conj:intro} that we started with: the functor $\omega^{\infty - q}\colon \HI_0(k) \to \SHS(k)(q)^\heart$ (obtained by factoring $\omega^\infty\colon \SH(k)^\eff \to \SHS(k)$ through $\omega_q\colon \SHS(k)(q) \to \SHS(k)$) is an equivalence, for $q > 0$ (see Conjecture \ref{conj:equivalence}).
Our third main result is some more progress towards establishing this conjecture. 

\begin{theorem}[See Theorem~\ref{thm:automatic-transfer-preservation}] \label{thm:intro-3}
Let $k$ be a perfect field, and $q>0$. Then the functor $\omega^{\infty - q}\colon \HI_0(k) \to \SHS(k)(q)^\heart$ is fully faithful.
\end{theorem}

To prove this result, we first show that the ``forgetful'' functor $i_q^\heart\colon \SHS(k)(q)^\heart \to \SHS(k)^\heart \wequi \HI(k)$ is fully faithful and has as essential image those sheaves $F$ such that $\ul{\pi}_0(f_q F) \to F$ is an isomorphism (Corollary \ref{corr:SHSn-heart}(2)). We thus need to show that sheaves of the latter form carry a unique structure of framed transfers. For this we use the homotopy coniveau tower again. Consider the case $q=1$. It is shown in \cite{levine-slice} that for a field $K$, there is a surjection $\epsilon\colon C^0(K, M, 1) \to \ul{\pi}_0(f_1 F)(K)$. The group $C^0(K, M, 1)$ is built out of contractions of strictly homotopy invariant sheaves, and in particular has a structure of transfers and $\ul{GW}$-module. Surjectivity of $\epsilon$ implies that there is at most one compatible structure on $\ul{\pi}_0(f_1 F)(K)$. We show that if $F$ is a homotopy module then this compatible structure exists and is indeed given by the canonical transfers on $F$ (Proposition \ref{prop:virtual-transfers-correct}). Theorem \ref{thm:intro-3} follows from this and our study of homotopy invariant framed sheaves in Section \ref{sec:generalized-transfers}.

\subsection*{Acknowledgments}
Our gratitude to Marc Levine cannot be overstated. Without his geometric intuition this paper would not have been possible. His guidance is particularly visible in Section \ref{sec:q-good-RS}.
We furthermore thank Jean Fasel for helpful discussions about the Rost-Schmid complex, Wataru Kai for insights on moving lemmas, and Marc Hoyois for comments on a draft of this paper.

\subsection*{Use of $\infty$-categories}
We think of the categories $\Spc$,  $\mathcal{D}(\Ab)$, $\HH(k)$, $\SH(k)$ etc. as $\infty$-categories (see \cite[\S2.2 and \S4.1]{bachmann-norms} for a definition of $\HH(k)$ and $\SH(k)$ as $\infty$-categories); concretely we have in mind the model of quasi-categories as set out in \cite{HTT,HA}.
For most parts of this paper this is irrelevant, and the reader can instead safely think of homotopy categories of the corresponding model categories. 
Certain isolated proofs really do use the higher structures; we will point this out explicitly each time.

\subsection*{Notation and conventions}
We use cohomological notation for complexes (cochain complexes), but we call them ``chain complexes" for brevity.

Whenever we write ``sheaf" we mean a Nisnevich sheaf, unless other is specified. Given a sheaf $M$, a morphism of schemes $f \colon X \to Y$ and $m \in M(Y)$, we sometimes denote $f^*(m)$ by $m|_X$.

If $M \in \Ab(X_\Nis)$ is a sheaf of $\ul{GW}$-modules, and $\scr L$ is a line bundle on $X$, we denote by $M(\scr L) := M \times_{\Gm} \scr L^\times = M \otimes_{\Z[\Gm]} \Z[L^\times]$  the twist of $M$ by $\scr L$. We also denote its sections by $M(X, \scr L)$. We write $X^{(d)} \subset X$ for the set of points of codimension $d$, i.e. those $x \in X$ such that $\dim X_x = d$.

For $X$ an essentially smooth $k$-scheme, we denote by $\omega_X$ the line bundle $\Omega^{max}_{X/k}$, i.e. the highest non-vanishing exterior power of the sheaf of Kähler differentials.  For $f\colon X \to Y$ a morphism of essentially smooth schemes, $\omega_{X/Y} = \omega_f := \omega_X \otimes f^*\omega_Y^{-1}$. Note that $\omega_{f}$ is isomorphic to the determinant of the cotangent complex $L_f$ of $f$.

For a point $x \in X$ with residue field $K$, we view $x$ as a scheme isomorphic to $Spec(K)$, coming with a canonical morphism of schemes $x \to X$.
If $X$ is a smooth scheme over a perfect field, then $x$ is an essentially smooth scheme.

We extend presheaves defined on smooth schemes to essentially smooth schemes by taking colimits; this is well-defined by \cite[Proposition 8.13.5]{EGAIV}.

Given objects $E, F$ in some (higher) category $\scr C$, we denote the set of (homotopy classes of) maps from $E$ to $F$ by $[E, F]$, at least when no confusion can arise.

We follow Morel's convention for the indexing of homotopy sheaves: given $E \in \SH(k)$, we denote by $\ul{\pi}_i(E)_j$ the sheaf associated with $X \mapsto [\Sigma^\infty_+ X[i], E \wedge \Gmp{j}]$.

We write $\HI(k)$ for the category of strictly homotopy invariant sheaves on $\Sm_k$ and $\HI_0(k)$ for the category of effective homotopy modules. For $M \in \HI_0(k)$ we abusively denote by $M$ also its image in $\HI(k)$ under the forgetful functor $\omega^\infty \colon \HI_0(k) \to \HI(k)$, which can be interpreted as forgetting the structure of framed transfers. For a sheaf $M \in \HI(k) \simeq  \SHS(k)^\heart$ we denote the corresponding motivic $S^1$-spectrum also by $M$. 

Given a vector bundle $V$ on a smooth variety $X$, we denote by $Th(V) = V/V \setminus 0$ the associated Thom space object, in various contexts. If $Z \subset X$ is a smooth closed subvariety with normal bundle $N_{Z/X}$, then in all our situations there will be a purity equivalence $X/X \setminus Z \wequi Th(N_{Z/X})$ \cite[Theorem 3.23]{hoyois-equivariant} \cite[Theorem 2.23]{A1-homotopy-theory} which we will use without further comment.

Unless stated otherwise, we will assume that the base field $k$ is perfect.

\section{The Rost-Schmid complex}
\label{sec:RS}

Let $M$ be a strictly homotopy invariant sheaf on $\Sm_k$ and $X \in \Sm_k$. Recall the Rost-Schmid complex $C^*(X, M)$ \cite[Definitions 5.7 and 5.11]{A1-alg-top}: this is a chain complex of abelian groups with \[ C^p(X, M) = \bigoplus_{x \in X^{(p)}} M_{-p}(x, \omega_{x/X}). \] Recall that for $M \in \HI(k)$ its \emph{contraction} $M_{-1}$ is the sheaf $X \mapsto \ker(M(X \times (\A^1 \setminus 0)) \to M(X \times 1))$ \cite[Definition 4.3.10]{morel-trieste}, and $M_{-n}$ for $n \ge 1$ denotes the iteration of this operation. This is canonically a $\ul{GW}$-module (the natural $\Gm$-action factors through a $\ul{GW}$-action by \cite[Lemma 3.49 and paragraph before]{A1-alg-top}), so the twisting makes sense.\footnote{Morel twists by $\Lambda_x^X := \det(T_x X)$. Since for the relative cotangent complex we have $L_{x/X} = C_{x/X}[1]$, where $C_{x/X}$ denotes the conormal sheaf, we find that $\omega_{x/X} = \det(L_{x/X}) = \det(C_{x/X})^\vee \wequi \det(T_x X)$. In other words, our twist is canonically isomorphic to Morel's.} Recall also that we identify the point $x$ with the essentially smooth scheme $Spec(k(x))$; so we write $M(x)$ for what is often denoted $M(k(x))$. The Rost-Schmid complex computes the (Zariski or Nisnevich) cohomology of $M$ on $X$, and is in fact isomorphic to the Gersten complex \cite[Corollary 5.44]{A1-alg-top}. Let us quickly recall the definition of the differentials. For every $x \in X^{(d)}$ and $y \in X^{(d+1)}$ there is a differential $\partial^x_y\colon M_{-d}(x, \omega_{x/X}) \to M_{-d-1}(y, \omega_{y/X})$, and the total differential is the sum of these local differentials. If $y \not\in \bar{x}$ we have $\partial^x_y = 0$. Otherwise $\partial^x_y$ is built out of certain canonical transfer and boundary maps \cite[Definition 5.11]{A1-alg-top}.
 
For an open subset $U \subset X$ with closed complement $Z$, let $C_Z^*(X, M)$ denote the kernel of the evident surjection of complexes $C^*(X, M) \to C^*(U, M)$. Since $C^*(X, M)$ computes $H^*(X, M)$ and similarly for $U$, we find that $C^*_Z(X, M)$ computes $H^*_Z(X, M)$.

We note that if $\scr L$ is a line bundle on $X$, then one may twist all the terms in the Rost-Schmid complex in an evident way, except for $C^0$. If $M$ is a homotopy module then this problem goes away and we obtain a new complex which we denote by $C^*(X, M(\scr L))$.

\subsection{Pullbacks}
\label{sec:RS-pullbacks}

We shall need to make some use of the functoriality of the Rost-Schmid complex.

\subsubsection{Flat pullback}
\label{ssec: flat pullback}
Recall that by~\cite[Corollary~5.44]{A1-alg-top}, there is a canonical isomorphism of chain complexes $C^*(X, M) \simeq E_{1}^{*, 0}(X, M)$, where right-hand side is the Gersten (Cousin) complex. It is defined as the $q=0$ line of the $E_1$-page of the coniveau spectral sequence (see~\cite[Section~1]{bloch-ogus-gabber}): 
\[
E_{1}^{p, q}(X, M) = \bigoplus_{x \in X^{(p)}} H_x^{p+q} (X, M) \Rightarrow H^{p+q} (X, M),
\]
where $H_x^n (X, M) = \colim_W H^n_{\overline{x}  \cap W} (W, M)$, the colimit being over open neighborhoods $W$ of $x$ in $X$. This spectral sequence arises from the fact that when $\dim X = d$, one has \[\oplus_{x \in X^{(p)}} H_x^{p+q} (X, M) \simeq \colim_{\{Z\}} H^{p+q}_{Z_p \setminus Z_{p+1}} (X \setminus Z_{p+1}, M)\] where $\{Z\}$ runs through chains of closed subsets $\emptyset = Z_{d+1} \subset Z_d  \subset \dots \subset Z_0 = X$ such that $\codim(Z_p, X) \ge p$. The $E_1$-page is then built from long exact sequences, associated with chains $\{Z\}$:
\[
\dots \to H^{p+q}_{Z_{p+1}} (X, M) \to H^{p+q}_{Z_p} (X, M) \to H^{p+q}_{Z_p \setminus Z_{p+1}} (X \setminus Z_{p+1}, M) \to H^{p+q+1}_{Z_{p+1}} (X, M) \to \dots
\]

A morphism $f \colon Y \to X$ induces a pullback map 
\[
f^* \colon H^{n}_{Z_p \setminus Z_{p+1}} (X \setminus Z_{p+1}, M) \to 
H^{n}_{f^{-1}(Z_p) \setminus f^{-1}(Z_{p+1})} (Y \setminus f^{-1}(Z_{p+1}), M).
\]
If $f$ is flat, $\codim(f^{-1}(Z_p), Y) \ge \codim(Z_p, X)$  (see \cite[Tags 02NM and 02R8]{stacks}), so for a chain $\{Z\}$ in $X$ its preimage $\{f^{-1}(Z)\}$ gives a chain in $Y$ with the same condition on codimensions. Hence for $x \in X^{(p)}$ we obtain $f^* \colon H_x^n (X, M) \to  \bigoplus_{y \in Y^{(p)}} H_y^n (Y, M)$ which is $0$ when $f^{-1}(x) = \emptyset$ and otherwise the induced by the pullback $H_x^n (X, M) \to H_{f^{-1}(x)}^n (Y, M)$. 

\begin{lemma} \label{lemm:flat-pullback-works}
Let $f \colon Y \to X$ be a flat morphism. 
\begin{enumerate}
\item The induced pullback map $f^* \colon E_{1}^{*, 0}(X, M) \to E_{1}^{*, 0}(Y, M)$ is a map of chain complexes.
\item For composable flat morphisms $f$ and $g$ one has $(f \circ g)^* = g^* \circ f^*$.
\end{enumerate}
\end{lemma}
\begin{proof}
This follows from functoriality of the pullback on cohomology with support.
\end{proof}

Let $x \in X^{(d)}$. Then since $f$ is flat we have $f^*\omega_{x/X} \wequi \omega_{f^{-1}(x)/Y}$ \cite[Tag 08QQ]{stacks}. If $f$ is furthermore smooth then $f^{-1}(x)$ is essentially smooth, and for every $y \in Y^{(d)}$ with $f(y) = x$ we have a canonical isomorphism $\omega_{y/Y} \wequi f^*\omega_{x/X}|_y$. Hence there is a canonical map $C^*(X, M) \to C^*(Y, M)$ sending $M_{-d}(x, \omega_{x/X})$ to $M_{-d}(y, \omega_{y/Y})$ via $f^*\colon M_{-d}(x) \to M_{-d}(y)$ and the previous isomorphism.

\begin{lemma} \label{lemm:compute-flat-pullback}
Let $f \colon Y \to X$ be a smooth morphism. Then the flat pullback $f^*\colon C^*(X, M) \to C^*(Y, M)$ is given by the above explicit map.
\end{lemma}
\begin{proof}
Let $y \in Y^{(d)}$ and $x = f(y)$. The component $\alpha\colon M_{-d}(x, \omega_{x/X}) \to M_{-d}(y, \omega_{y/Y})$ of the map $f^*\colon C^*(X, M) \to C^*(Y, M)$ is obtained as follows. Let $V$ be an open neighbourhood of $y \in Y$ and $f(V) \subset U$ an open neighbourhood of $x \in X$. Shrinking $U$ and $V$ sufficiently, we may assume that $C := \bar{x} \cap U$ and $D := \bar{y} \cap V$ are smooth and $D = f^{-1}(C) \cap V$ (for this latter condition we use that $f$ is smooth). Then there is an induced map $\alpha_{U,V}\colon Th(N_{D/V}) \wequi V/V \setminus D \to U/U \setminus C \wequi Th(N_{C/U})$. Applying $M$ and taking the colimit over $U, V$, we obtain the map $\alpha$. Note that $f\colon (V, D) \to (U, C)$ is a morphism of smooth closed pairs (see \cite[Section 3.5]{hoyois-equivariant}). It follows from functoriality of the purity equivalence in morphisms of smooth closed pairs that $\alpha_{U,V}$ is the thomification of the canonical map $\beta_{U,V}\colon N_{D/V} \to N_{C/U}$. The colimit of $M(Th(\beta_{U,V}))$ is the explicit pullback constructed above. This concludes the proof.
\end{proof}

\subsubsection{Action by $\ul{K}_1^{MW}$}
From now on we assume that $M \in \HI_0(k)$, i.e. $M$ is an (effective) homotopy module \cite[Section 5.2]{morel-trieste}. In particular we assume given strictly homotopy invariant sheaves $M_{+1}$, $M_{+2}$ and so on together with isomorphisms $(M_{+(n+1)})_{-1} \wequi M_{+n}$ (and $M_{+0} = M$).

Suppose given $A \in \ul{K}_1^{MW}(X)$, e.g. $A = [a]$ for some $a \in \scr O(X)^\times$. We define a morphism of graded abelian groups \[A \times\colon C^{*}(X, M) \to C^{*}(X, M_{+1}).\] The morphism $A\times$ in components is just given by $M_{-d}(x, \omega_{x/X}) \to M_{-d+1}(x, \omega_{x/X}), m \mapsto A \times m$ coming from the action of $\ul{K}_1^{MW}(X) \to \ul{K}_1^{MW}(x)$ (see~\cite[Lemma 3.48]{A1-alg-top} for details). 

\begin{lemma} \label{lemm:Utimes-commutativity}
Let $M \in \HI_0(k)$. The morphism of graded abelian groups $U\times$ commutes with the boundary up to multiplication by $\epsilon := -\lra{-1}$.
\end{lemma}
\begin{proof}
This is stated in \cite[just before Lemma 5.36]{A1-alg-top}.
\end{proof}

\subsubsection{Boundary morphism}
Let $Z \subset X$ be closed, smooth and everywhere of codimension $1$, and $U \subset X$ the open complement. There is a morphism of graded abelian groups \[\partial = \partial^U_Z\colon C^*(U, M) \to C^*(Z, M_{-1}(\omega_{Z/X})),\] defined as follows. Given $x \in U^{(d)}$ and $y \in \bar{x}^{(1)} \cap Z$ (so $y \in Z^{(d)}$), we have the Rost-Schmid boundary map $\partial^x_y\colon M_{-d}(x, \omega_{x/X}) \to M_{-d-1}(y, \omega_{y/X})$. We compose this with the isomorphism ${\omega_{y/X} \wequi \omega_{y/Z} \otimes \omega_{Z/X}}$, which is given by $(-1)^d$ times the canonical isomorphism. Altogether we obtain a morphism $M_{-d}(x, \omega_{x/X}) \to M_{-d-1}(y, \omega_{y/Z} \otimes \omega_{Z/X})$, as needed.

\begin{remark}\NB{I'm still not entirely happy about this. The map is basically thom iso following the boundary in cohomology with support. There is always a $\pm 1$ sign in the latter map. But I cannot see a natural reason why the thom iso should be given a sign $\lra{\pm 1}$.}
The peculiar choice of isomorphism above is somewhat justified by Lemma \ref{lemm:closed-pullback} below. Also the same factor occurs in the definition of $r$ in \cite[Proof of Theorem 5.38]{A1-alg-top}, and in \cite[p. 10]{asok2016comparing}.
\end{remark}

\begin{lemma} \label{lemm:boundary-commutativity}
Let $M \in \HI(k)$. The morphism of graded abelian groups $\partial^U_Z\colon C^*(U, M) \to C^*(Z, M_{-1}(\omega_{Z/X}))$ commutes with the boundary up to multiplication by $\epsilon$.
\end{lemma}
\begin{proof}
Let $x \in U$ and $z \in \bar{x}^{(2)} \cap Z$. Since $C^*(X, M)$ is a chain complex, we have
\[ 0 = (\partial)^2|^x_z = \sum_{y \in \bar{x}^{(1)} \cap U} \partial^y_z \partial^x_y + \sum_{y \in \bar{x}^{(1)} \cap Z} \partial^y_z \partial^x_y. \]
The first term is $\partial^U_Z \circ \partial$, and the second term is $\lra{-1} \partial \circ \partial^U_Z$, because of our choice of isomorphism. The result follows.
\NB{here the first sum is without extra sign because codimension of closed points in the sum got shifted by 1}
\end{proof}

\subsubsection{Closed pullback}
\label{subsec:closed-pullback}
Now suppose we are given $t \in \mathcal{O}(X)$ is such that $Z := Z(t)$ is smooth.  Write $i\colon Z \hookrightarrow X$ for the closed immersion.  Then for any $M \in \HI_0(k)$ there is a pullback morphism $i^*\colon C^{*}(X, M) \to C^{*}(Z, M)$, defined as the composite
\[ C^*(X, M) \to C^*(X \setminus Z, M) \xrightarrow{[t] \times } C^*(X \setminus Z, M_{+1}) \xrightarrow{\partial^{X\setminus Z}_Z} C^*(Z, M), \]
where in the last step we have trivialized $\omega_{Z/X}$ via $t$.  See also \cite[Section 3]{rost1996chow} \cite[Section 5.3]{A1-alg-top}. 

\begin{lemma} \label{lemm:closed-pullback}
Let $M \in \HI_0(k)$.
\begin{enumerate}
\item For $i\colon Z \to X$ as above, $i^*\colon C^{*}(X, M) \to C^{*}(Z, M)$ is a morphism of chain complexes.
\item $i^*$ is compatible with open immersions.
\item Let $f\colon X \to Z$ be smooth of relative dimension $1$, and suppose that also $f \circ i = \id_Z$. Then $i^*f^* = \id$.
\end{enumerate}
\end{lemma}
\begin{proof}
(1) follows from Lemmas \ref{lemm:Utimes-commutativity} and \ref{lemm:boundary-commutativity} (together with $\epsilon^2 = 1$). (2) is clear by construction. For (3), apply \cite[Lemma 5.36]{A1-alg-top} together with Lemma~\ref{lemm:flat-pullback-works}(2).
\end{proof}

\begin{remark}\label{rmk:closed-pullback-correct}
It follows from Lemma \ref{lemm:closed-pullback} that for $M \in \HI_0(k)$ we have a commutative diagram of chain complexes of sheaves on the small Nisnevich site of $X$
\begin{equation*}
\begin{CD}
M @>>> \ul{C}^*(X, M) \\
@VVV     @Vi^*VV      \\
i_* M @>>> i_* \ul{C}^*(Z, M).
\end{CD}
\end{equation*}
This implies that the map $i^*\colon H^*(X, M) \wequi h^*(C^*(X,M)) \to h^*(C^*(Z,M)) \wequi H^*(Z,M)$ is the canonical pullback map on cohomology, and similarly for cohomology with support.
\end{remark}

\begin{remark}
Let $M_* \in \HI_*(k)$.
Since $C^*(X, \omega^\infty M_*)$ only depends on $M_0$, Lemma \ref{lemm:closed-pullback} holds for $M_0$ in place of $M$ as well.
In fact there exists $M' \in \HI_0(k)$, an effective cover of $M_*$, with a map $M' \to M_* \in \HI_*(k)$ inducing an isomorphism $M'_0 \wequi M_0$.
Similar observations apply for some other results below, where we forgo apparent extra generality by working with effective homotopy modules only.
\end{remark}

\subsection{Transfers}
\label{sec:RS-transfers}
Let $f\colon X \to Y$ be a finite flat morphism of essentially smooth $k$-schemes\NB{i.e. $f$ is finite and for every component $Y_0$ of $Y$ and every component $X_0$ of $f^{-1}(Y_0)$ we have $\dim Y_0 = \dim X_0$ (miracle flatness).}. Then for any $M \in \HI_0(k)$ there is a transfer map \cite[Corollary 5.30]{A1-alg-top}
\[ \tr_f\colon C^{*}(X, M(\omega_{X/Y})) \to C^{*}(Y, M). \]
For the convenience of the reader, we recall its definition. Given $x \in X^{(d)}$ let $y = f(x)$. Then $y \in Y^{(d)}$ since $f$ is finite and flat. The component of $\tr_f$ from $M_{-d}(x, \omega_{x/X} \otimes \omega_{X/Y})$ to $M_{-d}(y, \omega_{y/Y})$ is given by the absolute transfer $\tr_{x/y}\colon M_{-d}(x, \omega_{x/y}) \to M_{-d}(y)$ which exists on any homotopy module (see \cite[Section 5.1]{A1-alg-top} and Example~\ref{ex:twisted-transfers-M-2}), twisted by $\omega_{y/Y}$.

We make use of the following result in transfer arguments. It is inspired by (and reduces to) \cite[Proposition B.1.4]{EHKSY}. Here and elsewhere, given a morphism $f: Y \to X \in \Sm_k$ and a trivialization $\omega_f \stackrel{\gamma}{\wequi} \scr O$, we denote the by $\tr_f^\gamma: C^*(Y, M) \to C^*(X, M)$ the morphism obtained from the absolute transfer $\tr_f: C^*(Y, M(\omega_f)) \to C^*(X, M)$ via the isomorphism $C^*(Y, M(\omega_f)) \wequi C^*(Y, M)$ induced by $\gamma$.

\begin{lemma} \label{lemm:pushpull-d}
Let $Y$ be (essentially) smooth and $X \hookrightarrow \A^1_Y$ be cut out by a monic global section $P$ of degree $n$. Write $f\colon X \to Y$ for the projection, and assume that $f$ is finite étale\NB{probably finite, hence flat, suffices}. Let $\gamma\colon \omega_{X/Y} \wequi \omega_{X/\A^1_Y} \otimes \omega_{\A^1_Y/Y}|_X \wequi \scr O_X$ be the canonical isomorphism induced by $P$ and the coordinate $t$. Then for any $M \in \HI_0(k)$ and a line bundle $\scr L$ on $Y$, the composite
\[ C^{*}(Y, M(\scr L)) \xrightarrow{f^*} C^{*}(X, M(f^* \scr L)) \xrightarrow{\tr_f^\gamma} C^{*}(Y, M(\scr L)) \]
is given by multiplication by $n_\epsilon := \sum_{i=1}^n \lra{(-1)^{i-1}}$.
\end{lemma}
\begin{proof}
Let $y \in Y^{(d)}$. Since $X \to Y$ is smooth, so is $f^{-1}(y) \to y$. In particular, $f^{-1}(y)$ is a finite union of points on $X$ of the same codimension. We have $f^* \omega_{y/Y} \wequi \omega_{f^{-1}(y)/X}$ ($\ast$), since $f$ is flat. According to Lemma \ref{lemm:compute-flat-pullback}, the pullback $f^*\colon C^{*}(Y, M(\scr L)) \to C^{*}(X, M(f^* \scr L))$ is just given on the component corresponding to $y$ by the map \[ M_{-d}(y, \omega_{y/Y} \otimes \scr L) \to M_{-d}(f^{-1}(y), \omega_{f^{-1}(y)/X} \otimes f^* \scr L) \wequi M_{-d}(f^{-1}(y), f^*(\omega_{y/Y} \otimes \scr L)). \] We have $\omega_{f^{-1}(y)/y} \wequi \omega_{X/Y}|_{f^{-1}(y)}$ ($\ast\ast$), again by flatness of $f$. The transfer map \[\tr_f\colon C^{*}(X, M(f^*(\scr L) \otimes \omega_{X/Y})) \to C^{*}(Y, M(\scr L))\] on the components corresponding to $f^{-1}(y)$ can be thus rewritten as 
\[M_{-d}(f^{-1}(y), \omega_{f^{-1}(y)/X} \otimes f^*(\scr L) \otimes \omega_{X/Y}|_{f^{-1}(y)}) \to M_{-d}(y, \omega_{y/Y} \otimes \scr L),\] using ($\ast$) and ($\ast\ast$). In other words, it suffices to prove the result in the case where $Y$ is the spectrum of a field extension of $k$ and $*=0$. In particular we may assume that $\scr L \wequi \scr O_Y$.

Observe that the pullback and transfer actually come from maps defined over $Y$.
We may thus base change everything to $Y$ and assume that $k=Y$.\footnote{$Y$ need not be the spectrum of a \emph{perfect} field, but this does not affect the argument below.}

We can view $M \in \SHS(k)$. The composite $\tr_f^\gamma \circ f^*\colon M_{-1}(k) \to M_{-1}(k)$ is given by pullback along a map of spectra $\beta \colon \Gm \to \Gm$. We will show that $\beta = n_\epsilon$. This implies what we want since $M$ is a homotopy module, so itself of the form $M'_{-1}$. We know that $[\Gm, \Gm]_{\SHS(k)} \simeq \ul{GW}(k) \simeq [\Gm, \Gm]_{\SH(k)}$ (this is essentially \cite[Theorems 6.39 and 3.37]{A1-alg-top}; see also \cite[Theorem 10.12]{bachmann-norms} and its proof), so it suffices to prove this for the image of $\beta$ in $\SH(k)$. The explicit construction of $\tr_f^\gamma$ using the homotopy purity theorem (reviewed for example in the beginning of Section \ref{subsec:projection-formulas}) shows that $\beta$ corresponds to the map
 \[ \P^1 \to \P^1/(\P^1 \setminus X) \wequi \A^1/(\A^1 \setminus X) \xrightarrow{P} \A^1/(\A^1 \setminus 0) \wequi T.\]
This is precisely the action of the tangentially framed correspondence $* \xleftarrow{f} X \xrightarrow{f} *$ (with the cotangent complex $L_f$ trivialized by $P$) on the sheaf $\ul{GW}$ (see~\cite[Section B.1.1]{EHKSY} for definitions, and Example \ref{ex: Voev transfers} for a closely related statement). Hence the result follows from \cite[Proposition B.1.4]{EHKSY} and~\cite[Example~3.1.6]{EHKSY2}. 
\end{proof}

\begin{lemma} \label{lemm:pushpull-ideal}
Let $r, s \in \Z$ be coprime, $k$ a field. Then $(r_\epsilon, s_\epsilon)$ generate the unit ideal of $GW(k)$.
\end{lemma}
\begin{proof}
We may assume that $r$ is odd, say $r = 2n+1$. Then $1 = r_\epsilon - nh$, where $h := 1 + \lra{-1}$. Since $r, s$ are coprime, there are $a, b \in \Z$ with $ar + bs = n$. Consequently (noting that $\lra{-1}h = h$) we have $nh = (a r_\epsilon + b s_\epsilon) h$ and so $1 = (1 - ah) r_\epsilon - bh s_\epsilon$. The result follows.
\end{proof}

\subsection{Functoriality in $M$}
\label{sec:RS-functoriality-M}
We remark that if $M \to N$ is any morphism of strictly homotopy invariant sheaves (respectively (effective) homotopy modules), then there is an induced morphism $C^*(X, M) \to C^*(X, N)$, which is compatible with the flat pullback and boundary morphism (respectively all constructions in the previous subsections).

\section{The $q$-good Rost-Schmid complex}
\label{sec:q-good-RS}

Let $M \in \HI(k)$.
\begin{definition} \label{defn:q-good-cx}
Let $X \in \Sm_k$, $q \ge 0$. Let $\scr C$ be a finite family of irreducible,  closed subsets of $X$.  For notational convenience, we always assume that $X \in \scr
C$.
\begin{enumerate}
\item A closed subset $Z \subset X$ is said to be \emph{in $q$-good position} with respect to $\scr C$ (we abbreviate this to ``$q$-good'') if for every $W \in \scr C$, the codimension of $Z \cap W$ inside $W$ is at least $q$.
\item We put
  \[C^d_{\scr C, q}(X, M) =
        \bigoplus_{\substack{x \in X^{(d)} \\ \bar{x} \qgood}}
             M_{-d}(x, \omega_{x/X}) \subset C^d(X, M). \]
  This is the eponymous ``$q$-good Rost-Schmid complex''.
\end{enumerate}
\end{definition}
Note that since $X \in \scr C$ we have $C^i_{\scr C, q}(X, M) = 0$ for $ i < q$. Note also that
\[ C^d_{\scr C, q}(X, M) = \colim_{Z \qgood} C^d_Z(X, M). \]
In particular $C^*_{\scr C,q}(X, M)$ indeed \emph{is} a subcomplex of $C^*(X, M)$, and it computes $\colim_Z H_Z^*(X, M)$, where the colimit is taken over $q$-good closed subsets.

The aim of this section is to show that the inclusion \[\alpha\colon C^*_{\scr C, q}(X, M) \to C^*(X, M)\] induces an isomorphism on the cohomology groups $h^i$, for $i>q$, at least in favorable cases. Note that if $q = 0$ then $\alpha$ is an isomorphism, so we shall usually assume that $q>0$. For readers familiar with \cite{levine2006chow}, let us point out that the results we are going to prove are analogs of the results in that paper, just for the Rost-Schmid complex instead of the homotopy coniveau tower. Naturally, we will employ similar strategies of proof: using a weak homotopy invariance property and the action by general translations we can deal with the case $X = \A^n$, and then employing general projections we can extend to all smooth affine $X$, provided that $\omega_X \wequi \scr O_X$. This latter condition is a new twist,  which occurs because we need to pushforward in an unoriented situation.

We shall assume throughout this section that $M$ is an effective  homotopy module (i.e. an object in the image of the forgetful functor $\omega^\infty \colon \HI_0(k) \to \HI(k)$). Some of our results hold without this assumption, but this would unnecessarily clutter our notations.

\subsection{Weak homotopy invariance}
In this subsection we establish an analog of the weak homotopy invariance result \cite[Lemma 3.2.3]{levine2006chow}. Throughout we fix $X \in \Sm_k$.

\begin{definition} \label{def:CdC}
We write $p\colon \A^1 \times X \to X$ for the canonical
projection and $i_0\colon X \to \A^1 \times X$ for the inclusion corresponding to $0 \in \A^1$.
Given a closed subset $Z \subset \A^1 \times X$, denote by $\bar{Z} \subset \P^1 \times X$ its projective closure, and write $i_\infty^{-1}(Z) = \bar{Z} \cap \{\infty\} \times X$\NB{would it make more sense to just write $i_\infty^{-1}(\bar Z)$?}, which we also identify with a closed subset of $X$.
We put\NB{It might be more natural to also require $\bar{x}$ to be $q$-good. This changes nothing in the arguments.}
\[ C^d_{\scr C, q}(\A^1 \times X, M)_h =
      \bigoplus_{\substack{x \in (\A^1 \times X)^{(d)} \\
                      i_0^{-1}(\bar{x}), i_\infty^{-1}(\bar{x}), \overline{p(x)} \qgood}}
      M_{-d}(x, \omega_{x/\A^1 \times X})
      \subset C^d(\A^1 \times X, M).
\]
\end{definition}
As before, it is immediate that $C^*_{\scr C, q}(\A^1 \times X, M)_h$ is a subcomplex of $C^*_{\scr C, q}(\A^1 \times X, M)$.

For $Z \subset X$, we have $p(p^{-1}(Z)) = i_0^{-1}(p^{-1}(Z)) = i_\infty^{-1}(p^{-1}(Z)) = Z$. It follows that \[p^*(C^*_{\scr C,q}(X, M)) \subset C^*_{\scr C, q}(\A^1 \times X, M)_h.\] Note that by construction \[i_0^*(C^*_{\scr C, q}(\A^1 \times X, M)_h) \subset C^*_{\scr C, q}(X, M) \subset C^*(X, M).\]

Consider the closed subscheme $i_\infty \colon \{ \infty \} \times X \hookrightarrow (\P^1 \setminus 0) \times X$, cut out by the section  $-1/T$ on $\P^1 \setminus 0$. According to Section~\ref{sec:RS-pullbacks}, there is an induced pullback map:
\[ i_\infty^*\colon C^{*}(\A^1 \times X, M) \to  C^{*}((\P^1 \setminus \{0, \infty\}) \times X, M) \to C^{*}( \{ \infty \} \times X, M) \simeq C^{*}(X, M).\]
  In more detail, one has 
$i_\infty^* = \partial^{\Gm \times X}_{\{\infty\} \times X} \circ ([-1/T] \times) \circ (\Gm \times X \hookrightarrow \A^1 \times X)^*$, where $\omega_{\{\infty\} \times X / (\P^1 \setminus \{0\}) \times X}$ is trivialized via $-1/T$.
As for $i_0^*$, we have by construction \[i_\infty^*(C^*_{\scr C, q}(\A^1 \times X, M)_h) \subset C^*_{\scr C, q}(X, M) \subset C^*(X, M).\]

\begin{proposition} \label{prop:htpy-inv}
Let $M \in \HI_0(k)$. The map $p^*\colon C^*_{\scr C, q}(X, M) \to C^*_{\scr C, q}(\A^1 \times X,M)_h$ is a chain homotopy equivalence with inverse $i_\infty^*$.\NB{I was slightly suspicious about the $*=q$ case. But note that if $Z \subset \A^1 \times X$ is such that $\overline{p(Z)}$ is of codimension $q$, then $Z \subset \A^1 \times \overline{p(Z)} =: Z'$ which is also of codimension $q$. In particular if $Z$ has codimension $q$ then $Z = Z'$. Thus the indexing sets for the sum in $C^q$ are canonically bijective. (The groups themselves are not, but this is fixed by taking the appropriate kernel.)}
\end{proposition}
\begin{proof}
By Lemma \ref{lemm:closed-pullback}(3) applied to $q: (\P^1_X \setminus 0) \to X$ we have $i_\infty^* p^* = i_\infty^* q^* = \id$. It hence suffices to exhibit a homotopy $H\colon C^*_{\scr C, q}(\A^1 \times X,M)_h \to C^{*-1}_{\scr C, q}(\A^1 \times X,M)_h$ between $p^* i_\infty^*$ and $\id$.

Morel proves \cite[Theorem 5.38]{A1-alg-top}\NB{I'm not totally sure that I believe this proof without assuming that $M$ is a homotopy module...}{} that $p^*\colon C^*(X, M) \to C^*(\A^1 \times X, M)$ is a quasi-isomorphism by exhibiting an explicit homotopy $H$ between $p^* i_\infty^*$ and $\id$, defined on $C^{* \ge 2}$. Under our assumption that $M$ is a homotopy module, $H$ extends to all of $C^*$, and remains a homotopy between $p^*i_\infty^*$ and $\id$. In order to conclude, we need to prove that $H(C^{*}_{\scr C, q}(\A^1 \times X, M)_h) \subset C^{*-1}_{\scr C,q}(\A^1 \times X, M)_h$.

Thus let $y \in (\A^1 \times X)^{(d)}$ and $p(y) = y'$. Then $y'$ has codimension $d$ or $d-1$. In the first case, $H$ is defined to be zero on the summand corresponding to $y$, so we are done. In the second case $H$ is defined to take the summand corresponding to $y$ to the summand corresponding to $\A^1 \times y'$, which satisfies the required conditions as soon as $y$ contributes to $C^*_{\scr C, q}(\A^1 \times X, M)_h$ (since then $\overline{y'}$ is $q$-good, by assumption). This concludes the proof.
\end{proof}

\begin{corollary} \label{cor:htpy-inv}
The two maps $i_0^*$, $i_\infty^*\colon C^*_{\scr C, q}(\A^1 \times X, M)_h \to C^*_{\scr C, q}(X, M)$ induce the same maps on cohomology groups.
\end{corollary}
\begin{proof}
We have $i_0^* p^* = \id$ by Lemma \ref{lemm:closed-pullback}(3). Since $p^*$ induces an isomorphism on cohomology groups by Proposition \ref{prop:htpy-inv}, $i_0^*$ is inverse to $p^*$ on cohomology, and in particular $i_0^* = i_\infty^*$, inverses being unique.
\end{proof}

\subsection{Easy moving}
Our aim in this subsection is to establish the following result.
\begin{theorem}\label{thm:moving}
Let $k$ be a perfect field, $M \in \HI_0(k)$, $n \ge 0, q \ge 0$, $L/k$ a field and $\scr C$ a finite family of closed, irreducible subsets of $\A^n_L$, with $\A^n_L \in \scr C$.

Then the canonical map $\alpha\colon C^*_{\scr C, q}(\A^n_L, M) \to C^*(\A^n_L, M)$ induces a surjection on $h^i$ for $i \ge q$, and an isomorphism for $i \ge q+1$.
\end{theorem}
If $n=0$ or $q=0$, $\alpha$ is an isomorphism, so we do not need to treat these cases. Of course $H^i(\A^n_L, M) = 0$ for $i > 0$, so the theorem just says that $h^i(C^*_{\scr C, q}(\A^n_L, M)) = 0$ for $i > q$.\NB{The story of the proof goes like this: Proposition \ref{prop:htpy-inv} shows that for a good cycle on $\A^1 \times X$, the fiber over zero is cohomologous to the ``infinitely far fiber''. Given a cycle of codimension $>0$ on $\A^n$, we can find a translation which ``as time goes to $\infty$ moves the cycle to $\infty$''. Since varieties are typically unbounded there is something to prove here, but essentially the only way this can go wrong is if we stupidly translate in an ``eventual tangent'' direction of our cycle. Moreover if the codimension is $>q$ and the cycle is $q$-good, we can ensure that its pullback to $\A^1 \times \A^n$ is also good: this is Proposition \ref{prop:geom}. Combining the two statements we find that our cycle is cohomologous to the empty cycle.}

We now recall the \emph{translation operations} \cite[Section 3.3]{levine2006chow}. That is, given a $k$-scheme $X$ and $v \in \A^n(k)$, we define the map \[h_v \colon \A^1 \times X \times \A^n \to X \times \A^n, (t, x, a) \mapsto (x, tv+a).\] Let $p \colon \A^1 \times X \times \A^n \to X \times \A^n$ be the projection. By construction, $h_v \circ i_0 = \id$, and $h_v \circ i_1 =: s_v$ is the translation along $v$. Note that $h_v, s_v$ are smooth: $s_v$ is an isomorphism, and $h_v$ is the composite of the isomorphism $\A^1 \times X \times \A^n \to \A^1 \times X \times \A^n, (t, x, a) \mapsto (t, x, tv+a)$ and the smooth projection $\A^1 \times X \times \A^n \to X \times \A^n$. The following is the main geometric input of our argument. It is an adaptation of \cite[Lemma 3.3.2]{levine2006chow}. Recall that if $V$ is an irreducible variety, we say that a certain property holds for a \emph{general point} of $V$ if there exists a non-empty open subset $U \subset V$ such that the property holds for all points of $U$. \NB{but $U$ may not have rational points}
\begin{proposition} \label{prop:geom}
Let $k$ be a field and $\scr C$ be a finite family of irreducible, closed subsets of $\A^n_k$. Let $X$ be a $k$-scheme of finite type. For $Z \subset X \times \A^n$ closed, we say that $Z$ is $q$-good if it is $q$-good with respect to $\{X \times W \mid W \in \scr C\}$.
\begin{enumerate}
\item Let $Z \subset X \times \A^n$ have codimension $\ge q$. Then for general $v$, $s_v^{-1}(Z) \subset X \times \A^n$ is $q$-good.
\item Let $Z \subset X \times \A^n$ have codimension $\ge q + 1$ and be $q$-good. Then for general $v$, the set $p(h_v^{-1}(Z))$ is $q$-good.
\item Suppose that $X = \spk$. Let $Z \subset \A^n$ be a proper closed subset. Then for general $v$, the set $p(h_v^{-1}(Z))$ is closed, and $i_\infty^{-1}(h_v^{-1}(Z))$ is empty (so in particular $q$-good).
\end{enumerate}
\end{proposition}
\begin{proof}
All the statements we wish to prove hold for $Z$ if and only if they hold for all irreducible components of $Z$. We may thus assume that $Z$ is irreducible.

Fix $W \in \scr C$, and consider the following closed subset:
$$\Sigma = \{(v, (z_1, z_2), w) \, | \, w + v = z_2\} \subset \A^n \times Z \times W,$$
where $(z_1, z_2) \in X \times \A^n$ are the corresponding coordinates of $z \in Z$.
Under the projection $\pi \colon \Sigma \to \A^n$ the fiber over $v_0 \in \A^n$ is given by $\Sigma_{v_0} := \pi^{-1}(\{v_0\}) \simeq s^{-1}_{v_0}(Z) \cap (X \times W)$\NB{$s^{-1}_v$ is just $s_{-v}$}. Let
\[ \Sigma^j = \{s \in \Sigma \mid \dim_s \Sigma_{\pi(s)} \ge \dim{W} + \dim{X} - c_Z + j\} \subset \Sigma. \]
Here $c_Z$ is the codimension of $Z$ in $\A^n_X$ (which is well-behaved since $Z$ is irreducible). In other words $\cup_{j>0} \Sigma^j$ consists of the points where the fiber has bigger dimension than expected, i.e. this is the ``bad locus'' where $s^{-1}_{v_0}(Z)$ meets $X \times W$ in unexpectedly low codimension.

By upper semi-continuity of fiber dimension \cite[02FZ]{stacks}, $\Sigma^j$ is closed in $\Sigma$. We put $B^j = \pi(\Sigma^j)$. By Chevalley's theorem \cite[Tag 054K]{stacks}, $B^j$ is constructible and in particular contains the generic points of its closure $\overline{B^j}$. Applying Lemma \ref{lemm:fiber-dimenision-bound} 
\NB{apply the lemma to $\pi \colon \Sigma \to \A^n$ in a generic point $\eta$}
below to these generic points, and using that $\Sigma \wequi Z \times W$, we find that
\[ \dim{\overline{B^j}} \le \dim{\Sigma} - \dim{\Sigma_\eta} \le \dim{Z} + \dim{W} - (\dim{W} + \dim{X} - c_Z + j) = n - j. \]
In particular $\overline{B^1}$ has positive codimension in $\A^n$, and (1) immediately follows.

For (2), since $\overline{B^2}$ has codimension $\ge 2$ and $0 \not\in B^2$, we find that a general line $L$ through the origin in $\A^n$ does not meet $B^2$ and is not contained in $\overline{B^1}$\NB{The map $B^2 \setminus 0 \to \P^{n-1}$ cannot be dominant for dimension reasons; hence the first statement. The set of lines contained in $\overline{B^1}$ is closed in $\P^{n-1}$, and not equal to $\P^{n-1}$, hence the second statement.}. Let $0 \ne v \in L$. Write $\pi'\colon \Sigma \to X \times W$ for the projection $(v, (z_1, z_2), w) \mapsto (z_1, w)$. Then $p(h_v^{-1}(Z)) \cap X \times W = \pi'(\Sigma_L)$, where $\Sigma_L := \pi^{-1}(L)$. Since $L$ avoids $B^2$ and meets $B^1$ in finitely many points, all but finitely many fibers of $\pi \colon \Sigma_L \to L$ have codimension $\ge q+1$ in $X \times W$ (after applying $\pi'$), and the others have codimension $q$. It follows that $p(h_v^{-1}(Z)) \cap X \times W$ has codimension $\ge q$ in $X \times W$, as needed.

It remains to establish (3), so now $X = \spk$.  We need to show that for $Z \subset \A^n$ of $\codim(Z, \A^n) \ge 1$ and a general $v$ the subset $p(h_v^{-1}(Z)) \subset \A^n$ is also closed, and that $i_\infty^{-1}(h_v^{-1}(Z)) = \emptyset$. We show both at the same time. Consider the embedding $\A^n \subset \P^n$, and let $\P^{n-1}_{\infty}$ be the hyperplane at infinity. For a vector $0 \ne v \in \A^n$ let $\{\overline{v}\} = (\A^1 \cdot v) \cap \P^{n-1}_{\infty}$. The morphism $h_v$ can be extended to a morphism $\overline{h_v} \colon \P^1  \times \A^n \to  \P^n$ by sending $(\infty, a) \mapsto \overline{v}$. Let $\overline{Z} \subset \P^n$ be the closure of $Z$. Since $Z$ was of codimension $\ge 1$ in $ \A^n$,  for a general vector $v$ one has $\overline{v} \notin \overline{Z}$\NB{If $0 \ne p \in k[x_1, \dots, x_n]$, then by definition the homogenization $P(X_0, \dots, X_n)$ is not divisible by $X_0$, and hence $q := P(0, X_1, \dots, X_n)$ is not the zero polynomial. Thus if $0 \ne p \in I(Z)$ then $0 \ne q \in I(\bar{Z} \cap \P^{n-1}_\infty)$.}. Hence $\overline{h_v}^{-1}(\overline{Z}) = h_v^{-1}(Z)$, and it is a closed subscheme of  $\P^1 \times \A^n$. In particular $i_\infty^{-1}(h_v^{-1}(Z)) = \emptyset$. The projection $\overline{p} \colon \P^1 \times \A^n \to \A^n$ is proper, hence $\overline{p}(h_v^{-1}(Z)) $ is closed in $\A^n$. It remains to notice that $\overline{p}(h_v^{-1}(Z)) = p(h_v^{-1}(Z))$.
\end{proof}

In the above proof we have made use of the following straightforward geometric fact, for which we unfortunately could not locate a reference.
\begin{lemma}\label{lemm:fiber-dimenision-bound} \todo{surely this must be citable??}
Let $f\colon X \to Y$ be a morphism of finite type $k$-schemes and $y \in Y$ with $X_y \ne \emptyset$. Then $\dim{X} \ge \dim{X_y} + \dim\bar{y}$.
\end{lemma}
\begin{proof}
Replacing $X$ by $f^{-1}(\bar{y})$ and $Y$ by $\bar{y}$, we may assume that $f$ is dominant and $Y$ is irreducible with generic point $y$. \NB{At this point \cite{eisenbud2013commutative}, Corollary 14.6 and its geometric version explained thereafter conclude} There exists a non-empty open subset $U \subset Y$ with $f(X) \supset U$\NB{E.g. because by Chevalley's theorem $f(X)$ is constructible.}. Since $\dim{U} = \dim{Y}$, we may replace $Y$ by $U$ and $X$ by $f^{-1}(U)$; now $f$ is surjective. Let $x \in X_y$ with $\dim_x X_y = \dim X_y =: e$. Let $x'$ be a closed specialization of $x$. Put $y' = f(x')$; this is a closed point \cite[Tags 00G1, 00GB]{stacks}. By upper semi-continuity of fiber dimension \cite[02FZ]{stacks}, we have $\dim_{x'} X_{y'} \ge \dim_x X_y = e$. By \cite[Tag 00OS]{stacks}, there is a sequence of proper closed subsets $Y = Y_0 \supset Y_1 \supset \dots \supset Y_d = \{y'\}$, with $d= \dim{Y}$. Let $X_i = f^{-1}(Y_i)$. Then $X_i \ne X_{i+1}$, since $f$ is surjective. Let $X_{y'} = X_d \supset X_{d+1} \supset \dots \supset X_{d+e} \ne \emptyset$ be a chain of proper closed subsets in $X_{y'}$; this exists by construction. We have shown that $\dim{X} \ge d+e$, as needed.
\end{proof}

\begin{proof}[Proof of Theorem \ref{thm:moving}.]
We put $h^i := h^i(C^*_{\scr C, q}(\A^n_L, M))$. By strict homotopy invariance of $M$ it suffices to show that $h^i = 0$ for $i > q$.

We first prove this assuming that $L$ is \emph{infinite}. Thus let $a \in h^i$ with $i > q$. We can find a $q$-good subset $Z$ of codimension $i>q$ supporting $a$ and a cycle $c$ supported on $Z$ representing $a$. Choose a vector $v \in \A^n_L(L)$ satisfying the conclusions of Proposition \ref{prop:geom}(2,3) for $Z$ (which is possible since any non-empty open subset of $\A^n$ over an infinite field has a rational point\footnote{We could not find a reference for this well-known statement. A proof is available on MathOverflow at \url{https://mathoverflow.net/a/264212/5181}.}). Then $h_v^*(c) \in C^*_{\scr C, q}(\A^1 \times \A^n_L, M)_h$ and $i_\infty^*(h_v^*(a)) = 0$, so $a = i_0^*(h_v^*(a)) = 0$ by Corollary \ref{cor:htpy-inv} and Lemma \ref{lemm:closed-pullback}(3).

Finally we use a transfer argument to treat the case where $L$ is finite. Let $a \in h^i(C^*_{\scr C, q}(\A^n_L, M))$ with $i > q$. Suppose $L$ has characteristic $p$, and let $l \ne p$ be a different prime. Let $p\colon Spec(L') \to Spec(L)$ be an infinite $l$-extension\NB{We have $L = \mathbb{F}_q$ and may put $L' = \cup_n \mathbb{F}_{q^{l^n}}$.}. The pullback map $p^*\colon C^*(\A^n_L, M) \to C^*(\A^n_{L'}, M)$ maps $C^*_{\scr C, q}(\A^n_L, M)$ into $C^*_{\scr C, q}(\A^n_{L'}, M)$, by Lemma \ref{lemm:finite-good}(3) below. Since the theorem is proved for $L'$, we find that $p^*(a) = 0$. By continuity, there is a finite subextension $Spec(L') \to Spec(L_0) \xrightarrow{q} Spec(L)$ with $q^*(a) = 0$. Since $L$ is perfect, $L_0/L$ is simple and we may choose an embedding $Spec(L_0) \hookrightarrow \A^1_L$. Using the minimal polynomial, we may trivialize $\omega_{L_0/L}$, and hence we obtain $\tr_q^0\colon C^{*}(\A^n_{L_0}, M) \to C^{*}(\A^n_{L}, M)$. By Lemma \ref{lemm:finite-good}(2) below again, the transfer $\tr_q^0$ maps $C^*_{\scr C, q}(\A^n_{L_0}, M)$ into $C^*_{\scr C, q}(\A^n_L, M)$. We have ${0 = \tr_q^0 q^*(a) = [L_0:L]_\epsilon a}$, by Lemma \ref{lemm:pushpull-d}. Applying the same argument again for some $p \ne l' \ne l$ and using Lemma \ref{lemm:pushpull-ideal} yields the desired result.
\end{proof}

In the above proof we made use of the following observation.
\begin{lemma}\label{lemm:finite-good}
Let $f\colon X \to Y$ be a morphism.
\begin{enumerate}
\item Suppose $f$ is finite and surjective. Let $\scr C$ a family of closed subsets of $X$. Denote by $f(\scr C)$ the family $\{f(W) \mid W \in \scr C\}$. Let $Z \subset Y$. Then $Z$ is $q$-good (with respect to $f(\scr C)$) if and only if $f^{-1}(Z)$ is $q$-good (with respect to $\scr C$).
\item Suppose that $f$ is finite, flat and surjective. Let $\scr C$ be a family of closed subsets of $Y$. Denote by $f^{-1}(\scr C)$ the family \[ \{W' \mid W' \text{ a component of $f^{-1}(W)$ for some } W \in \scr C\}. \] Let $Z \subset X$. Then $Z$ is $q$-good (with respect to $f^{-1}(\scr C)$) if and only if $f(Z)$ is $q$-good (with respect to $\scr C$).
\item Suppose that $f$ is flat. Let $Z \subset Y$ be $q$-good with respect to $\scr C$. Then $f^{-1}(Z)$ is $q$-good with respect to $f^{-1}(\scr C)$, defined as (2).
\end{enumerate}
\end{lemma}
\begin{proof}
Suppose $f$ is finite and surjective.
Let $Z_1 \subset X, Z_2 \subset Y$ be closed. Consider the composite $Z_1 \cap f^{-1}(Z_2) \xhookrightarrow{i} f^{-1}(f(Z_1) \cap Z_2) \xrightarrow{f'} f(Z_1) \cap Z_2$. Since $f$ is finite and surjective, so is its restriction $f'$. One checks easily that $f'i$ is also surjective. It follows that all three of $Z_1 \cap f^{-1}(Z_2), f^{-1}(f(Z_1) \cap Z_2)$ and $f(Z_1) \cap Z_2$ have the same dimension. Taking $Z_1 = X$ or $Z_2 = Y$ we in particular find that $\dim{Z_1} = \dim{f(Z_1)}$ and $\dim{Z_2} = \dim{f^{-1}(Z_2)}$.

Now (1) follows by taking $Z_1 \in \scr C$ and $Z_2 = Z$. For (2) we take $Z_1 = Z$ and $Z_2 \in \scr C$, noting that all components of $f^{-1}(Z_2)$ have the same dimension.

For (3) we just note that $f^{-1}(Z \times_X W) = f^{-1}(Z) \times_Y f^{-1}(W)$, and flat pullback preserves codimension \cite[Tags 02NM and 02R8]{stacks}. Hence $Z$ meets all components of $W$ in the same codimension as $f^{-1}Z$ meets $f^{-1}W$.
\end{proof}

\subsection{Hard moving}
We now prove a more general moving lemma. Let us first note the following reformulation.

\begin{lemma} \label{lemm:moving-reformulation}
Let $X \in \Sm_k$, $\scr C$ a family of closed subsets, $q \ge 0$ and $M \in \HI(k)$. Then the following statements are equivalent.
\begin{enumerate}
\item The canonical map $\alpha\colon C^*_{\scr C, q}(X, M) \to C^*(X, M)$ induces a surjection on $h^i$ for $i \ge q$, and an isomorphism for $i \ge q+1$.
\item For each $i \ge q$, $Z \subset X$ $q$-good (and closed) and each $a \in H^i(X \setminus Z, M)$ there exists $Z \subset W \subset X$ $q$-good (and closed) such that $a$ is in the kernel of $H^i(X \setminus Z, M) \to H^i(X \setminus W, M)$.
\end{enumerate}
\end{lemma}
\begin{proof}
Since filtered colimits preserve exact sequences of abelian groups, we have a long exact sequence
\[ \dots \to h^i(C^*_{\scr C, q}(X, M)) \to H^i(X, M) \to \colim_{\substack{Z \subset X \\ Z \qgood}} H^i(X \setminus Z, M) =: K^i \to \dots \]
Thus statement (1) is equivalent to $K^i = 0$ for $i \ge q$, which is just a reformulation of (2).
\end{proof}

The main argument of our hard moving result is contained in the following lemma.
\begin{lemma}\label{lemm:hard-moving}
Let $k$ be an infinite perfect field and let $X \in \Sm_k$ be affine with $\omega_X \wequi \scr O_X$. Let $M \in \HI_0(k)$ and $i \ge q \ge 0$.

Let $Z \subset X$ be $q$-good and $a \in H^i(X \setminus Z, M)$. Then there exists $Z \subset W \subset X$ $q$-good and closed such that the image of $a$ in $H^i(X \setminus W, M)$ vanishes.
\end{lemma}

Following \cite{levine2006chow}, we prove Lemma \ref{lemm:hard-moving} by employing \emph{general projections}. Thus let $X$ be affine of dimension $d$, embedded in $\A^N$. By a general projection, we mean a general element $\pi$ of the space of full rank, affine-linear projections from $\A^N$ onto $\A^d$. The geometric basis of the method of general projections is the following.
\begin{proposition} \label{prop:hard-moving-geometric}
Let $k$ be perfect and $X$ a smooth affine variety.
Possibly after a Veronese re-embedding of degree $\ge 2$ (see e.g. \cite[Section 3.2.2]{kai2015moving} for this notion), the following hold.
\begin{enumerate}
\item A general projection $\pi|_X$ is finite and flat.
\item Let $Z \subset X$ be $q$-good. For a general $\pi$, $\pi(Z) \subset \A^d$ is $q$-good (with respect to $\pi(\scr C)$).
\item Let $Z \subset X$ have positive codimension. Then for a general $\pi$, $Z \to \pi(Z)$ is birational, and $\pi$ is étale around the generic points of $Z$.
\item Let $Z \subset X$ have codimension $\ge q$ and be $r$-good, with $r < q$. Then for a general $\pi$, we have an equality of closed subsets $\pi^{-1}(\pi(Z)) = Z \cup W$, where $W$ has codimension $\ge q$ and is $(r+1)$-good.
\item Let $Z \subset X$ have codimension $\ge q$ and be $r$-good, $r < q$. For general $\pi$, $\pi(Z \cap R)$ is $(r+1)$-good, where $R$ denotes the ramification locus of $\pi$.
\item Let $Z \subset X$. For general $\pi$, $\pi^{-1}(\pi(R))$ meets $Z$ in codimension $\ge 1$.
\end{enumerate}
\end{proposition}
\begin{proof}
Statement (1) is for example mentioned in \cite[last sentence before Section 3.2.1]{kai2015moving}. Statement (2) follows from statement (4) with $r=q-1$, and Lemma \ref{lemm:finite-good}(1). For statement (3), see \cite[Section 3.2.3]{kai2015moving} and \cite[Lemma II.3.5.4]{levine1998mixed}. Statement (4) is the content of \cite[Lemma II.3.5.6]{levine1998mixed}.

(5) By Lemma \ref{lemm:finite-good}(1), we need to show that for general $\pi$, $\pi^{-1}(\pi(R \cap Z))$ is $(r+1)$-good. For general $\pi$ we have $\pi^{-1}(\pi(R \cap Z)) \subset \pi^{-1}(\pi(Z)) = Z \cup W$, where $W$ is $(r+1)$-good, by (4). Hence it suffices that $\pi^{-1}(\pi(R \cap Z)) \cap Z \subset \pi^{-1}(\pi(R)) \cap Z$ is $(r+1)$-good. Let $F \in \scr C$. Since $Z$ is $r$-good, $Z \cap F$ has codimension $\ge r$ in $F$. It thus suffices that $\pi^{-1}(\pi(R))$ meets $Z \cap F$ in positive codimension. We have thus reduced to (6).

(6) We may assume that $Z$ is irreducible. We thus need to show that $Z \not\subset \pi^{-1}(\pi(R))$, for general $\pi$. The condition is open, so we need only show that the set of such projections $\pi$ is non-empty. For this we may enlarge the base field, so we may assume that $Z$ has a rational point in each component. Thus we may assume that $Z$ is a rational point $x$. A general projection is finite and flat, and hence we have $\pi^{-1}(\pi(R)) \not\ni x$ if and only if $\pi$ is étale along $\pi^{-1}(\pi(x))$. This is true for general $\pi$, by \cite[last paragraph of Section 3.2.2]{kai2015moving}.
\end{proof}

\begin{proof}[Proof of Lemma \ref{lemm:hard-moving}.]
We prove the result by induction on $q$. For $q=0$ there is nothing to prove, since we may take $W=X$. Hence assume that the result has been proved for some $q-1$, we shall prove it for $q$. Since $i \ge q \ge q-1$, by induction there exists $Z \subset W_1 \subset X$ $(q-1)$-good such that $a$ vanishes when restricted to $X \setminus W_1$. The exact sequence $H^i_{W_1 \setminus Z}(X \setminus Z, M) \to H^i(X \setminus Z, M) \to H^i(X \setminus W_1, M)$ shows that $a$ can be supported on $W_1$. In other words there exists a cycle $c \in C^i_{W_1}(X \setminus Z, M)$ representing $a$. Because of the form of the Rost-Schmid complex, we may also assume that $W_1$ has codimension $\ge i \ge q$. Using Proposition \ref{prop:hard-moving-geometric} and the fact that $k$ is infinite (so a non-empty open subset of affine space has a rational point), we find a projection $\pi\colon X \to \A^d$ with the following properties:
\begin{enumerate}
\item $\pi$ is finite flat,
\item $\pi(Z)$ is $q$-good,
\item $\pi$ is birational on $W_1$ and étale around the generic points of $W_1$,
\item $\pi^{-1}(\pi(W_1)) = W_1 \cup W_2$ with $W_2$ $q$-good,
\item $\pi(W_1 \cap R)$ is $q$-good, where $R$ is the ramification locus of $\pi$.
\end{enumerate}
Put $W_3' = \pi(Z \cup [W_1 \cap R])$. By (2), (5) and Lemma \ref{lemm:finite-good}(1) we know that $W_3 := \pi^{-1}(W_3')$ is $q$-good. Choose a trivialization of $\omega_X$, i.e. a nowhere vanishing section $s \in \omega_X(X)$. We can write \[det(d\pi) := d\pi_1 \wedge \dots \wedge d\pi_d = u s, \quad u \in \scr O(X).\] By definition, the vanishing locus of $u$ is precisely $R$. In particular $u$ is a unit around the generic points of $W_1$, and so there is a well-defined chain $\lra{u} c$. Moreover $\partial(\lra{u} c)|_{X \setminus W_3} = \lra{u} \partial(c)|_{X \setminus W_3} = 0$, since $u$ is a unit around all points of $W_1 \setminus R \supset supp(c|_{X \setminus W_3})$. The trivialization $s$ of $\omega_X$ and the canonical trivialization of $\omega_{\A^d}$ provide us with a trivialization of $\omega_{X/\A^d}$, and hence via the construction from Section \ref{sec:RS-transfers} with a transfer map $\tr^s_\pi\colon C^{*}(X \setminus W_3, M) \to C^{*}(\A^d \setminus W_3', M)$. We obtain a cycle $\tr^s_\pi(\lra{u}c|_{X \setminus W_3})$ on $\A^d \setminus W_3'$. By easy moving, i.e. Theorem \ref{thm:moving}, and the reformulation in Lemma \ref{lemm:moving-reformulation}, we know that the result we want holds for $\A^d$. It follows that there is a $q$-good subset $W_3' \subset W_4' \subset \A^d$ such that $[\tr^s_\pi(\lra{u}c|_{X \setminus W_3})]|_{\A^d \setminus W_4'} = 0$. Let $W_4 = \pi^{-1}(W_4')$; this is $q$-good by Lemma \ref{lemm:finite-good}(1) again. We conclude that $c|_{X \setminus W_4}$ is cohomologous to \[c' := c - \pi^*(\tr^s_\pi(\lra{u}c|_{X \setminus W_4})),\] using that transfers are compatible with pullback along open immersions, as is clear from the pointwise definition of both operations. Here $\pi^*$ denotes the flat pullback from Section \ref{sec:RS-pullbacks}. We claim that $c'$ is supported on $W_2$; this will conclude the proof (taking $W = W_4 \cup W_2$).
\NB{Main idea of the proof: subtract from $c$ its push-pull to affine space, by easy moving the resulting cycle is cohomologous to $c$, by the miracle of general projections it is supported in better codimension. Because of the twists occuring in push-pull operations, we need the scaling by $\langle u \rangle$, and for this reason we need to make sure that ramification locus of a general projection is of reasonably humble size.}

To establish the claim, we need to compute $\pi^*(\tr^s_\pi(\lra{u}c|_{X \setminus W_4}))$, or at least its components at the generic points $\eta \in W_1$ of codimension $i$ in $X$. Let $b \in C^i_{W_1}(X, M)$ be any chain. We shall compute $\pi^*(\tr^s_\pi(b))_\eta$ in terms of $b_\eta \in M_{-i}(\eta, \omega_{\eta/X})$. Since $\pi$ is birational on $W_1$, $\eta \to \pi(\eta)$ is an isomorphism, and $\eta$ is the only point of $W_1$ mapping to $\pi(\eta)$. It follows that $\tr_\pi^s(b)_{\pi(\eta)} \in M_{-i}(\pi(\eta), \omega_{\pi(\eta)/\A^d})$ is the element corresponding to $b_\eta$ under the isomorphisms $\eta \wequi \pi(\eta)$ and \[ \omega_{\pi(\eta)/\A^d} \wequi \omega_{\eta/\A^d} \wequi \omega_{\eta/X} \otimes \omega_{X/\A^d} \wequi \omega_{\eta/X} \otimes \omega_X \otimes \omega_{\A^d}^\vee \wequi \omega_{\eta/X}, \]
where in the last isomorphism we have trivialized $\omega_X$ via $s$ and $\omega_{\A^d}$ via $\tau := dt_1 \wedge \dots \wedge dt_d$. Indeed this follows from the definition of the transfer on the Rost-Schmid complex (Section \ref{sec:RS-transfers}), our choice of trivialization of $\omega_{X/\A^d}$ and the observation that if $L/K$ is an isomorphism of fields, then the corresponding transfer is inverse to pullback\NB{this is essentially obvious. I thought I had seen a referencable statement like this in \cite{A1-alg-top} somewhere, but I can't find it right now.}. If $b' \in C^i(\A^d, M)$ and $\eta \in W_1$ as before, then since $\pi$ is étale and hence smooth around $\eta$ we know (from Lemma \ref{lemm:compute-flat-pullback}) that $\pi^*(b')_\eta \in M_{-i}(\eta, \omega_{\eta/X})$ corresponds to $b'_{\pi(\eta)} \in M_{-i}(\pi(\eta), \omega_{\pi(\eta)/X})$ under the isomorphisms $\eta \wequi \pi(\eta)$ and $\omega_{\eta/X} \wequi \pi^* \omega_{\pi(\eta)/\A^d}|_\eta$. All in all we find that $\pi^*(\tr_\pi^s(b))_\eta$ corresponds to $b_\eta \in M_{-i}(\eta, M)$, \emph{twisted by an automorphism of $\omega_{\eta/X}$}. Tracing through the above explanation, the automorphism is
\[ \omega_{\eta/X} \stackrel{(i)}{\wequi} \omega_{\eta/\A^d} \otimes \omega_{X/\A^d}^\vee \stackrel{(ii)}{\wequi} \omega_{\eta/\A^d} \stackrel{(iii)}{\wequi} \omega_{\eta/X}, \] where (i) is the tautological isomorphism, (ii) comes from $\omega_{X/\A^d} \wequi \scr O_X$ via $s$, and (iii) is induced by $\pi^*$. 

Let us show that this automorphism is precisely multiplication by $u(\eta)^{-1}$. To do so, we write everything in terms of the absolute sheaves $\omega_\eta$ and so on; the automorphism is thus \[ \omega_\eta \otimes \omega_X^\vee \wequi \omega_\eta \otimes \omega_{\A^d}^\vee \otimes \omega_{\A^d} \otimes \omega_X^\vee \wequi \omega_\eta \otimes \omega_{\A^d}^\vee \xrightarrow{\id \otimes \pi^*} \omega_\eta \otimes \omega_X^\vee. \] Here the first equivalence is via the canonical isomorphism $\omega_{\A^d}^\vee \otimes \omega_{\A^d} \wequi \scr O_X$ and the second one uses $\omega_{\A^d} \wequi \scr O_X$ via $\tau$ and $\omega_X^\vee \wequi \scr O_X$ via $s^\vee$. We can write any element in the source group as $\alpha \otimes s^\vee$ (for some $\alpha \in \omega_\eta$), which is mapped to $\alpha \otimes \tau^\vee \otimes \tau \otimes s^\vee$ under the first equivalence (note that this is the same as $\alpha \otimes \tau'^\vee \otimes \tau' \otimes s^\vee$ for generator $\tau' \in \omega_{\A^d}$). Then the second equivalence takes this to $\alpha \otimes \tau^\vee$, and the third one to $\alpha \otimes \pi^*(\tau)^\vee$. Hence the automorphism is multiplication by the unit $v$ such that $\pi^*(\tau)^\vee = v s^\vee$. Since $\pi^*(\tau) = det(d\pi)$, we obtain $v = u^{-1} = u(\eta)^{-1}$.

Hence $\pi^*(\tr_\pi^s(b))|_\eta = \lra{u(\eta)} b_\eta$, and consequently $\pi^*(\tr_\pi^s(\lra{u}c))_\eta = c_\eta$. This concludes the proof.
\end{proof}

\begin{theorem} \label{thm:hard-moving}
Let $k$ be a perfect field, $M \in \HI_0(k)$, $X \in \Sm_k$ affine, $q \ge 0$. Assume that $\omega_X \wequi \scr O_X$. Let $\scr C$ be a finite family of irreducible, closed subsets of $X$ (as always with $X \in \scr C$). Then the canonical map $\alpha\colon C^*_{\scr C,q}(X, M) \to C^*(X, M)$ induces a surjection on $h^i$ for $i \ge q$, and an isomorphism for $i \ge q+1$.
\end{theorem}
\begin{proof}
If $k$ is infinite, then the result follows via the reformulation in Lemma \ref{lemm:moving-reformulation} from Lemma \ref{lemm:hard-moving}. It thus remains to deal with $k$ finite.

We shall reduce to the case of infinite fields by a transfer argument. Let $Z \subset X$ be $q$-good and $a \in H^i(X \setminus Z, M)$. Suppose $char(k) = p$. Let $l \ne p$ be another prime, and $k'/k$ an infinite $l$-extension. Since the theorem is proven for $k'$, there exists $Z_{k'} \subset Z_1 \subset X_{k'}$ $q$-good such that $a|_{X_{k'} \setminus Z_1} = 0$. By continuity, there exists a finite subextension $k \subset k_0 \subset k'$ such that $Z_1 = Z_2 \times_{k_0} k'$, for some $Z_{k_0} \subset Z_2 \subset X_{k_0}$ $q$-good, and $a|_{X_{k_0} \setminus Z_2} = 0$. By Lemma \ref{lemm:finite-good}(2), the image $Z_3$ of $Z_2$ in $X$ is $q$-good. Replacing $Z$ by $Z_3$, we may assume that $a|_{X_{k_0} \setminus Z_{k_0}} = 0$. We have a transfer $\tr^0\colon C^*(X_{k_0} \setminus Z_{k_0}, M) \to C^*(X \setminus Z, M)$, obtained via embedding $Spec(k_0)$ into $\A^1_k$ and trivializing $\omega_{k_0/k}$ using the minimal polynomial. Now $0 = \tr^0 a|_{X_{k_0} \setminus Z_{k_0}} = [k_0:k]_\epsilon a$, by Lemma \ref{lemm:pushpull-d}. Applying the same argument for some $l \ne l' \ne p$ yields the desired result, using Lemma \ref{lemm:pushpull-ideal}.
\end{proof}

\section{The Bloch-Levine-Rost-Schmid complex}
\label{sec:BLRS}
Recall the cosimplicial scheme $\Delta^\bullet$ \cite[Section 2.1]{levine2008homotopy}.
\begin{definition} \label{def:BLRS}
Let $X$ be a smooth scheme, $M$ a strictly homotopy invariant sheaf and $q \ge 0$. Call a closed set $Z \subset X \times \Delta^n$ \emph{$q$-good} if it meets all faces in codimension $\ge q$. We define the weight $q$ \emph{Bloch-Levine-Rost-Schmid complex} $C^*(X, M, q)$ of $M$ on $X$ as follows. We put \[C^n(X, M, q) = C^n(X, M) \, \text{  if } \, n \ge q.\] For $n = q - i$ with $i > 0$ we put \[ C^n(X, M, q) = \colim_{\substack{Z \subset X \times \Delta^i \\ Z \qgood}} H^q_Z(X \times \Delta^i, M), \]
which is the same as $q$-th cohomology of the $q$-good Rost-Schmid complex $C^*_{\scr C,q}(X\times \Delta^i, M)$ where $\scr C$ is the family of all faces $X\times \Delta^j, \, j \leq i$. 
The differential $C^n(X, M, q) \to C^{n+1}(X, M, q)$ for $n \ge q$ is the Rost-Schmid differential. For $n < q-1$ it is obtained as the alternating sum of the pullbacks along the face maps. For $n=q-1$, note that $Z \subset X \times \Delta^0$ is $q$-good just when it has codimension $\ge q$; hence
\[ \colim_{\substack{Z \subset X \times \Delta^0 \\ Z \qgood}} H^q_Z(X, M) \wequi \ker(C^q(X, M) \to C^{q+1}(X, M)). \]
We define the differential as the composite
\[ C^{q-1}(X, M, q) =  \colim_{\substack{Z \subset X \times \Delta^1 \\ Z \qgood}} H^q_Z(X \times \Delta^1, M) \xrightarrow{i_1^* - i_0^*} \colim_{\substack{Z \subset X \times \Delta^0 \\ Z \qgood}} H^q_Z(X, M) \hookrightarrow C^q(X, M). \]
\end{definition}
It is immediate from the definition that $C^*(X,M,q)$ is indeed a complex. Moreover, by homotopy invariance of $M$, we find that $C^*(X, M, 0)$ is chain homotopy equivalent to $C^*(X, M)$.

\begin{example} \label{ex:higher-chow}
If $M = \ul{K}^M_q$, then $C^*(X, \ul{K}^M_q, q)$ coincides up to a shift with Bloch's cycle complex \cite{bloch1986algebraic}. Indeed, the Rost-Schmid part vanishes because $C^n(X, \ul{K}^M_q) = 0$ for $n>q$ (since $\ul{K}_i^M = 0$ for $i<0$), and the Bloch-Levine part is exactly Bloch's cycle complex because for $Z \subset X \times \Delta^i$ of codimension $q$ one computes that $H^q_Z(X \times \Delta^i, \ul{K}^M_q)$ is the free abelian group on points of codimension $q$ in $X \times \Delta^i$ which lie in $Z$. Consequently we find that $\CH^q(X, n) \wequi h^{q-n}(C^*(X, \ul{K}^M_q, q))$, where $\CH^q(X, n)$ denotes Bloch's higher Chow groups.
\end{example}

\begin{definition}\label{def:higher-chow-witt}
Guided by Example \ref{ex:higher-chow}, we put $\wt\CH^q(X, n) := \mathbb{H}^{q-n}_\Zar(C^*(X, \ul{K}^{MW}_q, q))$ and call this the \emph{higher Chow-Witt groups} of $X$.
\end{definition}

\begin{remark}
In~\cite[Example~4.1.12(2)]{DJK}, the term ``higher Chow-Witt groups" is used as the name for the Borel-Moore homology theory, associated to Milnor-Witt motivic cohomology. In favorable cases this definition agrees with ours, see~Corollary~\ref{corr:MW-mot-coho}.
\end{remark}

\begin{remark}
In contrast to our definition, Bloch defines his higher Chow groups as just the cohomology groups of his cycle complex, not the Zariski hypercohomology. Since Bloch's cycle complex satisfies Zariski descent (as a consequence of the localization theorem \cite[Corollary 3.2.2]{levine2008homotopy}), taking hypercohomology would not change the answer in Bloch's case. The comparison Theorem \ref{thm:comparison} below implies that for $X$ affine with $\omega_X \wequi \scr O_X$, taking hypercohomology is unnecessary, too.
\end{remark}

\subsection{Comparison with the homotopy coniveau tower}
Recall the homotopy coniveau tower from \cite{levine2008homotopy}. In brief, for $E \in \SHS(k)$ one puts\footnote{Observe that the indexing category is filtered, so this is automatically a homotopy colimit.} \[ E^{(q)}(X, n) = \colim_{\substack{Z \subset \Delta^n_X \\ \qgood}} E_Z(\Delta^n_X), \] where $E_Z(\Delta^n_X)$ means $E(\Delta^n_X/\Delta^n_X \setminus Z)$. The above construction is clearly functorial in $n \in \Delta^\op$; hence we obtain a simplicial spectrum. Its geometric realization is denoted by $E^{(q)}(X)$. Since a $(q+1)$-good subset is $q$-good, there is an evident map $E^{(q+1)}(X, \bullet) \to E^{(q)}(X, \bullet)$ (hence the name ``tower''). Moreover all of these construction can be made functorial in $X$, in an appropriate sense (but this is highly nontrivial; see the reference).

Hence in particular, for any $q \ge 0$ and $M \in \HI(k)$ we obtain a simplicial spectrum $M^{(q)}(X, \bullet)$ whose geometric realization is denote by $M^{(q)}(X)$. \NB{It is clear by construction that $M^{(q)}(X, \bullet)$ is in fact a simplicial $H\Z$-module.}

\begin{remark}
\label{rem: coniveau model}
Note that the Rost-Schmid complex is not functorial in pullbacks along closed embeddings of faces\footnote{Even if $M$ is a homotopy module, so that we can use the closed pullback construction from Section \ref{subsec:closed-pullback}, the pullback morphism depends on a choice of equations for the face and so if $i, j$ are composable face inclusions we cannot ensure that $(i \circ j)^* = j^* \circ i^*$, on the level of complexes. See also \cite[Section 13]{rost1996chow}.}, so we cannot use it as a strict model for $M^{(q)}(X, \bullet)$. Nevertheless, each separate term $R\Gamma_Z(\Delta^n_X, M)$ is equivalent to the complex $C^*_Z(\Delta^n_X, M)$.
\end{remark}

\begin{theorem} \label{thm:comparison}
Let $k$ be a perfect field, $q \ge 0$, $M \in \HI(k)$, $X \in \Sm_k$ affine such that $\omega_X \wequi \scr O_X$. Assume that $M_{-q}$ is a homotopy module.
Then there is a canonical equivalence of spectra 
\[M^{(q)}(X) \wequi C^*(X, M, q).\]
Here we treat $C^*(X,M,q)$ as an $H\Z$-module, and hence spectrum, in the usual way.
\end{theorem}
We will prove a stronger version in Proposition \ref{prop:second-comparison} below.
\begin{remark}
The equivalence depends on neither the isomorphism $\omega_X \wequi \scr O_X$ nor the choice of a homotopy module structure on $M_{-q}$.
Ultimately these assumptions are only used to apply Theorem \ref{thm:hard-moving} to deduce that certain canonical maps are isomorphisms.
\end{remark}

\begin{corollary}
\label{corr:coniveau hom module}
Let $k$ be a perfect field, $q \ge 0$, $M \in \HI(k)$ such that $M_{-q}$ is a homotopy module. Then there is a canonical equivalence of spectra  \[M^{(q)}(X) \wequi L_{\Zar} C^*(X, M, q),\] where $L_\Zar$ denotes the Zariski-localization functor on the category of presheaves of chain complexes (also known as taking hypercohomology).
\end{corollary}

\begin{proof}
The proof of Proposition \ref{prop:second-comparison} below shows that for any $X$, there is a canonical map $\alpha\colon C^*(X, M, q) \to M^{(q)}(X)$, which is functorial in $X$ in open immersions. By Theorem \ref{thm:comparison}, $L_\Zar(\alpha)$ is an equivalence. Since the right hand side is a Zariski sheaf, the statement follows.
\end{proof}

\begin{corollary} \label{corr:MW-mot-coho}
Suppose that $char(k) \ne 2$ and $q \ge 0, p \in \Z$. Then there is a canonical isomorphism \[\wt{\CH}^q(X, 2q-p) \wequi H^{p,q}(X, \tilde\Z)\] between the higher Chow-Witt groups of a smooth $k$-scheme $X$ and the Milnor-Witt motivic cohomology of $X$ in the sense of \cite[Definition 6.6]{calmes2014finite}.
\end{corollary}

\begin{proof}
MW-motivic cohomology is often only defined for smooth schemes over perfect fields. However, the constructions given turn cofiltered limits with smooth affine transition maps into filtered colimits, and hence it makes sense to define MW-motivic cohomology of pro-smooth (often called essentially smooth) $k$-schemes as a colimit. Since any field is pro-smooth over its prime subfield, this allows us to define MW-motivic cohomology of smooth schemes over any field. Our definition of higher Chow-Witt groups works over any field, and also turns pro-smooth cofiltered limits into colimits. Hence to prove the claim it suffices to treat the case where $k$ is perfect.

MW-motivic cohomology is represented by the spectrum $f_0 \ul{K}^{MW}$, which is the effective cover of the homotopy module of Milnor-Witt $K$-theory \cite[Theorem 5.2]{bachmann-criterion}. Since the homotopy coniveau tower implements the slice tower \cite[Theorem 9.0.3]{levine2008homotopy}, we have $(f_0 \ul{K}^{MW} \wedge \Gmp{q})(X) \wequi (\ul{K}_q^{MW})^{(q)}(X)$. Hence the result follows from Corollary~\ref{corr:coniveau hom module}.
\end{proof}

\begin{remark}
Note that for any $p, q \in \Z$ such that $p \geq 2q-1$ \NB{the case $p = 2q-1$ is always  confusing for me} it was known that $H^{p,q}(X, \tilde\Z) \wequi H^{p-q}(X, \ul{K}_q^{MW})$ for a smooth scheme $X$~\cite[Theorem~4.2.4]{deglise-fasel}. This comparison agrees with our result since the right-hand side can be computed via the Rost-Schmid complex.
\end{remark}

\subsection{Functoriality and the proof of comparison}
\label{sec:BLRS-pullback}
Let $f\colon Y \to X$ be a flat map. Then we have obvious maps $f^*\colon C^n(X, M, q) \to C^n(Y, M, q)$ for $n<q$. We also have such maps for $n \ge q$, via the construction from Section \ref{sec:RS-pullbacks}. Thus altogether we obtain a map of graded abelian groups \[f^*\colon C^*(X, M, q) \to C^*(Y, M, q).\]

\begin{lemma} \label{lemm:BLRS-flat-pullback} ${}$ 
\begin{enumerate}
\item The map $f^*\colon C^*(X, M, q) \to C^*(Y, M, q)$ is a morphism of chain complexes.
\item The following diagram commutes in the derived category
\begin{equation*}
\begin{CD}
C^*(X, M, q) @>>> M^{(q)}(X) \\
@Vf^*VV           @Vf^*VV \\
C^*(Y, M, q) @>>> M^{(q)}(Y). \\
\end{CD}
\end{equation*}
\end{enumerate}
\end{lemma}
\begin{proof}
(1) Note that if we have a commutative square in $\Sm_k$ as follows
\begin{equation*}
\begin{CD}
B' @>h>> B \\
@Ai'AA  @AiAA \\
A' @>g>> A,
\end{CD}
\end{equation*}
and given $Z \subset B$, then the following square also commutes ($\ast$)
\begin{equation*}
\begin{CD}
H^q_{h^{-1}(Z)}(B', M) @<h^*<< H^q_Z(B, M) \\
@Vi'^*VV @Vi^*VV   \\
H^q_{i'^{-1} h^{-1}(Z)}(A', M) @<g^*<< H^q_{i^{-1}(Z)}(A, M).
\end{CD}
\end{equation*}

Applying this with $A = X \times \Delta^n$, $B = X \times \Delta^{n+1}$, $A' = Y \times \Delta^n$, $B' = Y \times \Delta^{n+1}$, $g, h$ induced by $f$ and $i, i'$ corresponding to an inclusion of a face, we deduce that the map $f^*\colon C^{* \le q-1}(X, M, q) \to C^{* \le q-1}(Y, M, q)$ is a morphism of chain complexes. The map $f^*\colon C^{* \ge q}(X, M, q) \to C^{* \ge q}(Y, M, q)$ is also a morphism of chain complexes; this follows from Lemma \ref{lemm:flat-pullback-works}. Let us put \[ C'^q(X, M, q) = \colim_{\substack{Z \subset X \times \Delta^0 \\ Z \qgood}} H^q_Z(X, M), \] and similarly for $Y$. Consider the diagram
\begin{equation*}
\begin{CD}
C^{q-1}(X, M, q) @>f^*>> C^{q-1}(Y, M, q) \\
@V{\partial}VV             @V{\partial}VV \\
C'^{q}(X, M, q) @>f^*>> C'^{q}(Y, M, q)    \\
@VjVV                       @VjVV           \\
C^q(X, M, q)    @>f^*>> C^q(Y, M, q).
\end{CD}
\end{equation*}
Here the middle horizontal morphism $f^*$ is the one coming from the functoriality of $H_Z(X, M)$ in $(X,Z)$; just like the top horizontal morphism.
The map $\partial$ denotes the boundary map of the homotopy coniveau tower, i.e. an alternating sum of maps of the form $i^*$, and $j$ denotes the canonical inclusion. In particular the vertical composites are the boundaries in the BLRS complex; what we need to show is that the diagram commutes. Commutativity of ($\ast$) immediately implies that the upper square commutes, so it suffices to show that the lower square commutes. By construction of the flat pullback from Section~\ref{ssec: flat pullback}, this follows again from ($\ast$), applied with $B = X$, $B' = Y$, $Z$ of codimension $q$ and $A=U$ an open neighbourhood in $X$ of a generic point $\eta$ of $Z$, $A'$ an open neighbourhood of a generic point of $f^{-1}(Z)$ over $\eta$. (If $f$ is smooth, this can also be seen more directly by appealing to Lemma \ref{lemm:compute-flat-pullback}.)

(2) The construction of the map $C^*(X,M,q) \to M^{(q)}(X)$ in the proof of Proposition \ref{prop:second-comparison} below is easily seen to be functorial in flat maps.
\end{proof}

In order to define pullbacks along closed immersions, we need to further adapt the BLRS complexes.

\begin{definition}
Let $W \subset X$ be a closed subset. We call $Z \subset X \times \Delta^n$ \emph{$(W,q)$-good} if $Z$ meets all faces of $X \times \Delta^n$ in codimension $\ge q$, and also meets all faces of $W \times \Delta^n$ in codimension $\ge q$. In particular, $Z \subset X$ is $(W,q)$-good if $Z \subset X$ and $Z \cap W \subset W$ are of codimension $\ge q$. We put
\[ C^n(X, M, q)_W = \colim_{\substack{Z \subset X \\ Z\,(W,q)\text{-good}}} C^n_Z(X, M) \, \text{ if } \, n \ge q. \]
For $n = q-i$ with $i>0$ we put
\[ C^n(X, M, q)_W = \colim_{\substack{Z \subset X \times \Delta^i \\ Z\,(W,q)\text{-good}}} H^q_Z(X \times \Delta^i, M). \]
\end{definition}
There is an obvious map $C^*(X,M,q)_W \to C^*(X, M, q)$ which is in fact an injection, and exhibits $C^*(X,M,q)_W$ as a subcomplex of $C^*(X,M,q)$. Recall that there is a similarly adapted homotopy coniveau tower $M^{(q)}(X)_W$ \cite[Section 7.4]{levine2006chow}. Note that if $W = \emptyset$, then $C^*(X,M,q)_W = C^*(X, M, q)$, and similarly for the homotopy coniveau tower. Consequently the next proposition is indeed a strengthening of Theorem \ref{thm:comparison}.

\begin{proposition} \label{prop:second-comparison}
Let $k$ be a perfect field, $M \in \HI(k)$, $q \ge 0$ such that $M_{-q}$ is a homotopy module,  $X \in \Sm_k$ affine such that $\omega_X \wequi \scr O_X$, $W \subset X$ a closed subset. Then there is a canonical equivalence \[C^*(X,M,q)_W \xrightarrow{\sim} M^{(q)}(X)_W,\] making the following diagram commute
\begin{equation*}
\begin{CD}
C^*(X,M,q)_W @>>> M^{(q)}(X)_W \\
@VVV               @VVV        \\
C^*(X,M,q)   @>>> M^{(q)}(X).
\end{CD}
\end{equation*}
\end{proposition}
\begin{remark}
Note that the vertical map $M^{(q)}(X)_W \to M^{(q)}(X)$ is an equivalence if $M$ is a homotopy module or $k$ is infinite \cite[Theorem 2.6.2(2)]{levine2006chow}. Consequently so is the other vertical map $C^*(X,M,q)_W \to C^*(X,M,q)$. Thus we think of $C^*(X,M,q)_W$ as a version of the BLRS complex adapted to the geometry of $W \subset X$.
\end{remark}
\begin{proof} 
\NB{Main idea of the proof: consider $E_1$-page of the spectral sequence for homotopy groups of the simplicial spectrum $M^{(q)}(X, \bullet)$ and note that in degrees $i>q$ those homotopy groups are constant in the simplicial direction, thanks to the moving lemma and strict homotopy invariance of $M$. Hence on $E_2$ the spectral sequence has only (cohomology of) BLRS complex and thus degenerates. In s.s. BL part is vertical in coordinate $q$, RS part is horizontal in coordinate $0$ (hopefully). $RS$ part of BLRS corresponds to homotopy groups of $M^{(q)}$ in degrees $-q$ and lower.}
Throughout this proof we employ the theory of $\infty$-categories.
We identify ordinary $1$-categories with appropriate $\infty$-categories, via the nerve functor.
All categories are $\infty$-categories, all functors are $\infty$-functors, and so on.

Let $\scr D$ denote the ``fat simplex'' category; in other words the category of finite totally ordered sets and injections. There is a canonical functor $\scr D \to \Delta$ which is coinitial \cite[Lemma 6.5.3.7]{HTT}. Hence if $F\colon \Delta^\op \to \scr C$ is any functor (with $\scr C$ an $\infty$-category), then $\colim_{\Delta^\op} F \wequi \colim_{\scr D^\op} F$. 

Denote by $\Fun(\scr D^\op, H\Z\Mod)' \subset \Fun(\scr D^\op, H\Z\Mod)$ those functors $F$ such that $F([n]) \in H\Z\Mod_{\ge -q}$ for $n > 0$.\footnote{Here $H\Z\Mod_{\ge -q} \subset H\Z\Mod$ denotes the full subcategory on those chain complexes $C_*$ such that $H_i(C_*) = 0$ for $i < -q$.} Since colimits in diagram categories are computed sectionwise, the canonical inclusion $\Fun(\scr D^\op, H\Z\Mod)' \to \Fun(\scr D^\op, H\Z\Mod)$ has a right adjoint $\tau$, by the adjoint functor theorem. We determine $\tau$, as follows. For each $n$ we have the functor $ev_n\colon \Fun(\scr D^\op, H\Z\Mod) \to H\Z\Mod, F \mapsto F([n])$. This functor preserves limits so has a left adjoint $F_n$, given by left Kan extension \cite[Proposition 4.3.2.17]{HTT}. The (defining) formula for the left Kan extension says that for $X \in H\Z\Mod$ we have $F_n(X)([m]) \wequi \colim_{[m] \hookrightarrow [n]} X = \Map_{\scr D}([m], [n]) \otimes X$. This implies that $F_0(H\Z\Mod) \subset \Fun(\scr D^\op, H\Z\Mod)'$ and also for $n>0$ we have $F_n(H\Z\Mod_{\ge -q}) \subset \Fun(\scr D^\op, H\Z\Mod)'$. From this we deduce that for $F \in \Fun(\scr D^\op, H\Z\Mod)$ we have $\tau(F)([0]) = F([0])$ and $\tau(F)([n]) = \tau_{\ge -q}(F([n]))$ for $n > 0$. We verify similarly that the semi-simplicial structure maps in $\tau(F)$ are the canonical ones, and that the adjunction morphism $\tau(F) \to F$ is the canonical one.

Applying this to $F = M^{(q)}(X, \bullet)_W$ we obtain a morphism of semi-simplicial objects \[\alpha\colon \tau(M^{(q)}(X, \bullet)_W) \to M^{(q)}(X, \bullet)_W.\] We claim that (1) $|\alpha| := \colim_{\scr D^\op} \alpha$ is an equivalence and (2) $|\tau(M^{(q)}(X, \bullet)_W)| \wequi C^*(X, M, q)_W$. All results follow from this.

To prove claim (1), we shall utilize the spectral sequence of a semisimplicial object \cite[Proposition 1.2.4.5, Variant 1.2.4.9]{HA} (applied to $\scr C = H\Z\Mod$ with its standard $t$-structure). We thus obtain a morphism of spectral sequences \[\hat{\alpha}\colon E_*^{**}(\tau(M^{(q)}), X) \to E_*^{**}(M^{(q)}, X).\] On the $E_1$-page we see the semisimplicial abelian groups $A^{(i)}_\bullet = \pi_{-i}(M^{(q)}(X, \bullet)_W)$. \NB{$-i$ reflects the fact that cohomology $h^i$ of a complex corresponds to its $\pi_{-i}$ as a spectrum.} More specifically $E_1^{*,-i}(M^{(q)})$ is the unnormalized chain complex associated with $A^{(i)}_\bullet$. Also by construction we have
\begin{equation*}
  E_1^{*,-i}(\tau(M^{(q)})) = \begin{cases} E_1^{*,-i}(M^{(q)}), \, -i \le q \text{ or } * \le 0  \\ 0, \, \text{else.}  \end{cases}
\end{equation*}

The spectral sequence $E_*^{**}(M^{(q)})$ wants to converge to $\pi_*(M^{(q)}(X)_W)$ and similarly $E_*^{**}(\tau(M^{(q)}))$ wants to converge to $\pi_*(\tau(M^{(q)}(X)_W))$. We shall show that the spectral sequences converge and $\hat{\alpha}$ induces an isomorphism on the $E_2$-pages. This implies that $|\alpha|$ is an equivalence, by spectral sequence comparison. In order to get convergence, we wish to know that $M^{(q)}(X, \bullet)_W\colon \scr D^\op \to H\Z\Mod$ takes values in $H\Z\Mod_{\ge -N}$ for some $N$, and similarly for $\tau(M^{(q)})$. In other words we need to show that $E_1^{*,-i} = 0$ for $i > N$. In order to get the isomorphism of $E_2$-pages, it suffices to show the following condition: ($\ast$) for $i>q$ the semisimplicial abelian group $A^{(i)}_\bullet$ is constant. This also implies the vanishing we need for convergence: indeed for $i > \max\{q,\dim{X}\} =: N$ we get $A^{(i)}_* \wequi A^{(i)}_0 \wequi H^i_W(X, M) = 0$.

It is thus enough to show ($\ast$): each of the semi-simplicial structure maps of $A^{(i)}$ for $i>q$ is an isomorphism. Let $j\colon F \times X \to \Delta^r \times X$ be the inclusion of a codimension $1$ face; we have to show that $j^*\colon H^i_{q, \scr F}(\Delta^r \times X, M) \to H^i_{q, \scr F}(F \times X, M)$ is an isomorphism. Here $\scr F$ denotes the union of the families $F' \times X$ and $F' \times W$, where $F'$ runs through all faces of $\Delta^r$. Let $\tilde{M}$ be a homotopy module with $\tilde{M}_{-q} \wequi M_{-q}$. We have a diagram
\begin{equation*}
\begin{CD}
H^i_{q, \scr F}(\Delta^r \times X, \tilde M) @= H^i_{q, \scr F}(\Delta^r \times X, M) @>{j^*}>> H^i_{q, \scr F}(F \times X, M) @= H^i_{q, \scr F}(F \times X, \tilde M) \\
@VVV                                        @VVV                                  @VVV                     @VVV                     \\
H^i(\Delta^r \times X, \tilde{M})            @= H^i(\Delta^r \times X, M)             @>{j^*}>> H^i(F \times X, M)             @= H^i(F \times X, \tilde{M}).  \\
\end{CD}
\end{equation*}
The horizontal identifications come from the fact that by looking at the Rost-Schmid complex, we see that $H^i(Y, M)$ for $i > q$ only depends on $M_{-q} \wequi \tilde M_{-q}$. The vertical maps are the natural ones from cohomology with support to cohomology without support; in particular the diagram commutes. Since $i>q$, the outer vertical maps are isomorphisms by Theorem \ref{thm:hard-moving}. The lower middle horizontal map is an isomorphism by strict homotopy invariance of $M$. Hence $j^*$ is an isomorphism.

To prove claim (2), it remains to observe that we have
\begin{gather*}
 \tau(M^{(q)}(X, \bullet)_W)_n = \tau_{\ge -q}(M^{(q)}(X, n)_W) \simeq C^{q-n}(X, M, q)_W \quad \text{if } n>0; \\
\tau(M^{(q)}(X, \bullet)_W)_0 = M^{(q)}(X, 0)_W \simeq C^{*\ge q}(X, M, q)_W
\end{gather*}
essentially by construction. Indeed by Remark~\ref{rem: coniveau model}, we may compute $M^{(q)}(X, n)_W$ as a colimit of complexes $C^*_Z(\Delta^n_X, M)$ where $Z$ runs through $(W,q)$-good subsets of $\Delta^n_X$.
We have that $C^*_Z(\Delta^n_X, M) = 0$ for $*<q$ (since $Z$ has codimension $q$), hence $\tau_{\ge -q} C^*_Z(\Delta^n_X, M)$ is given by the chain complex concentrated in cohomological degree $q$ with value the abelian group $H^q_Z(\Delta^n_X, M)$.
The truncation commutes with filtered colimits, so the first equivalence follows; the other one is similar.

The geometric realization of the semi-simplicial object $\tau(M^{(q)}(X, \bullet)_W)_*$ is the associated unnormalized chain complex, which is isomorphic to $C^*(X,M,q)_W$ by inspection. This concludes the proof. 
\end{proof}

Now let $i\colon W \hookrightarrow X$ be a closed immersion with $W$ \emph{smooth}. For $n < q$ we have an obvious map \[i^*\colon C^n(X, M, q)_W \to C^n(W, M, q),\] induced by contravariance of cohomology with support. In fact, note that \[ \left( \tau_{\ge -q}C^*(X, M, q)_W \right)^q = \colim_{\substack{Z \subset X \times \Delta^0 \\ Z \qgood}} H^q_Z(X, M), \] and consequently the construction extends to $ \tau_{\ge -q} C^*(X, M, q)_W$, by the same formula.
\begin{lemma} \label{lemm:BLRS-closed-pullback} ${}$ 
\begin{enumerate}
\item The map of graded abelian groups $i^*\colon  \tau_{\ge -q}C^*(X, M, q)_W \to  \tau_{\ge -q}C^*(W, M, q)$ is a morphism of chain complexes.
\item The above construction induces a commutative diagram in the derived category
\begin{equation*}
\begin{CD}
 \tau_{\ge -q} C^*(X, M, q)_W @>>> M^{(q)}(X)_W \\
@Vi^*VV                                @Vi^*VV  \\
 \tau_{\ge -q} C^*(W, M, q) @>>> M^{(q)}(W)     \\
\end{CD}
\end{equation*}
\end{enumerate}
\end{lemma}
\begin{proof}
(1) Essentially the same as the proof of Lemma \ref{lemm:BLRS-flat-pullback}(1).

(2) We use ideas and notation from the proof of Proposition \ref{prop:second-comparison}. Given $F\colon \scr D^\op \to H\Z\Mod$, we had the truncation $\tau(F) \to F$. We can form a further truncation $\sigma(F) \to \tau(F)$, where $\sigma(F)_0 = \tau_{\ge -q} \tau(F)_0$. This is constructed precisely as before: let $\sigma$ be the right adjoint to $\Fun(\scr D^\op, H\Z\Mod)'' \hookrightarrow \Fun(\scr D^\op, H\Z\Mod)'$, where $F \in \Fun(\scr D^\op, H\Z\Mod)''$ exactly when $F([0]) \in H\Z\Mod_{\ge -q}$ as well. Thus we obtain a commutative diagram
\begin{equation*}
\begin{CD}
\sigma(M^{(q)}(X, \bullet)_W) @>>> \tau(M^{(q)}(X, \bullet)_W) @>>> M^{(q)}(X, \bullet)_W \\
@VVV                                         @VVV                                        @Vi^*VV                      \\
\sigma(M^{(q)}(W, \bullet)) @>>> \tau(M^{(q)}(W, \bullet)) @>>> M^{(q)}(W, \bullet). \\
\end{CD}
\end{equation*}
We have $|\sigma(M^{(q)}(X, \bullet)_W)| =  \tau_{\ge -q} C^*(X, M, q)_W$ and similarly for $W$, and the left hand vertical map is the one constructed above. The result follows. \NB{Here $M^{(q)}(X, \bullet)$ is the whole simplicial complex (bigraded thing), the truncation $\tau$ leaves the L-shape BLRS complex, and further truncation $\sigma$ leaves only the BL part.}
\end{proof}

\begin{remark} \label{rmk:reconstruction}
Employing the usual strictification procedures (see e.g.~\cite[Theorem 4.1.1]{levine2008homotopy}) the construction $X \mapsto  \tau_{\ge -q} C^*(X, M, q)$ can be promoted to a functor $F\colon (\Sm_k)^\op \to H\Z\Mod$ together with a natural transformation $\alpha\colon F \Rightarrow M^{(q)}$. Since $M^{(q)}$ is $-1$-connected in the Nisnevich topology (see \cite[Proposition 3.2(1)]{levine-slice}) and $F(X) \wequi \tau_{\ge -q} C^*(X, M, q) \wequi \tau_{\ge -q} M^{(q)}(X)$ Nisnevich locally on $X$, we see that $\alpha$ is a Nisnevich-equivalence (i.e. induces an isomorphism on homotopy sheaves). Since $M^{(q)}$ is a sheaf in the Nisnevich topology, we conclude that
\[ M^{(q)} \wequi L_\Nis F \in \Fun((\Sm_k)^\op, H\Z\Mod). \]
In other words, we have reconstructed $M^{(q)}$ from the complexes $ \tau_{\ge -q}C^*(\ph, M, q)$ together with the smooth and closed pullback maps constructed above. More concretely, we have reconstructed $M^{(q)}$ from the sheaf $M_{-q}$ together with the (smooth and closed) pullback maps on cohomology of $M$ with support in codimension $q$ (which as a group only depends on $M_{-q}$). 
\end{remark}

\begin{remark} \label{rmk: helpless}
If $M$ is a homotopy module, then the closed pullback on cohomology with support can be computed using the construction from Section \ref{subsec:closed-pullback}. In particular, this operation only depends on the sheaf $M_{-q}$ together with its structure of transfers and $\ul{GW}$-module. For a general $M$ we also can define a closed pullback by the construction from Section \ref{subsec:closed-pullback} (at least when $q \ge 2$), but we do not know if it coincides with the sheaf-theoretic pullback on cohomology with support. For this reason, even if $M_{-q}$ is a homotopy module we cannot conclude that $\ul{\pi}_0 M^{(q)}$ is a homotopy module.
\NB{If the pullbacks would coincide, $M^{(q)}$ would only depend on $M_{-q}$, so for $q \ge 2$ we would choose a homotopy module $M'$ such that $M_{-q} \simeq M'_{-q}$, and get $\ul{\pi}_0 M^{(q)} \simeq \ul{\pi}_0 M'^{(q)}$, where RHS is a homotopy module because effectivity tower commutes with $\omega^\infty$.}
\end{remark}

\section{Strictly homotopy invariant sheaves with generalized transfers}
\label{sec:generalized-transfers}

\subsection{Presheaves with $\A^1$-transfers}
Recall that if $M \in \Pre(\Sm_k)$ is a presheaf of sets on $\Sm_k$ and $X$ is an essentially smooth $k$-scheme (e.g. the spectrum of a finitely generated field extension of $k$) then we can make unambiguous sense of $M(X)$, by taking a colimit.
\begin{definition}
By a \emph{presheaf with $\A^1$-transfers} we mean a presheaf of abelian groups $M \in \Ab(\Sm_k)$ together with for each finitely generated field $K/k$ the structure of a $\ul{GW}(K)$-module on $M(K)$, and for each point $x \in (\A^1_K)^{(1)}$ (i.e. each monogeneous extension $K(x)/K$) a \emph{transfer map} \[\tau_x\colon M(K(x)) \to M(K).\] A morphism of presheaves with $\A^1$-transfers from $M_1$ to $M_2$ is a morphism of presheaves $\phi\colon M_1 \to M_2$ such that $\phi$ is compatible with the $\ul{GW}$-module structures and the transfers. We denote the category of presheaves with $\A^1$-transfers by $\Ab(\Sm_k)^\Atr.$ We write $\HI(k)^\Atr \subset \Ab(\Sm_k)^\Atr$ for the full subcategory on those presheaves with $\A^1$-transfers such that the underlying presheaf is strictly homotopy invariant.
\end{definition}

We note that in the definition of the category $\Ab(\Sm_k)^\Atr$ (and consequently also $\HI(k)^\Atr$) we do not ask for many of the usual compatibilities: we do not require $M$ to be a presheaf of $\ul{GW}$-modules, we do not require any base change or projection formulas, and so on. In cases of practical interest, these additional properties will usually hold, of course.

\begin{example} \label{ex:M-1-transfers}
The functor $\HI(k) \to \HI(k), M \mapsto M_{-1}$ factors canonically through the forgetful functor $\HI(k)^\Atr \to \HI(k)$, yielding $\HI(k) \to \HI(k)^\Atr, M \mapsto M_{\wh{-1}}$: for the structure of the transfers, see \cite[Section 4.2]{A1-alg-top}; for the structure of a module over $\ul{GW} = \ul{K}_0^{MW}$, see \cite[Lemma 3.49]{A1-alg-top}. The constructions of these structures make it clear that if $M \to M' \in \HI(k)$ then $M_{-1} \to M'_{-1}$ respects the transfers and $\ul{GW}$-module structure; hence we indeed have an induced morphism $M_{\wh{-1}} \to M'_{\wh{-1}}$. For the iterated contractions $M_{-n}$, we similarly denote by $M_{\wh{-n}} = (M_{-n+1})_{\wh{-1}}$ the canonical lift to $\HI(k)^\Atr$.
\end{example}

\begin{example} \label{ex:sheffheart-transfers}
The functor $\omega^\infty\colon \SH(k)^{\eff\heart} \to \SHS(k)^\heart \wequi \HI(k)$ factors canonically through the forgetful functor $\HI(k)^\Atr \to \HI(k)$, yielding $\wh{\omega^\infty}\colon \SH(k)^{\eff\heart} \to \HI(k)^\Atr$. Indeed, ${\omega^\infty(E) \wequi (\omega^\infty((E \wedge \Gm)_{\le 0}))_{-1}}$, and so we can use the factorization from Example \ref{ex:M-1-transfers}. \NB{$E \wedge \Gm$ as a spectrum need not be in the heart, so we truncate it; $(E \wedge \Gm)_{\le 0}$ corresponds to the shifted homotopy module.}
\end{example}

\subsection{Twisted transfers}
\label{subsec:twisted-transfers}

\begin{definition}
Let $M \in \Ab(\Sm_k)^\Atr$. We say that $M$ \emph{satisfies the first projection formula} if for every monogeneous extension $K(x)/K$, every $a \in GW(K)$ and $m \in M(K(x))$ we have $\tau_x(a|_{K(x)} m) = a \tau_x(m)$. 
\end{definition}

Suppose that $M \in \Ab(\Sm_k)^\Atr$. For $X$ the spectrum of a finitely generated field extension, denote by $M(\omega)(X)$ the group $M(\omega_X)(X) = M(X) \otimes_{\Z[\scr O^\times(X)]} \Z[\omega_X^\times]$. Assume that $M$ satisfies the first projection formula. Given a finite monogeneous extension $K(x)/K$, define \[\tr^x_{K(x)/K}\colon M(\omega)(K(x)) \to M(\omega)(K)\] as follows. Denote by $x \in \A^1_K$ the closed point corresponding to $K(x)/K$. We have $\omega_x \wequi \omega_{x/\A^1_K} \otimes \omega_{\A^1_K}|_x$. The coordinate $t$ on $\A^1$ induces $\omega_{\A^1_K} \wequi \omega_K|_{\A^1_K}$. The second fundamental exact sequence induces $\omega_{x/\A^1_K} \wequi (m_x/m_x^2)^*$, which we may trivialize by the (dual of the) minimal polynomial of $x$. Consequently we have found an isomorphism $\omega_x \wequi \omega_K|_x$. Now define $\tr^x_{K(x)/K}$ as $\tau_x$, twisted by the above isomorphism of line bundles. In other words $\tr^x_{K(x)/K}(m \otimes \alpha|_{K(x)}) = \tau_x(m) \otimes \alpha$; this is well-defined by the projection formula assumption. See also \cite[Section 5.1]{A1-alg-top}.

\begin{remark}
Suppose that $K(x)/K$ is separable. Let $df_1 \wedge \dots \wedge df_r$ generate $\omega_K$, so that the $df_1 \wedge \dots \wedge df_r|_x$ generates $\omega_x$, by separability. Tracing through the definitions, the above isomorphism $\omega_x \wequi \omega_K|_x$ sends $df_1 \wedge \dots \wedge df_r$ to $P'(x) df_1 \wedge \dots \wedge df_r|_x$\NB{or $1/P'(x)$?}, where $P$ denotes the minimal polynomial of $x$. Hence we get the formula \[\tr(a \otimes df_1 \wedge \dots \wedge df_r|_x) = \tau_x(\lra{P'(x)} a) \otimes df_1 \wedge \dots \wedge df_r.\] This recovers the construction of the cohomological transfers from the geometric ones in \cite[Section 4.2]{A1-alg-top}. In general, our construction coincides with Morel's \emph{absolute transfers}; see \cite[Remark 5.6(2)]{A1-alg-top}.
\end{remark}

More generally, given a finite extension $K(x_1, \dots, x_n)/K$, define recursively \[ \tr^{x_1, \dots, x_n}_{K(x_1, \dots, x_n)/K} = \tr^{x_1, \dots, x_{n-1}}_{K(x_1, \dots, x_{n-1})/K} \circ \tr^{x_n}_{K(x_1, \dots, x_n)/K(x_1, \dots, x_{n-1})}\colon M(\omega)(K(x_1, \dots, x_n)) \to M(\omega)(K). \]

\begin{definition} \label{def:twisted-transfers}
Let $M \in \Ab(\Sm_k)^\Atr$. We will say that $M$ \emph{admits twisted transfers} if $M$ satisfies the first projection formula and for any finite extension $L/K$ and any two sets of generators $x_1, \dots, x_n; y_1, \dots, y_m \in L$ we have $\tr^{x_1, \dots, x_n}_{L/K} = \tr^{y_1, \dots, y_m}_{L/K}$. We denote by $\Ab(\Sm_k)^\tw \subset \Ab(\Sm_k)^\Atr$ the full subcategory on those presheaves with $\A^1$-transfers that admit twisted transfers. We also put $\HI(k)^\tw = \HI(k)^\Atr \cap \Ab(\Sm_k)^\tw$. For $M \in  \Ab(\Sm_k)^\tw$ we put $\tr_{L/K} := \tr_{L/K}^{x_1, \dots, x_n}$, for any choice of generators $x_1, \dots, x_n \in L$.
\end{definition}

We stress that $\HI(k)^\tw$ is still ``too large'', i.e. contains objects that are not well-behaved. For example, we did not require the individual $GW(K)$-module structures to come from a $\ul{GW}$-module structure. Nonetheless we find $\HI(k)^\tw$ useful as a book-keeping device.

\begin{example} \label{ex:twisted-transfers-M-2}
For $n \ge 2$ and $M \in \HI(k)$ we have $M_{\wh{-n}} \in \HI(k)^\tw \subset \HI(k)^\Atr$. This is the content of \cite[Section 5.1]{A1-alg-top} (for the projection formula, see Lemma \ref{lemm:untwisted-projection}(3)).
\end{example}

\begin{example} \label{ex:twisted-transfer-sheffheart}
For $E \in \SH(k)^{\eff\heart}$ we have $\widehat{\omega^\infty}(E) \in \HI(k)^\tw \subset \HI(k)^\Atr$. Indeed we have $\omega^\infty(E) = (\omega^\infty((E \wedge \Gmp{2})_{\le 0}))_{-2}$.
\end{example}

As we have seen in Example \ref{ex:twisted-transfers-M-2}, if $M \in \HI(k)$ then $M_{-2}$ admits twisted transfers. The following lemma provides a geometric explanation for a weaker statement.
\begin{lemma} \label{lemm:twisted-transfer-as-collapse}
Let $M \in \HI(k)$, $K/k$ finitely generated and $Spec(L) \in (\A^n_K)^{(n)}$. Consider the collapse map
\begin{gather*}
  p\colon (\P^1)^{\wedge n} \wedge Spec(K)_+ \to (\P^1)^{\wedge n} \wedge Spec(K)_+ / \left[ (\P^1)^{\wedge n} \wedge Spec(K)_+ \setminus Spec(L) \right] \\
     \wequi \A^n_K / \A^n_K \setminus Spec(L) \wequi Th(N_{Spec(L)/\A^n_K}).
\end{gather*}
Then $(p[-n])^* = \tr_{L/K}\colon M_{-n}(L, \omega_{L/K}) \to M_{-n}(K)$.
\end{lemma}
Let us clarify the statement a bit. The embedding $Spec(L) \hookrightarrow \A^n_K$ provides us with generators $x_1, \dots, x_n$ of $L$ over $K$. We are claiming that $(p[-n])^* = \tr_{L/K}^{x_1, \dots, x_n} \otimes \omega_{K}^{-1}$ in the notation of Definition \ref{def:twisted-transfers}. Technically speaking, this only makes sense if $M_{-n}$ satisfies the projection formula, which we do not have if $n=1$. In this case by $\tr_{L/K}\colon M_{-1}(L, \omega_{L/K}) \to M_{-1}(K)$ we mean the map $M_{-1}(L, \omega_{L/K}) \wequi M_{-1}(L) \to M_{-1}(K)$, where we are given a generator $x \in L$, the first equivalence is via the minimal polynomial of $x$ and the second map is $\tau_x$, defined for $M_{-1}$ (see Example~\ref{ex:M-1-transfers}).
\begin{proof}
In the case $n=1$ the map $(p[-1])^*$ is the transfer map by construction, so we will reduce to this case.

Let us first clarify the following standard abuse of notation. By definition, $(\P^1)^{\wedge n} \wedge Spec(K)_+$ is the (pre)sheaf obtained from $(\P^1)^{\times n} \times Spec(K)$ by contracting down the sub(pre)sheaf $\partial (\P^1)^{\times n} \times Spec(K)$.

Now $Spec(L) \in \A^n_K \subset (\P^1)^{\times n} \times Spec(K)$ is disjoint from $\partial (\P^1)^{\times n} \times Spec(K)$, so it makes sense to define $(\P^1)^{\wedge n} \wedge Spec(K)_+ / \left[ (\P^1)^{\wedge n} \wedge Spec(K)_+ \setminus Spec(L) \right]$ as $(\P^1)^{\times n} \times Spec(K) / [(\P^1)^{\times n} \times Spec(K) \setminus Spec(L)]$. Throughout this proof, we will commit similar notational abuses without comment.

We denote $P^n := (\P^1)^{\wedge n}$.
We put $Z_i = Spec(K(x_1, \dots, x_i))$ and hence obtain a tower of finite morphisms $Spec(L) = Z_n \to Z_{n-1} \to \dots \to Z_1 \to Z_0 = Spec(K)$, together with closed embeddings $Z_{i+1} \hookrightarrow \A^1_{Z_i}$. Since $Z_{i+1}$ is finite over $Z_i$, the induced map $Z_{(i+1)+} \hookrightarrow P^1 \wedge Z_{i+}$ ``is also a closed embedding'', in the sense that $Z_{i+1} \hookrightarrow \P^1 \times Z_i$ is a closed embedding with image disjoint from $\{\infty\} \times Z_i$. Smashing with $P^{n-i-1}$ we obtain a filtration
\[ Z_{n} \hookrightarrow P^1 \wedge Z_{(n-1)+} \hookrightarrow P^2 \wedge Z_{(n-2)+} \hookrightarrow \dots \hookrightarrow P^n \wedge Z_{0+}. \]
We have thus factored the total collapse map $p$ into the composite of partial collapse maps
\[ p_i\colon P^n \wedge Z_{0+} \sslash P^n \wedge Z_{0+} \setminus P^{n-i} \wedge Z_{i+} \to P^{n} \wedge Z_{0+} \sslash P^{n} \wedge Z_{0+} \setminus P^{n-i-1} \wedge Z_{(i+1)+}, \]
where for a pointed set $A$ and (not necessarily pointed) subset $B \subset A$ we put $A \sslash B = A / (B \cup \{*\})$. \NB{I.e. ``crush $B$ down to the base point''.}

Note that for pointed sets $A, B, C$ with $A \subset B$ we have 
\[ (\ast) \quad C \wedge B \sslash (C \wedge B \setminus C \wedge A) \wequi C \wedge [B \sslash B \setminus A]. \]
Applying this with $C = P^{n-i}$, $B = P^i \wedge Z_{0+}$ and $A = Z_{i+}$ we find that the source of $p_i$ is \[P^{n-i} \wedge [P^i \wedge Z_{0+} \sslash P^i \wedge Z_{0+} \setminus Z_{i+}] \wequi P^{n-i} \wedge Th(N_{Z_{i} \subset P^i \wedge Z_{0+}}).\] Noting that $[P^{n-i} \wedge Th(N_{Z_{i} \subset P^i \wedge Z_{0+}}), M[n]] \wequi M_{-n}(Z_i, \omega_{Z_i/Z_0})$ we find that $(p_i[-n])^*$ takes the form
 \[\alpha := (p_i[-n])^*\colon M_{-n}(Z_{i+1}, \omega_{Z_{i+1}/Z_0}) \to M_{-n}(Z_{i}, \omega_{Z_{i}/Z_0}). \] We shall prove that this is precisely the twisted transfer for the extension $Z_{i+1}/Z_i$ (twisted by $\omega_{Z_0}$), which implies the desired result.

Noting that ($\ast$) is functorial, we find that $p_i = \id_{P^{n-i-1}} \wedge p_i'$, where
\[ p_i'\colon P^{i+1} \wedge Z_{0+} \sslash P^{i+1} \wedge Z_{0+} \setminus P^1 \wedge Z_{i+} \to P^{i+1} \wedge Z_{0+} \sslash P^{i+1} \wedge Z_{0+} \setminus Z_{(i+1)+} \]
is the canonical collapse map.
Let $X \to \A^i_{Z_0}$ be the henselization of $Z_i$ in $\A^i_{Z_0}$. Then $P^1 \wedge X_+ \to P^{i+1} \wedge Z_{0+}$ is a Nisnevich neighbourhood of $P^1 \wedge Z_{i+}$ and hence $p_i'$ is canonically isomorphic to the collapse map
\[ p_i''\colon P^1 \wedge X_+ \sslash P^1 \wedge X_+ \setminus P^1 \wedge Z_{i+} \to P^1 \wedge X_+ \sslash P^1 \wedge X_+ \setminus Z_{(i+1)+}. \]
Now $Z_i \hookrightarrow X$ has a smooth retraction (see e.g. \cite[Corollary
5.11]{deglise-regular-base}), so we obtain $\pi\colon X \to \A^i_{Z_i}$ with $\pi^{-1}(0) = Z_i$. In other words $\pi$ is a Nisnevich neighbourhood of zero. It follows that $p_i''$ is (using these choices) isomorphic to
\[ p_i'''\colon P^1 \wedge \A^i_{Z_i+} \sslash P^1 \wedge \A^i_{Z_i+} \setminus P^1 \wedge Z_{i+} \to P^1 \wedge \A^i_{Z_i+} \sslash P^1 \wedge \A^i_{Z_i+} \setminus Z_{(i+1)+}. \]
Using ($\ast$) once more, we find that $p_i''' = \id_{T^i} \wedge q_i$, where \[q_i \colon P^1 \wedge Z_{i+} \to P^1 \wedge Z_{i+} \sslash P^1 \wedge Z_{i+} \setminus Z_{(i+1)+}\] is the collapse map. This gives us precisely the definition of the twisted transfer \[\beta\colon M_{-n}(Z_{i+1}, \omega_{Z_{i+1}/Z_{i}}) \to M_{-n}(Z_{i}).\] Tracing through the definitions we find that $\alpha = \beta(\omega_{Z_{i}/Z_0})$, as needed.\tom{I am still not entirely happy with this argument. Think about it again before submission to a journal.}
\end{proof}

\begin{remark} 
If $M$ is a homotopy module, the above argument can be phrased more succinctly in the language of tangentially framed correspondences \cite[Section 2.3]{EHKSY}. Let $f\colon X \to Y$ be a finite flat, and hence syntomic, morphism of semilocal, essentially smooth $k$-schemes. Taking determinants induces a bijection between the set of homotopy classes of trivializations of $[L_f]$ in the K-theory space $K(X)$ and trivializations of the line bundle $\omega_f$. \NB{Here we are comparing $(K(X)^0)_{\le 1}$ and $Pic(X)$ as $1$-groupoids via $det$. Both have $\pi_0 = 0$ and $\pi_1= \scr O^\times$ on semilocal rings.}
For every trivialization $\tau$ of $[L_f]$ we obtain the tangentially framed transfer $\tr^{[\tau]}\colon M(X) \to M(Y)$, and the untwisted transfer $\tr^{det([\tau])}\colon M(X) \wequi M(X, \omega_{X/Y}) \to M(Y)$. 

We claim that $\tr^{det([\tau])} = \tr^{[\tau]}$ and that if $X \hookrightarrow \A^n_Y$ then $\tr^{[\tau]}$ is the pullback along the collapse map \[(\P^1_Y)^{\wedge n} \to (\P^1_Y)^{\wedge n} / (\P^1_Y)^{\wedge n} \setminus X \wequi Th(N_{X/\A^n_Y}) \wequi T^n \wedge X_+,\] where the last equivalence is via $\tau$. Note that the first claim follows from the second: if $X/Y$ is monogeneous, then pullback along the collapse map is the definition of $\tr^{det([\tau])}$, and for the general case note that both sides are compatible with composition, and at least if $Y$ is a field then $f$ can be factored into a sequence of monogeneous extensions.

To prove the second claim, we may assume given a Nisnevich neighbourhood $U$ of $X$ in $\A^n_Y$, a smooth retraction $r\colon U \to X$ and global sections $f_1, \dots, f_n \in \scr O(U)$ cutting out $X$ and inducing $[\tau]$. Then $\tr^{[\tau]}$ is defined, as in Example~\ref{ex: Voev transfers}, to be the pullback along the composite \[(\P^1_Y)^{\wedge n} \to (\P^1_Y)^{\wedge n} / (\P^1_Y)^{\wedge n} \setminus X \wequi U/U \setminus X \xrightarrow{(f_\bullet, r)} T^n \wedge X_+.\] It suffices to observe that ($L_\mot$ of) the map $(f_\bullet, r)$ is homotopic to the purity equivalence. Altogether we get the same description of $\tr^{det([\tau])}$ as stated in Lemma~\ref{lemm:twisted-transfer-as-collapse}.
\tom{I haven't checked this in detail. Think about it again before submitting to journal.}
\end{remark}

\subsection{Framed transfers}
We recall the following definitions from~\cite[Section~2.1]{EHKSY}, originally they are due to Voevodsky~\cite{voevodsky2001notes}.
\begin{definition}
Let $X$, $Y \in \Sm_k$. 
An \emph{(equationally) framed correspondence} $(Z, U, f_\bullet, g)$ from $X$ to $Y$ of level $n \geqslant 0$ consists of a closed subscheme $Z \subset \A^n_X$, finite over $X$, an étale neighbourhood $U \to \A^n_X$ of $Z$, a morphism $(f_1, \dots, f_n) \colon U \to \A^n$ such that $Z = f_\bullet^{-1}(0)$ as schemes, and a morphism $g\colon U \to Y$. 
Two framed correspondences are equivalent if they have the same support and the morphisms are the same up to refining the étale neighbourhoods. We denote by $\Fr_n(X, Y)$ the set of framed correspondences of level $n$ up to equivalence; it is pointed by the correspondence with empty support. 

Composition of framed correspondences is defined via maps
$\Fr_n(X, Y) \times \Fr_m(Y, V) \to \Fr_{n+m}(X, V)$ by sending
$((Z, U, f_\bullet, g), (Z', U', f'_\bullet, g')) \mapsto (Z \times_Y Z', U \times_Y U', (f_\bullet, f'_\bullet), g' \circ \pr_{U'}).$
With this composition we get a category $\Fr_*(k)$, where objects are  smooth $k$-schemes and morphisms are given by $\Fr_*(X, Y) = \vee_{n=0}^{\infty} \Fr_n(X, Y)$. There is a canonical functor $\Sm_k \to \Fr_*(k)$, sending morphisms of $k$-schemes to framed correspondences of level $0$. 
\end{definition}

\begin{definition}
A \emph{presheaf with framed transfers} is a presheaf on the category $\Fr_*(k)$.
A presheaf with framed transfers is called $\A^1$-invariant (respectively a Nisnevich sheaf) if its restriction to $\Sm_k$ is.
A presheaf with framed transfers $F$  is \emph{stable} if $F(\sigma_X) = \id_{F(X)}$, where $\sigma_X = (X, \A^1_X, \pr_{\A^1}, \pr_X) \in \Fr_1(X, X)$.
\end{definition}

\begin{example} \label{ex: Voev transfers}
Let $E \in \SH(k)$. Then the sheaf $\Omega^\infty(E)$ acquires canonical framed transfers as follows (see~\cite[Section~3.2.6]{EHKSY2}). A framed correspondence $\alpha = (Z, U, f_\bullet, g) \in \Fr_n(X, Y)$ induces a map of quotient sheaves
\[ (\P^1)^{\wedge n} \wedge X_+ \to (\P^1)^{\wedge n} \wedge X_+ / (\P^1)^{\wedge n} \wedge X_+ \setminus Z \wequi U/U \setminus Z \xrightarrow{(f_\bullet, g)} \A^n/\A^n \setminus 0 \wedge Y_+, \]
which in turn induces a map of spectra \[ \overline{\alpha} \wedge (\P^1)^{\wedge n} \colon \Sigma^\infty X_+ \wedge (\P^1)^{\wedge n} \to \Sigma^\infty Y_+ \wedge \A^n/\A^n \setminus 0 \wequi \Sigma^\infty Y_+ \wedge (\P^1)^{\wedge n} . \]
Then  $\alpha^* \colon \Omega^\infty(E)(Y) \to \Omega^\infty(E)(X)$ is given by the pullback in $E$ along $\overline{\alpha}$.
\end{example}

Note that if $E \in \SH(k)^{\eff\heart}$, then $\Omega^\infty(E) \wequi \ul{\pi}_0(E)_0$, and so $\ul{\pi}_0(E)_0$ acquires framed transfers.
\begin{theorem} \label{thm:identify-effective-heart}
Let $k$ be a perfect field. Then the functor 
\[ \underline{\pi}_0(\ph)_0 \colon \SH(k)^{\eff} \to \sideset{}{^{\fr}}\prod(k) \] 
induces an equivalence between $\SH(k)^{\eff\heart}$ and the category $\prod^{\fr}(k)$ of $\A^1$-invariant stable Nisnevich sheaves with framed transfers of abelian groups.
\end{theorem}

\begin{remark}
By~\cite[Theorem~1.1]{GarkushaPaninStrictHomInv}, $\A^1$-invariant stable Nisnevich sheaves with framed transfers of abelian groups are necessarily strictly $\A^1$-invariant, at least when the base field $k$ is perfect. This result was proved in \textit{loc. cit.} only for perfect fields $k$ that are infinite and of characteristic not $2$; these additional assumptions were later removed in~\cite{DruzhininKyllingOddChar} and~\cite{DruzhininPaninChar2} respectively.
\end{remark}

Similar results to Theorem~\ref{thm:identify-effective-heart} have also been proved in \cite[Proposition 3.11]{anan-neshitov2017transfers} and \cite[Proposition 29]{bachmann-tambara}.
\begin{proof}
We freely use the language of $\infty$-categories in this proof.

We use the motivic recognition principle \cite[Theorem 3.5.14]{EHKSY}, which holds over any perfect field. To recall it briefly, there is a semiadditive $\infty$-category $\Cor^\fr(\Sm_k)$ under $\Sm_k$. Write $\HH^\fr(k)$ for the full subcategory of $\PSh(\Cor^\fr(\Sm_k))$ consisting of $\A^1$-invariant presheaves of spaces on $\Cor^\fr(\Sm_k)$ that satisfy Nisnevich descent after restriction to $\Sm_k$ via the canonical functor $\Sm_k \to \Cor^\fr(\Sm_k)$. Since $\Cor^\fr(\Sm_k)$ is semiadditive, each object of $\HH^\fr(k)$ is a presheaf of $\E_\infty$-monoids in a natural way. Denote by $\HH^\fr(k)^\gp \subset \HH^\fr(k)$ the subcategory of presheaves of grouplike $\E_\infty$-monoids. Then there is a canonical equivalence $\HH^\fr(k)^\gp \wequi \SH(k)^\veff$; this is the recognition principle.

Recall that for any $\infty$-category $\scr C$, there is the subcategory $\scr C_{\le 0}$ of $0$-truncated objects, i.e. those objects $E \in \scr C$ such that for every $F \in \scr C$ the space $\Map(F, E)$ is $0$-truncated. With this notation, $\SH(k)^{\eff\heart} = \SH(k)^\veff_{\le 0}$ and consequently $\SH(k)^{\eff\heart} \wequi \HH^\fr(k)^\gp_{\le 0}$. In other words $\SH(k)^{\eff\heart}$ is equivalent to the category of presheaves on $\Cor^\fr(\Sm_k)$ which are $0$-truncated, Nisnevich sheaves, $\A^1$-invariant and grouplike. For an $\infty$-category $\scr C$, we have $\PSh(\scr C)_{\le 0} \wequi \Fun(\scr C^\op, \Spc_{\le 0}) \wequi \Pre(\h\scr C)$ \cite[Proposition 1.2.3.1]{HTT}, where $\h\scr C$ denotes the homotopy category of $\scr C$, and $\Pre$ means presheaves of sets. Hence $\SH(k)^{\eff\heart}$ is equivalent to the subcategory of $\Pre(\h\Cor^\fr(\Sm_k))$ consisting of Nisnevich sheaves (of sets) that are grouplike and $\A^1$-invariant. This is the same thing as $\A^1$-invariant Nisnevich sheaves of abelian groups on $\h\Cor^\fr(\Sm_k)$. This category is equivalent to $\prod^\fr(k)$ by the same argument as in~\cite[Remark~3.4.10]{EHKSY} (replacing Zariski descent with Nisnevich descent in \textit{loc.cit.}). 
\end{proof}

Suppose we are given a framed correspondence $(Z, U, f_\bullet, g) \in \Fr_n(X, Y)$. We recall the construction of the colomology class $c(f_\bullet) \in H^n_Z(U, \ul{K}^{MW}_n)$, corresponding to the Koszul complex of the regular sequence $f_\bullet$ (see~\cite[p.~12]{deglise-fasel}).
Denote by $|f_i|$ the vanishing locus of $f_i$; then $Z = |f_1| \cap \dots \cap |f_n|$ as a set. Each $[f_i] \in \oplus_{u \in U^{(0)}} K_1^{MW}(u)$ gives an element $\partial [f_i] \in \oplus_{x \in U^{(1)}} K_0^{MW}(x, \omega_{x/U})$ supported on $|f_i|$, which  defines a cycle $c(f_i) \in H^1_{|f_i|}(U, \ul{K}_1^{MW})$. One defines then \[c(f_\bullet) = c(f_1) \times \ldots \times c(f_n) \in H^n_Z(U, \ul{K}_n^{MW}).\] 

If furthermore $Z_\red$ is \emph{smooth}, using the canonical isomorphism $\omega_{Z_\red/U} \wequi \omega_{Z_\red/X}$ we find that 
\[c(f_\bullet) \in H^0(Z_\red, \ul{GW}(\omega_{Z_\red/X})) \simeq H^n_Z(U, \ul{K}^{MW}_n).\]

\begin{lemma} \label{lemm:framed-transfer-via-twisted}
Let $E \in \SH(k)^{\eff\heart}$ and $\alpha = (Z, U, f_\bullet, g) \in \Fr_n(X, Y)$ such that $Z_\red$ is (essentially) smooth. Then $\alpha^*\colon E(Y) \to E(X)$ is given by the composite
\[ E(Y) \xrightarrow{i^*g^*} E(Z_\red) \xrightarrow{\times c(f_\bullet)} H^0(Z_\red, E(\omega_{Z_\red/X})) \xrightarrow{\tr_{Z_\red/X}} E(X). \]
Here $i\colon Z_\red \to U$ denotes the closed immersion and $\tr_{Z_\red/X}$ denotes the twisted transfer from Section \ref{sec:RS-transfers}. \NB{There is defined twisted transfer map on RS-complexes, which after taking $h^0$ gives transfer on $E$.}
\end{lemma}

\begin{proof}
Put $F = \ul{\pi}_0(E \wedge \Gmp{n})$; so in particular $E \wequi F_{-n}$. By Example~\ref{ex: Voev transfers}, 
$\alpha^*$ is the pullback in $F[n]$ along the following map
\[ (\P^1)^{\wedge n} \wedge X_+ \to (\P^1)^{\wedge n} \wedge X_+ / (\P^1)^{\wedge n} \wedge X_+ \setminus Z \wequi U/U \setminus Z \xrightarrow{(f_\bullet, g)} \A^n/\A^n \setminus 0 \wedge Y_+. \]
In particular we may assume that $U=Y$; then $h = (f_\bullet, \id_U) \colon U \to \A^n \times U$ is a closed embedding. 
In that case we express the pull-back along $h$ in $F[n]$ as the following composition:
\[ 
\beta\colon E(U) \xrightarrow{\sim} H^n_{\{0\} \times U}(\A^n \times U, F) \xrightarrow{h^*} H^n_Z(U, F) \wequi H^0(Z_\red, E(\omega_{Z_\red/U})),
 \]
and apply the canonical isomorphism $\omega_{Z_\red/U} \simeq  \omega_{Z_\red/X}$.
The left hand arrow is the Thom isomorphism, and it factors as 
\[E(U) \xrightarrow{\pr_U^*} H^0(\A^n \times U, E) \xrightarrow{\times t_n}  H^n_{\{0\} \times U}(\A^n \times U, F).\] Here $t_n  \in H^n_{\{0\} \times U}(\A^n \times U, \ul{K}_n^{MW})$ is the oriented Thom class of the trivial vector bundle over $U$ of rank $n$ (see~\cite[Definition~3.4]{levine-enum}). Let $y_i$ denote the coordinate functions on $\A^n$.
Note that $t_n = \partial_1[y_1] \times \dots \times \partial_n[y_n]$ (where $\partial_i = \partial^{\A^n \times U}_{\{y_i =0\}}$); this holds since the Thom class is multiplicative with respect to direct sums of vector bundles \cite[Prop.~3.7(2)]{levine-enum} and $t_1 = \partial_1[y_1]$ by construction \cite[p.~29]{levine-enum}. 

Let $a \in E(U)$; then $h^*(t_n \times a) = h^*(t_n) \times h^*(a) = h^*(t_n) \times a$, because $h^* \circ \pr_U^* = (\pr_U \circ h)^* = \id$. Since $U \simeq h(U) \subset \A^n \times U$ is cut out by the equations $\{f_i = y_i\}_{i=1}^n$, we observe that 
\[h^*(t_n) = h^*(\partial_1[y_1]) \times \dots \times h^*(\partial_n[y_n]) = \partial'_1[f_1] \times \dots \times \partial'_n[f_n] = c(f_\bullet)\] (where $\partial'_i = \partial^{U}_{\{f_i =0\}}$ in the Rost-Schmid complex for $\ul{K}_n^{MW}$; i.e. this is the boundary map in the long exact sequence of cohomology with support). This holds because the long exact sequence of cohomology with support is compatible with pullbacks. We have thus shown that ${\beta(a) = c(f_\bullet) \times i^*(a)}$.

Finally, the map $H^0(Z_\red, E(\omega_{Z_\red/X})) \xrightarrow{\tr_{Z_\red/X}} E(X)$ was computed in Lemma \ref{lemm:twisted-transfer-as-collapse} as the pullback along the collapse map $\gamma\colon (\P^1)^{\wedge n} \wedge X_+ \to (\P^1)^{\wedge n} \wedge X_+ / (\P^1)^{\wedge n} \wedge X_+ \setminus Z$ (here we reduce to the case of $X$ the spectrum of a field by unramifiedness of $E$).

Since $\alpha^* = \gamma^* \circ \beta$ as pointed out at the beginning, this concludes the proof. 
\end{proof}

\begin{corollary} \label{corr:omega-infty-hat-ff}
Let $k$ be a perfect field. Then the functor $\wh{\omega^\infty}\colon \SH(k)^{\eff\heart} \to \Ab(\Sm_k)^\Atr$ is fully faithful.
\end{corollary}
\begin{proof}
Under our assumptions on $k$, by Theorem~\ref{thm:identify-effective-heart} we have an equivalence between $\SH(k)^{\eff\heart}$ and the category of (strictly) homotopy invariant sheaves with an action by framed correspondences. It is thus enough to show the following. If $E, F \in \SH(k)^{\eff\heart}$ and $\phi\colon \wh{\omega^\infty}(E) \to \wh{\omega^\infty}(F) \in \Ab(\Sm_k)^\Atr$ is a morphism of the corresponding sheaves, preserving the $\ul{GW}$-action and (twisted) transfers, then $\phi$ preserves the action by framed correspondences.

Let $\alpha = (Z, U, f_\bullet, g)$ be a framed correspondence from $X$ to $Y$. Let $X^{(0)}$ be the set of generic points of $X$; this is an essentially smooth scheme. Denote by $\alpha_0$ the restriction of $\alpha$ to $X^{(0)}$; this is a framed correspondence from $X^{(0)}$ to $Y$. The way composition of framed correspondences is set up implies that the following diagram commutes
\begin{equation*}
\begin{CD}
F(Y) @>{\alpha^*}>> F(X) \\
@|                  @VVV \\
F(Y) @>{\alpha_0^*}>> F(X^{(0)}),
\end{CD}
\end{equation*}
and similarly for $E$. By unramifiedness, the restriction $F(X) \to F(X^{(0)})$ is injective. It follows that $\phi$ preserves the action of $\alpha$ if and only if it preserves the action of $\alpha^0$. Consequently we may assume that $X$ is a finite disjoint union of spectra of fields. Then $Z_\red$ is essentially smooth, and hence by Lemma \ref{lemm:framed-transfer-via-twisted} the action of $\alpha$ is determined in terms of the twisted transfers and the $\ul{GW}$-action. By assumption, $\phi$ preserves the latter two, hence $\phi$ preserves the action of $\alpha$. This was to be shown.
\end{proof}

\begin{definition}
We denote by $\HI(k)^\fr \subset \HI(k)^\tw \subset \Ab(\Sm_k)^\Atr$ the essential image of the (fully faithful) functor $\wh{\omega^\infty}\colon \SH(k)^{\eff\heart} \to \Ab(\Sm_k)^\Atr$.
\NB{A priori $\HI(k)^\fr$ is not the same category as $\prod^{\fr}(k)$ because here morphisms don't necessarily have to preserve framed transfers. However, they turn out to be equivalent.}
\end{definition}

In a rather roundabout fashion, we have now arrived at the main theorem of this section.
\begin{theorem} \label{thm:M-2-delooping}
Let $k$ be a perfect field and $M \in \HI(k)$. Then $M_{\wh{-3}} \in \HI(k)^\fr \subset \Ab(\Sm_k)^\Atr$. If $char(k)=0$, also $M_{\wh{-2}} \in \HI(k)^\fr \subset \Ab(\Sm_k)^\Atr$.
\NB{-3 appears appears in Prop~\ref{prop:transfer-specialization}}
\end{theorem}

\begin{proof}
Let $F = M_{-2}$ or $M_{-3}$.
Let $\alpha = (Z, U, f_\bullet, g)$ be a framed correspondence from $X$ to $Y$. Let $\eta \in X$ be a generic point. For $a \in F(Y)$, define
\[ \alpha^*(a)_\eta := \tr_p(c(f_\bullet)_\eta i^*(g^*(a))) \in F(\eta). \]
Here $p\colon (Z_\eta)_\red \to \eta$ is the canonical finite projection, and $i\colon (Z_\eta)_\red \to U$ is the inclusion. The elements $\alpha^*(a)_\eta$ for various $\eta$ define an element $\alpha^*(a)^{(0)} \in F(X^{(0)})$. We claim that $\alpha^*(a)^{(0)} \in F(X) \subset F(X^{(0)})$ (here we use that $F$ is unramified). By the strong form of unramifiedness, it suffices to show that $\alpha^*(a)^{(0)} \in F(X_x)$ for each $x \in X^{(1)}$ (where $X_x$ denotes the localization in $x$). This is proved in Lemma \ref{lemm:framed-correspondence-integral} below. We denote by $\alpha^*(a) := \alpha^*(a)^{(0)} \in F(X) \subset F(X^{(0)})$ this element.

It remains to prove that this action is compatible with composition of framed correspondences. Denote by $\beta = (Z', V, h_\bullet, k)$ a further framed correspondence from $Y$ to $W$. The following diagram illustrates some of these schemes and maps
\begin{equation*}
\begin{tikzcd}
Z_{12} \ar[d] \ar[rr] & & Z' \ar[r, hookrightarrow] \ar[d] & V \ar[r, "k"] \ar[dl] & W \\
Z_\red \ar[d] \ar[r, hookrightarrow] & U \ar[r, "g"] \ar[dl] & Y \\
X.
\end{tikzcd}
\end{equation*}
Here $Z_{12} := (Z \times_Y Z')_\red$.

Let $b \in F(W)$. We need to show that $\alpha^* \beta^* (b) = (\beta \circ \alpha)^*(b)$. Composition on both sides is compatible with open immersions, so we may assume that $X$ is connected. Using unramifiedness, we may further assume that $X$ is the spectrum of a field, and in particular that $Z_\red$ is essentially smooth. Denote the composite $Z_\red \hookrightarrow U \xrightarrow{g} Y$ by $r$.
Consider the composed correspondence \[ \beta \circ r = ( Z_\red \times_Y Z', Z_\red \times_Y V, h_\bullet \circ \pr_V, k \circ \pr_V). \] In particular the reduced scheme corresponding to the support of $\beta \circ r$ is $Z_{12}$. Note that $Z_{12}$ is essentially smooth, being finite and reduced over $Z_\red$, which is a finite disjoint union of spectra of fields by our assumption on $X$. We have \[ (\beta \circ r)^*(b) = \tr_{Z_{12}/Z_\red}(c(h_\bullet \circ \pr_V) \times b) \] (we denote by $b$ also its restriction (pullback) to any scheme over $W$, in particular $Z_{12}$). 
By Lemma \ref{lemm:framed-correspondences-compose} below, we find that $r^*\beta^*(b) = (\beta \circ r)^*(b)$. Using the projection formula (Corollary \ref{corr:twisted-projection}(3)) and compatibility of transfers with composition, we find that
\[ \alpha^*\beta^*(b) = \tr_{Z_\red/X}(c(f_\bullet) \times (\beta \circ r)^*(b)) = \tr_{Z_{12}/X}(\pr_U^*(c(f_\bullet)) \times c(h_\bullet \circ \pr_V) \times b). \]  
Since $\pr_U^*(c(f_\bullet)) = c(f_\bullet \circ \pr_U)$,
 we obtain that
\[ \alpha^*\beta^*(b) = \tr_{Z_{12}/X}(c(f_\bullet \circ \pr_U, h_\bullet \circ \pr_V) \times b). \]
This is precisely $(\beta \circ \alpha)^*(b)$, as was to be shown.
\end{proof}

\begin{lemma}\label{lemm:framed-correspondence-integral}
Notation as in the proof of Theorem \ref{thm:M-2-delooping}.

Let $C$ be the localization of a smooth $k$-scheme in a point of codimension $1$, and let $\alpha = (Z, U, f_\bullet, g)$ be a framed correspondence from $C$ to $Y$. Write $\eta \in C$ for the generic point. Then for $a \in F(Y)$ we have $\alpha^*(a)_\eta \in F(C) \subset F(\eta)$.
\end{lemma}
\begin{proof}
We shall denote by $a$ also its restriction to other schemes, when no confusion can arise (for example, we may say that $a \in F(U)$). We may assume that $Y=U$ and $g=\id$.

Let $x$ be the closed point of $C$. We are attempting to prove that $0 = \partial_x(\alpha^*(a)_\eta) \in F_{-1}(x, \omega_{x/C})$. \NB{$0 \to F(C) \to F(\eta) \to F_{-1}(x, \omega_{x/C})$ is exact.}
Let $\tilde{Z}$ denote the normalization of $Z_\red$. Then $\tilde{Z}$ is essentially smooth and $\tilde{p}\colon \tilde{Z} \to C$ is finite \cite[Tags 035R, 0335, 032S]{stacks}. In particular there is a transfer map $\tr_{\tilde{p}}\colon C^*(\tilde{Z}, F(\omega_{\tilde{Z}/C})) \to C^*(C, F)$ which is a morphism of complexes \cite[Corollary 5.30]{A1-alg-top}. We have $a \in C^0(\tilde{Z}, F)$ and $c := c(f_\bullet) \in C^0(\tilde{Z}, \ul{GW}(\omega_{\tilde{Z}/C}))$ (here we use that $(\omega_{\tilde{Z}/C})_\eta \wequi (\omega_{Z_\red/C})_\eta$, $\tilde{Z}$ and $Z$ being birational). We obtain $ac \in C^0(\tilde{Z}, F(\omega_{\tilde{Z}/C}))$ and by definition, $\alpha^*(a)_\eta = \tr_{\tilde{p}}(ac)$. Since $\tr_{\tilde{p}}$ is a morphism of complexes, what we need to show is that $0 = \tr_{\tilde{p}}(\partial(ac))$. Here $\partial\colon C^0(\tilde{Z}, F(\omega_{\tilde{Z}/C})) \to C^1(\tilde{Z}, F(\omega_{\tilde{Z}/C}))$ is the Rost-Schmid differential, which in degree zero is just the canonical boundary map of a strictly homotopy invariant sheaf. Since $a \in F(\tilde{Z})$ we get $\partial(ac) = a \partial(c)$ (see Lemma \ref{lemm:boundary-product-formula}(4) in the next subsection). 

The transfer map $\tr_{\tilde{p}}\colon C^1(\tilde{Z}, F(\omega_{\tilde{Z}/C})) \to C^1(C, F)$ factors as\NB{$Z$ isn't smooth so $C^*(Z, M)$ and $\omega_{Z/C}$ are not really defined. But $C^1(Z, F(\Omega_{Z/C}))$ still makes sense. Not sure if we should elaborate on this. At that point we no longer care about the complex, only the transfer on fields twisted by stuff.}
\[ C^1(\tilde{Z}, F(\omega_{\tilde{Z}/C})) \xrightarrow{\tr_1} C^1(Z, F(\omega_{Z/C})) \xrightarrow{\tr_2} C^1(C, F), \]
transfers being compatible with composition.
Since $a$ is pulled back from $Z$, by the projection formula (Corollary \ref{corr:twisted-projection}(3)) we have $\tr_1(a\partial c) = a\tr_1(\partial c)$. It thus suffices to show that $\tr_1(\partial c) = 0$. But $\tr_1 \circ \partial $ is precisely the definition of the boundary map $\partial'$ in the Rost-Schmid complex computing $H^n_Z(U, \ul{K}_n^{MW})$ \cite[Definition 5.11]{A1-alg-top}. In other words, by definition the following diagram commutes
\begin{equation*}
\begin{CD}
                               @. C^0(\tilde{Z}, \ul{GW}(\omega_{\tilde{Z}/U})) @>{\partial}>>  C^1(\tilde{Z}, \ul{GW}(\omega_{\tilde{Z}/U})) \\
    @.                               @A{\wequi}AA                                  @V{\tr_1}VV  \\
\ul{GW}(Z^{(0)}, \omega_{Z/U}) @= C^n_Z(U, \ul{K}_n^{MW}) @>{\partial'}>> C^{n+1}_Z(U, \ul{K}_n^{MW}) @= \ul{W}(Z^{(1)}, \omega_{Z/U}).
\end{CD}
\end{equation*}
Since actually $c = c(f_\bullet) \in H^n_Z(U, \ul{K}_n^{MW})$ by construction, we must have $\partial'(c) = 0$, which is what we wanted to show.
\end{proof}

\begin{lemma}\label{lemm:framed-correspondences-compose}
Notation as in the proof of Theorem \ref{thm:M-2-delooping}.

Let $\alpha = (Z, U, f_\bullet, g)$ be a framed correspondence from $X \to Y$ and $h\colon X' \to X$ a morphism of smooth schemes. Denote by $\alpha \circ h$ the correspondence $(Z', U \times_X X', f_\bullet \circ h, g \circ h)$. Then for $a \in F(Y)$ we have $h^*(\alpha^*(a)) = (\alpha \circ h)^*(a)$.
\end{lemma}
\begin{proof}
If the result is true for composable morphisms $h_1, h_2$, then it is also true for $h_1 \circ h_2$. We may assume that $g=\id$. Note that essentially by definition, the result holds if $h$ is an open immersion. In particular we may assume that $X'$ and $X$ are connected. As usual we may factor $h$ into a smooth morphism and a regular immersion, at least locally on $X'$ \cite[Tags 069M and 07DB]{stacks}, so we deal with those in turn.

Suppose now that $h$ is smooth. Then $h$ is dominant (being open between connected schemes \cite[Tag 01UA]{stacks}) and thus using unramifiedness of $F$ we may replace $X, X'$ by their generic points. Since $h$ is smooth, the following square is cartesian\NB{i.e. $X' \times_X Z_\red \to Z_\red$ is a smooth morphism to a union of spectra of fields, so the source must be reduced}
\begin{equation*}
\begin{CD}
Z'_\red @>>> Z_\red \\
@VVV          @VVV  \\
X'      @>h>> X.
\end{CD}
\end{equation*}
The result now follows by naturality of $c(f_\bullet)$ (i.e. $h^*c(f_\bullet) = c(f_\bullet \circ h)$) and the smooth base change formula from Lemma \ref{lemm:smooth-base-change}.

Now suppose that $h$ is a regular immersion. We may replace $X'$ by its generic point $x$ and so in particular factor $h = i \circ h'$, where $i$ is a regular immersion of codimension $1$ and $h'$ is still a regular immersion (e.g. use \cite[Tag 00NQ]{stacks}). Thus we may assume that $h$ is of codimension $1$. We may further replace $X$ by its localization $C$ in $x$. From now on we use the notation from the proof of Lemma \ref{lemm:framed-correspondence-integral}.

For any essentially smooth $1$-dimensional scheme $W$ with closed point $x$ and generic point $\eta$, given $a \in \scr O(\eta)^\times$ we can consider the operation $s^{[a]} = \partial([a] \bullet)\colon F(\eta) \to F(x, \omega_{x/X})$. We will need to use some technical properties of this operation established in the next section. In particular by Lemma \ref{lemm:boundary-product-formula}, the map $s^{[a]}$ is $\ul{GW}$-linear, and hence can be twisted.

Now we get back to the proof. Recall that $h$ is a regular immersion (of codimension $1$). We know that $\alpha^*(a) = \tr_{\tilde{p}}(ac) \in F(C)$. Then by Corollary \ref{corr:pullback-formula} we have $h^*(\alpha^*(a)) = \partial^t_x([t] \tr_{\tilde{p}}(ac))$, where $t$ denotes a chosen uniformizer of $C$ in $x$. In other words \[h^*(\alpha^*(a)) = s^{[t]} \tr_{\tilde{p}}(ac) \in F(x, \omega_{x/X}) \wequi F(x),\] where $\omega_{x/X} \wequi \scr O_x$ via $t$. By Proposition \ref{prop:transfer-specialization}, we get $h^*(\alpha^*(a)) = \tr_{\tilde{p}}(s^{[t]} a c)$. As before we factor $\tr_{\tilde{p}}$ on $C^1$ as $\tr_2 \circ \tr_1$ and use the projection formula (Corollary \ref{corr:twisted-projection}(2)) and Lemma \ref{lemm:boundary-product-formula}(2) to obtain
\[ h^*(\alpha^*(a)) = \tr_2(a \tr_1(\partial ([t] c))), \]
and again we have $\tr_1(\partial ([t] c)) = \partial'([t] c)$, where $\partial'$ denotes the boundary map on $C^n_Z(U, \ul{K}_n^{MW})$. Since $t$ also cuts out $U_x \subset U$, it follows from Remark \ref{rmk:closed-pullback-correct} 
 that $\tr_1(\partial ([t] c)) = h^*(c) = c(f_\bullet \circ h)$. We have thus found
\[ h^*(\alpha^*(a)) = \tr_2(a c(f_\bullet \circ h)), \]
which is precisely the definition of $(\alpha \circ h)^*(a)$.
\end{proof}

\subsection{More about the $M_{-2}$}
In the proof of Theorem \ref{thm:M-2-delooping} we made use of some technical facts regarding specialization and transfer maps on sheaves of the form $M_{-2}$. We collect them here. We are not trying to be exhaustive; instead we only record the facts that we really use. Throughout $M$ denotes a strictly homotopy invariant sheaf.

\subsubsection{Boundary formulas}
Recall first that for any essentially smooth local 1-dimensional scheme $X$ with generic point $\eta$ and closed point $x$, there is the canonical \emph{boundary map} $\partial\colon M(\eta) \to M_{-1}(x, \omega_{x/X})$ \cite[Lemma 5.10]{A1-alg-top}. Moreover $M_{-1}$ is automatically a sheaf of $\ul{GW}$-modules \cite[Lemma 3.49]{A1-alg-top}. Furthermore we have a bilinear pairing $\ul{K}_1^{MW} \otimes M_{-1} \to M$. Applying this to $M_{-1}$ we obtain a pairing $\ul{K}_1^{MW} \otimes M_{-2} \to M_{-1}$ which is in fact $\ul{GW}$-linear \cite[Lemma 3.49]{A1-alg-top}. Finally by contracting the $\ul{GW}$-action $\ul{K}_0^{MW} \otimes M_{-1} \to M_{-1}$ we obtain $\ul{K}_{-1}^{MW} \otimes M_{-1} \to M_{-2}$. The following result shows that all of these actions behave sensibly with respect to the boundary maps.

$ $

\begin{lemma} \label{lemm:boundary-product-formula} ${}$ 
\begin{enumerate}
\item For $a \in \ul{K}_1^{MW}(X)$ and $m \in M_{-1}(\eta)$ we have $\partial(a m) = \epsilon \bar{a} \partial(m)$, where $\bar{a} \in K_1^{MW}(x)$ denotes the restriction of $a$ from $X$ to $x$.
\item For $a \in \ul{K}_1^{MW}(\eta)$ and $m \in M_{-1}(X)$ we have $\partial(a m) = \partial(a) \bar{m}$, where $\bar{m} \in M_{-1}(x)$ denotes the restriction of $m$ from $X$ to $x$.
\item For $b \in \ul{GW}(X)$ and $m \in M_{-1}(\eta)$ we have $\partial(b m) = \bar{b} \partial(m)$, where $\bar{b} \in GW(x)$ denotes the restriction of $b$ from $X$ to $x$.
\item For $b \in \ul{GW}(\eta)$ and $m \in M_{-1}(X)$ we have $\partial(b m) = \partial(b) \bar{m}$.
\end{enumerate}
\end{lemma}
\begin{proof}
We prove (1) and (2) at the same time, by explaining that these results are essentially formal. At the end we explain how to adapt the argument to (3) and (4). (We note also that (1) and (3) were proved in \cite[Lemma 5.10]{A1-alg-top}, but our proof would not be simplified by only establishing (2) and (4).) 

We have a morphism of sheaves $\ul{K}_1^{MW} \otimes M_{-1} \to M$. This induces a morphism \[\ul{K}_1^{MW}(X) \otimes H^1_x(X, M_{-1}) \to H^1_x(X, M),\] and similarly $H^1_x(X, \ul{K}_1^{MW}) \otimes M_{-1}(X) \to H^1_x(X, M)$. Write $\partial'\colon M(\eta) \to H^1_x(X, M)$ for the boundary map in the long exact sequence with support. Then in case (1) we have $\partial'(am) = \bar{a}\partial'(m)$ and in case (2) we have $\partial'(am) = \partial'(a)\bar{m}$, where all the actions are the ones just constructed \cite[II.10.2]{iversen2012cohomology}. This is the formal part of the argument. There are moreover canonical isomorphisms $H^1_x(X, M) \wequi M_{-1}(x, \omega_{x/X})$ and so on, and through this isomorphism $\partial'$ and $\partial$ correspond. However the map $\theta\colon H^1_x(X, M_{-1}) \to M_{-2}(x)$ is \emph{not} a morphism of $\ul{K}_1^{MW}(X)$-modules; instead tracing through the definitions on finds that a switch on $\Gm \wedge \Gm$ is involved, and hence $\theta(am) = \epsilon a \theta(m)$. This explains the factor of $\epsilon$ in (1). For (2), no switch is involved, and no $\epsilon$ appears.

We can prove (3) and (4) similarly, by using instead the morphism of sheaves $\ul{GW} \otimes M_{-1} \to M_{-1}$.
\end{proof}

If $\pi$ is a uniformizer for $X$, then $\pi$ defines a trivialization of $\omega_{x/X}$, and hence $\partial\colon M(\eta) \to M_{-1}(x, \omega_{x/X})$ becomes a map $\partial^\pi\colon  M(\eta) \to M_{-1}(x)$.
\begin{corollary} \label{corr:pullback-formula}
Let $m \in M_{-1}(X)$ and write $i\colon x \to X$ for the closed inclusion. Pick a uniformizer $\pi$ for $X$. Then $i^*(m) = \partial^\pi ([\pi] m)$.
\end{corollary}
\begin{proof}
By Lemma \ref{lemm:boundary-product-formula} we get $\partial^\pi([\pi]m) = \partial^\pi([\pi]) \bar{m} = \partial^\pi([\pi]) i^*(m)$. It remains to observe that $\partial^\pi([\pi]) = 1$ \cite[Theorem 3.15]{A1-alg-top}.
\end{proof}

\subsubsection{Projection formulas}
\label{subsec:projection-formulas}
Now we come to the transfers. Recall that given a finite monogeneous extension $L/K$ with chosen generator $x \in L$, the transfer $\tr\colon M_{-1}(L, \omega_{L/K}) \to M_{-1}(K)$ can be defined in two ways. The generator $x$ provides an embedding $Spec(L) \hookrightarrow \A^1_K \subset \P^1_K$, and we can consider the composition $\P^1_K \to \P^1_K / \P^1_K \setminus Spec(L) \wequi Th(\omega_{L/K})$ of the collapse map and the purity equivalence. 

There is an equivalent algebraic definition: given $m \in M_{-1}(L, \omega_{L/K})$ one has to find $m' \in M(K(T))$ such that $\partial_x(m') = m$ and for $x \ne y \in (\A^1_K)^{(1)}$ we have $\partial_y(m') = 0$. Then $\tr(m) = \partial_\infty(m')$. Here $\partial_\infty$ corresponds to the point at infinity of $\P^1$, and the uniformizer $-1/T$. This result is explained in \cite[Section 4.2, p.99]{A1-alg-top}. It turns out that this construction is independent of the choice of $x$ \cite[Section 5.1]{A1-alg-top}. Note also that $x$ canonically trivializes $\omega_{L/K}$, so we have an untwisted transfer $\tau_x\colon M_{-1}(L) \to M_{-1}(K)$.

\begin{lemma}[untwisted projection formulas] \label{lemm:untwisted-projection}
Given $L = K(x)/K$, the transfer $\tau_x\colon M_{-1}(L) \to M_{-1}(K)$ satisfies the following.
\begin{enumerate}
\item For $a \in K_1^{MW}(K)$ and $m \in M_{-2}(L)$ we have $\tau_x(am) = a \tau_x(m)$.
\item For $b \in GW(L)$ and $m \in M_{-1}(K)$ we have $\tau_x(bm) = \tau_x(b) m$.
\end{enumerate}
Furthermore the transfer $\tau_x\colon M_{-2}(L) \to M_{-2}(K)$ satisfies the following.
\begin{enumerate}
\setcounter{enumi}{2}
\item For $b \in GW(K)$ and $m \in M_{-2}(L)$ we have $\tau_x(bm) = b\tau_x(m)$.
\end{enumerate}
\end{lemma}
\begin{proof}
(1) Pick $m' \in M_{-1}(K(T))$ with $\partial_x(m') = m$ and $\partial_y(m') = 0$ else. Then $\partial_x(a \epsilon m') = a \partial_x(m') = a m$, and $\partial_y(a \epsilon m') = a \partial_y(m') = 0$ else, by Lemma \ref{lemm:boundary-product-formula}(1). Hence $\tau_x(am) = \partial_\infty(a \epsilon m') = a \partial_\infty(m') = a \tau_x(m)$.

(2) Pick $b' \in K_1^{MW}(K(T))$ with $\partial_x(b') = b$ and $\partial_y(b') = 0$ else. Conclude as for (1), appealing to Lemma \ref{lemm:boundary-product-formula}(2).

(3) Pick $m' \in M_{-1}(K(T))$ as in (1). Note that all the boundary maps are $GW(K)$-linear, by Lemma \ref{lemm:boundary-product-formula}(3). Hence $\partial_x(bm') = b\partial_x(m') = bm$ and $\partial_y(bm') = b\partial_y(m') = 0$ for $y \ne x$. Hence $\tau_x(bm) = \partial_\infty(bm') = b\partial_\infty(m') = b \tau_x(m)$.
\end{proof}

Formula (3) above implies that the transfer can be twisted by an arbitrary line bundle: given a line bundle $\scr L$ on $Spec(K)$, there is a transfer $\tau_x\colon M_{-2}(L, \scr L|_L) \to M_{-2}(K, \scr L)$ and similarly for the twisted transfer we have $\tr\colon M_{-2}(L, \omega_{L/K} \otimes \scr L|_L) \to M_{-2}(K, \scr L)$.
\begin{corollary}[twisted projection formulas] \label{corr:twisted-projection}
Let $L/K$ be a finite extension. Let $\scr L_1, \scr L_2$ be line bundles on $Spec(K)$.
Then the transfer $\tr\colon M_{-2}(L, \omega_{L/K} \otimes \scr L_1|_L \otimes \scr L_2|_L) \to M_{-2}(K, \scr L_1 \otimes \scr L_2)$ satisfies the following.
\begin{enumerate}
\item For $a \in K_1^{MW}(K, \scr L_1)$ and $m \in M_{-2}(L, \omega_{L/K} \otimes \scr L_2|_L)$ we have $\tr(am) = a \tr(m)$.
\item For $b \in GW(L, \omega_{L/K} \otimes \scr L_1|_L)$ and $m \in M_{-1}(K, \scr L_2|_L)$ we have $\tr(bm) = \tr(b) m$.
\item For $b \in GW(K, \scr L_1)$ and $m \in M_{-2}(L, \omega_{L/K} \otimes \scr L_2|_L)$ we have $\tr(bm) = b \tr(m)$.
\end{enumerate}
\end{corollary}
\begin{proof}
We may trivialize $\scr L_1$ and $\scr L_2$, and hence ignore them. We can furthermore pass to the symmetrical absolute transfer $\tr\colon M_{-1}(L, \omega_L) \to M_{-1}(K, \omega_K)$ in each of the above statements. For example to prove (1), it suffices to prove: for $a \in K_1^{MW}(K)$ and $m \in M_{-2}(L, \omega_L)$ we have $\tr(am) = \tr(a) m \in M_{-1}(K, \omega_K)$. The advantage is that if now $L/L_1/K$ is an intermediate extension, in order to prove the statement for $\tr_{L/K}$ it suffices to prove it for $\tr_{L/L_1}$ and $\tr_{L_1/K}$. This way we reduce to monogeneous extensions, i.e. Lemma \ref{lemm:untwisted-projection}.
\end{proof}

\subsubsection{Smooth base change formula}
Recall that if $p\colon X \to Y$ is any finite flat (i.e. componentwise dominant) morphism of essentially smooth schemes, then the transfer on fields $\tr = \tr_p\colon M_{-2}(k(X), \omega_{X/Y}) \to M_{-2}(k(Y))$ induces in fact also $\tr_p\colon M_{-2}(X, \omega_{X/Y}) \to M_{-2}(Y)$ \cite[Corollary 5.30]{A1-alg-top}.

\begin{lemma} \label{lemm:smooth-base-change}
Consider a cartesian square
\begin{equation*}
\begin{CD}
X' @>f'>>  X   \\
@Vp'VV   @VpVV \\
Y' @>f>>   Y
\end{CD}
\end{equation*}
of essentially smooth $k$-schemes, with $p$ finite flat and $f$ smooth. Then \[f^* \tr_p = \tr_{p'} f'^*\colon M_{-2}(X, \omega_{X/Y}) \to M_{-2}(Y').\]
\end{lemma}
\begin{proof}
This is clear if $f$ is an open immersion. By unramifiedness, we may thus assume that $Y'$ is the spectrum of a field. Since $f$ is smooth, the image of $Y'$ in $Y$ is a generic point. We may replace $Y$ by this generic point and hence assume that $Y$ is also the spectrum of a field. Now $X$, being finite over $Y$ and essentially smooth, is a disjoint union of finitely many spectra of fields, and similarly for $X'$. Since both sides are compatible with composition, we can assume that $X$ is monogeneous over $Y$. We then obtain an embedding $X \hookrightarrow \P^1_Y$ and similarly for $X', Y'$, and in fact a morphism of smooth closed pairs (see \cite[Section 3.5]{hoyois-equivariant}) $(\P^1_{Y'}, X') \to (\P^1_{Y}, X)$. Naturality of the purity equivalence in smooth closed pairs implies that the right hand square in the following diagram commutes
\begin{equation*}
\begin{CD}
\P^1_{Y'} @>>> \P^1_{Y'} / \P^1_{Y'} \setminus X' @= Th(N_{X'/\P^1_{Y'}}) \\
@VVV              @VVV                                  @VVV              \\
\P^1_{Y}   @>>> \P^1_{Y} / \P^1_{Y} \setminus X @= Th(N_{X/\P^1_{Y}}).
\end{CD}
\end{equation*}
The left hand square commutes trivially. The horizontal composites induce the transfers, and the vertical maps induce $f^*$ and $f'^*$. This proves the result.
\end{proof}

\subsubsection{Transfer and twisted specialization}
Let $X$ be an essentially smooth scheme and $a \in K_1^{MW}(X^{(0)})$. Consider the operation \[s^a := \partial(a \bullet)\colon C^0(X, M_{-2}) \to C^1(X, M_{-1}).\] By Lemma \ref{lemm:boundary-product-formula}(3), $s^a$ is $\ul{GW}(X)$-linear and thus can be twisted: for any line bundle $\scr L$ on $X$, we obtain $s^a\colon C^0(X, M_{-2}(\scr L)) \to C^1(X, M_{-1}(\scr L))$.

\begin{proposition} \label{prop:transfer-specialization}
Let $p\colon X' \to X$ be a finite flat morphism of essentially smooth schemes and ${a \in K_1^{MW}(X^{(0)})}$. The diagram
\begin{equation*}
\begin{CD}
C^0(X', M_{-2}(\omega_{X'/X})) @>s^a>> C^1(X', M_{-1}(\omega_{X'/X})) \\
@V\tr_pVV                            @V\tr_pVV                \\
C^0(X, M_{-2}) @>s^a>> C^1(X, M_{-1}) \\
\end{CD}
\end{equation*}
commutes in the following cases:
\begin{enumerate}
\item $M = M'_{-1}$;
\item $X$ is the spectrum of a dvr and $X'/X$ is monogeneous;
\item $char(k) = 0$ and $X$ of dimension $1$.
\end{enumerate}
\end{proposition}
\begin{proof}
(1) Let $m \in C^0(X', M_{-2}(\omega_{X'/X})) = C^0(X', M'_{-3}(\omega_{X'/X}))$. Then $am \in C^0(X', M'_{-2}(\omega_{X'/X}))$ and we get \[ \tr s^a(m) = \tr \partial(am) = \partial \tr(am) = \partial(a \tr m) = s^a(\tr m), \] where we have used \cite[Corollary 5.30]{A1-alg-top} to commute $\tr$ and $\partial$, and Corollary \ref{corr:twisted-projection}(1) to commute $\tr$ and $a$.

(2) Exactly the same argument works, using \cite[Theorem 5.19]{A1-alg-top} instead of \cite[Corollary 5.30]{A1-alg-top}.

(3) All our operations commute with étale base change. It follows that we may replace $X$ by its henselization in a closed point. In this case $X' \to X$ is a composite of two monogeneous extensions \cite[Remark 5.28]{A1-alg-top}, and so we have reduced to (2).
\end{proof}

\begin{remark}
The operation $s^a$ also makes sense on $M_{-1}$, as $s^a\colon C^0(X, M_{-1}) \to C^1(X, M)$, and one may show that it is still $\ul{GW}$-linear. Moreover Proposition \ref{prop:transfer-specialization}(2) remains valid in this more general setting, but the proof is much more difficult.
\end{remark}

\section{$\Gm$-stabilization of $\ul{\pi}_0$}
\label{sec:Gm-stab}

\subsection{Generalities}
Throughout we fix a perfect field $k$.  
Let $n \ge 0$ and write $\SHS(k)(n)$ for the localizing subcategory of $\SHS(k)$ generated by $\Gmp{n} \wedge \SHS(k)$. We have the canonical inclusion $i_n\colon \SHS(k)(n) \hookrightarrow \SHS(k)$ with right adjoint $r_n$. We denote $f_n = i_n \circ r_n$.

We have the functor $\sigma_n\colon \SHS(k) \to \SHS(k)(n), E \mapsto E \wedge \Gmp{n}$ with right adjoint \[\omega_n\colon \SHS(k)(n) \to \SHS(k), \quad E \mapsto \Omega_{\Gm}^n(i_n(E)).\] 

We also have the  adjunction 
\[ \sigma^\infty\colon \SHS(k) \adj \SH(k)^\eff\colon \omega^\infty, \]
induced by stabilization with respect to $\Gm$.
 The functor $\sigma^\infty\colon \SHS(k) \to \SH(k)^\eff$ factors through $\sigma_n$ as 
 \[\sigma^{\infty-n}\colon \SHS(k)(n) \to \SH(k)^\eff, \quad E \mapsto \sigma^\infty(i_n(E)) \wedge \Gmp{-n},\] and $\sigma^{\infty-n}$ has a further right adjoint $\omega^{\infty-n}\colon \SH(k)^\eff \to \SHS(k)(n)$, factoring $\omega^\infty$ through $\omega_n$. We illustrate these factorizations in the pair of adjoint commutative diagrams below.\footnote{In order to see that the right adjoint of $\wedge \Gm\colon \SHS(k)(n) \to \SHS(k)(n+1)$ is really given by $\Omega_\Gm$, it suffices to prove that $\Omega_\Gm(\SHS(k)(n+1)) \subset \SHS(k)(n)$, which follows from \cite[Theorem 7.4.2]{levine2008homotopy}, at least over infinite fields. We will not actually use this fact.}
\begin{equation*}
\begin{tikzcd}
\dots  \ar[rddd, bend left=70, "\sigma^{\infty-n}"] &&&& \dots \ar[d, "\Omega_\Gm"] \ar[ddd, bend right=100, "\omega_n" left] \\
\ar[u, "\wedge \Gm" right]
\SHS(k)(2) \ar[rdd, bend left, "\sigma^{\infty-2}"] &\hspace{0.5in} \ar[rr,leftrightsquigarrow, shift right=0.3in, "\text{pass to}", "\text{adjoints}" below]&&\hspace{0.1in}& \SHS(k)(2) \ar[d, "\Omega_\Gm"] \ar[dd, bend right=60, "\omega_2" left] \\
\ar[u, "\wedge \Gm" right]
\SHS(k)(1) \ar[rd, "\sigma^{\infty-1}"]     &&&& \SHS(k)(1) \ar[d, "\Omega_\Gm" right, "\omega_1" left] \\
\ar[uuu, bend left=100, "\sigma_n"]
\ar[uu, bend left=60, "\sigma_2"]
\ar[u, "\wedge \Gm" right, "\sigma_1" left]       
\SHS(k) \ar[r, "\sigma^\infty"] & \SH(k)^\eff &&& \SHS(k) &
                                  \ar[l, "\omega^\infty" swap]
                                  \ar[lu, "\omega^{\infty-1}" swap]
                                  \ar[luu, "\omega^{\infty-2}" swap, bend right]
                                  \ar[luuu, "\omega^{\infty-n}" swap, bend right=70]
                                  \SH(k)^\eff
\end{tikzcd}
\end{equation*}

Note that we treat $\SHS(k)$ and similar categories as $\infty$-categories; thus in (1) below ``limits and colimits'' means ``homotopy limits and colimits''. Since these are stable categories and stable functors, preservation of colimits or limits is equivalent to preservation of sums or products.

\begin{lemma} \label{lemm:SHSn-basics} ${}$ 
\begin{enumerate}
\item The following functors are conservative and preserve all limits and colimits: $\omega_n$, $\omega^\infty$, $\omega^{\infty - n}$.
\item The composite functor $i_n \omega^{\infty - n}\colon \SH(k)^\eff \to \SHS(k)$ is equivalent to the functor ${E \mapsto \omega^\infty(E \wedge \Gmp{n})}$.
\end{enumerate}
\end{lemma}
\begin{proof}
(1) All the functors are stable and have left adjoints preserving compact generating families.  This implies the claim.

(2) The composite $\SHS(k)(n) \xrightarrow{i_n} \SHS(k) \xrightarrow{\sigma^\infty(\ph) \wedge \Gmp{-n}} \SH(k)^\eff$ is $\sigma^{\infty-n}$. By adjunction, it follows that $\omega^{\infty-n}$ is equivalent to $r_n \omega^\infty(\ph \wedge \Gmp{n})$. It is thus enough to prove that \[\omega^\infty(\SH(k)^\eff(n)) \subset \SHS(k)(n) \subset \SHS(k),\] where $\SH(k)^\eff(n)$ is the localizing subcategory of $\SH(k)^\eff$ generated by $\Gmp{n} \wedge \SH(k)^\eff$. This follows from the fact that the effectivity tower commutes with $\omega^\infty$, as proved by Levine \cite[Theorems 7.1.1 and
9.0.3]{levine2008homotopy} (his Theorem 7.1.1 is stated only for infinite perfect fields, but it applies more generally to any $E$ satisfying ``axiom A3'', e.g. any spectrum of the form $\omega^\infty(\ph)$; c.f. Remark 9.0.4 of \emph{loc.cit.}).
\end{proof}

The category $\SHS(k)$ carries a canonical $t$-structure, with the non-negative part generated under colimits and extensions by $\Sigma^\infty_{S^1} X_+$, for $X \in \Sm_k$. This is the homotopy $t$-structure: the proof of \cite[Theorem 2.3]{hoyois-algebraic-cobordism} for the case of $\SH(k)$ applies essentially unchanged also to $\SHS(k)$. We denote the homotopy sheaves by $\ul{\pi}_i(E)$. We similarly put a $t$-structure on $\SHS(k)(n)$ with non-negative part generated by $\sigma_n(\SHS(k)_{\ge 0})$. Finally, we put a $t$-structure on $\SH(k)^\eff$, with non-negative part generated by $\sigma^\infty(\SHS(k)_{\ge 0})$. This is the effective homotopy $t$-structure \cite[Proposition 4(2)]{tom2017slices}. 

\begin{lemma} \label{lemm:SHSn-t} ${}$ 
\begin{enumerate}
\item The following functors are right-$t$-exact: $\sigma_n, i_n, \sigma^\infty, \sigma^{\infty -n}$.
\item The following functors are $t$-exact: $\omega_n, r_n, \omega^\infty, \omega^{\infty - n}$.
\item Let $E \in \SHS(k)(n)$. The following are equivalent: (i) $E \in \SHS(k)(n)_{\ge 0}$, (ii) $\ul{\pi}_i(i_n E)_{-n} = 0$ for $i < 0$, (iii) $\ul{\pi}_i(i_n E) = 0$ for $i < 0$. Similarly, the following are equivalent: (i') $E \in \SHS(k)(n)_{\le 0}$, (ii') $\ul{\pi}_i(i_n E)_{-n} = 0$ for $i > 0$.
\end{enumerate}
\end{lemma}
\begin{proof}
Statement (1) is clear by construction. Hence the functors from (2) are left-$t$-exact. Note that if $F$ is a conservative $t$-exact functor and $G$ is any stable functor such that $FG$ is (right) $t$-exact, then $G$ is (right) $t$-exact\NB{Let $X \ge 0$. We need to show that $GX \ge 0$, i.e. $(GX)_{<0} = 0$. By conservativity and t-exactness of $F$, this is the same as $(FGX)_{<0} = 0$}. For (2), it is thus enough to prove that the following functors are (right) $t$-exact: $\omega_n, \omega^\infty, \omega_n r_n$. The statement about $\omega^\infty$ follows from the description of the $t$-structures in terms of homotopy sheaves. We have $\omega_n r_n = \Omega_\Gm^n$ 
\NB{because left adjoint of $\omega_n r_n$ is $i_n \sigma_n  = \Sigma^n_{\Gm}$}
 and $\ul{\pi}_i(\Omega_\Gm^n E) = \ul{\pi}_i(E)_{-n}$ \cite[Lemma 4.3.11]{morel-trieste}, so $\omega_n r_n$ is also $t$-exact. Finally note that $\omega_n = \Omega_\Gm^n \circ i_n$ is a composite of right-$t$-exact functors, so is right-$t$-exact.

It remains to prove (3). Since $\ul{\pi}_i(\omega_n E) = \ul{\pi}_i(i_n E)_{-n}$ and $\omega_n$ is conservative and $t$-exact, (i) is equivalent to (ii) and (i') is equivalent to (ii'). Since $i_n$ is right-$t$-exact, (i) implies (iii).  Clearly (iii) implies (ii) as well. This concludes the proof.
\end{proof}

\begin{corollary} \label{corr:SHSn-heart} ${}$ 
\begin{enumerate}
\item The functor $\omega_n^\heart\colon \SHS(k)(n)^\heart \to \SHS(k)^\heart \wequi \HI(k)$ is conservative and preserves limits and colimits. In particular it is
  monadic.
\item The functor $i_n^\heart = \ul{\pi}_0 \circ i_n\colon \SHS(k)(n)^\heart \to \SHS(k)^\heart \wequi \HI(k)$ is fully faithful. Its essential image consists of those sheaves $F \in \HI(k)$ such that the canonical map $\ul{\pi}_0(i_n r_n F) \to F$ is an isomorphism.
\end{enumerate}
\end{corollary}
\begin{proof}
(1) If $\scr C$ is any presentable stable $\infty$-category with an accessible $t$-structure and heart $i\colon \scr C^\heart \hookrightarrow \scr C$ and $D\colon I \to \scr C^\heart$ is a diagram, then $\colim_I D \wequi \pi_0(\colim_I iD)$, and similarly for limits. Consequently if $F$ is any $t$-exact functor preserving colimits (respectively limits), then $F^\heart$ also preserves colimits (respectively limits). Hence $\omega_n^\heart$ preserves limits and colimits, and is clearly still conservative (since the same holds for $\omega_n$, by Lemma \ref{lemm:SHSn-basics}(1), and $\omega_n$ is $t$-exact by Lemma \ref{lemm:SHSn-t}(2)). It is monadic since all the categories involved are presentable.

(2) Let $F \in \SHS(k)(n)^\heart$. Then $i_n(F) \in \SHS(k)_{\ge 0}$ and moreover $\ul{\pi}_i(i_n F)_{-n} = 0$ for $i > 0$, by Lemma \ref{lemm:SHSn-t}(ii'). It follows that for $E \in \SHS(k)(n)$ we have $[i_n E, i_n F] = [i_n E, \ul{\pi}_0(i_n F)]$\NB{E.g. $[i_n E[*], (i_n F)_{>0}] = [E[*], r_n((i_n F)_{>0})]$, and $r_n((i_n F)_{>0}) = 0$ since $\omega_n r_n((i_n F)_{>0}) = \Omega_\Gm^n (i_n F)_{>0} = 0$}. Suppose that
further $E \in \SHS(k)(n)^\heart$. Then $[i_n E, \ul{\pi}_0(i_n F)] = [(i_n E)_{\le 0}, \ul{\pi}_0(i_n F)]$, since $\ul{\pi}_0(i_n F) \in \SHS(k)_{\le 0}$.  Since $(i_n E)_{\le 0} = i_n^\heart(E)$ ($i_n$ being right-$t$-exact) and $\ul{\pi}_0(i_n F) = i_n^\heart(F)$, this is the fully faithfulness statement we wanted.

It remains to describe the essential image of $i_n^\heart$. Let $F \in \HI(k) \wequi \SHS(k)^\heart$. Since $r_n$ is $t$-exact, $\ul{\pi}_0(i_n r_n F) = i_n^\heart(r_n^\heart F)$ is in the essential image of $i_n^\heart$. Thus if $F \wequi \ul{\pi}_0(i_n r_n F)$, then $F$ is in the essential image of $i_n^\heart$. Conversely, suppose that $F$ is in the essential image of $i_n^\heart$, say $F \wequi i_n^\heart(E)$ for some $E \in \SHS(k)(n)^\heart$. Then $\ul{\pi}_0(i_n r_n F) = i_n^\heart(r_n^\heart F) = i_n^\heart r_n^\heart i_n^\heart E$. Note that if $\alpha, \beta$ are two right-$t$-exact functors, then $\alpha^\heart \beta^\heart \wequi (\alpha \beta)^\heart$. It follows that $r_n^\heart i_n^\heart \wequi (r_n i_n)^\heart \wequi \id^\heart = \id$, and hence $\ul{\pi}_0(i_n r_n F) \wequi i_n^\heart E \wequi F$. This concludes the proof.
\end{proof}

\subsection{Virtual transfers}
Given $M \in \HI(k)$ and a field $K$, recall the group
\[ C^0(K, M, 1) = \bigoplus_{x \in (\A^1_K)^{(1)} \setminus \{0, 1\}} H^1_x(\A^1_K, M). \]
The edge map in the spectral sequence for $M^{(1)}(K, \bullet)$ induces a map $C^0(K, M, 1) \to \ul{\pi}_0 M^{(1)}(K)$. It follows from \cite[Proposition 3.2(2)]{levine-slice} that this map is a surjection. We will construct a $\ul{GW}$-module structure and transfers on $C^0(K, M, 1)$. If $M \in \SHS(k)(1)^\heart$ then $C^0(K, M, 1) \to M(K)$ is surjective, and consequently there is at most one compatible $\ul{GW}$-module structure and transfers on $M(K)$. For this reason we call the structure on $C^0(K,M,1)$ \emph{virtual transfers}.

\begin{definition}
Let $M \in \HI(k)$.
\begin{enumerate}
\item Let $K$ be a field. We give $C^0(K,M,1)$ the structure of a module over $GW(K)$ coming from the isomorphisms $H^1_x(\A^1_K, M) \wequi M_{-1}(x)$ and the $\ul{GW}$-module structure on $M_{-1}$.
\item Let $K(x)/K$ be a monogeneous extension. We define a map $\tau_x\colon C^0(K(x), M, 1) \to C^0(K, M, 1)$ as follows. Suppose given a closed point $y \in \A^1_{K(x)}$ with image $z \in \A^1_K$. We need to define $(\tau_x)^y_z\colon M_{-1}(y) \to M_{-1}(z)$. But $k(y)/k(z)$ is generated by $x$, so we obtain the map $(\tau_x)^y_z$ from the construction $M_{\wh{-1}} \in \HI^{\Atr}$ of Example \ref{ex:M-1-transfers}.
\end{enumerate}
\end{definition}

 Note that the construction $C^*(X, M, q)$ is obviously functorial in $M$, and hence so is $C^0(K, M, 1)$. Then the above definitions are functorial in $M$, in the following sense.

\begin{lemma} \label{lemm:virtual-transfers-functorial}
Let $\alpha\colon M \to N \in \HI(k)$. Then the following hold.
\begin{enumerate}
\item Let $K/k$ be a field. The following diagram commutes
\begin{equation*}
\begin{CD}
C^0(K, M, 1) @>\alpha>> C^0(K, N, 1) \\
@V{\epsilon}VV              @V{\epsilon}VV  \\
M(K)         @>\alpha>> N(K).
\end{CD}
\end{equation*}
\item Let $K/k$ be a field. The morphism $\alpha\colon C^0(K, M, 1) \to C^0(K, N, 1)$ is a morphism of $GW(K)$-modules. In fact, for any $x \in (\A^1_K)^{(1)}$, the morphism $\alpha\colon M_{-1}(x) \to N_{-1}(x)$ is a morphism of $GW(x)$-modules.
\item Let $K(x)/K$ be a monogeneous extension. Then the following diagram commutes
\begin{equation*}
\begin{CD}
C^0(K(x), M, 1) @>\alpha>> C^0(K(x), N, 1) \\
@V{\tau_x}VV              @V{\tau_x}VV \\
C^0(K, M, 1) @>\alpha>> C^0(K, N, 1) \\
\end{CD}
\end{equation*}
\end{enumerate}
\end{lemma}
\begin{proof}
(1) follows from the functoriality $M^{(1)}(K, \bullet) \to N^{(1)}(K, \bullet)$ of the homotopy coniveau tower and the associated spectral sequence and the edge map. (2) and (3) follow from the fact that multiplication by $\lra{a}$ and $\tau_x$ are obtained by applying $M$ (respectively $N$) to a morphism (of pro-objects) in $\SHS(k)$.
\end{proof}

Our main observation in this section is that if $M$ is already a homotopy module, then the virtual structures above are indeed the correct ones, in an appropriate sense. In order to prove this we need to understand the map $\epsilon\colon C^0(K, M, 1) \to M(K)$. Here is what we can say for general $M$.

\begin{lemma} \label{lemm:epsilon-computation}
Let $M \in \HI(k)$, $z \in \A^1_K \setminus \{0,1\}$ a closed point and $a \in H^1_z(\A^1_K, M)$. Pick $\tilde{a} \in M(\A^1_K \setminus z)$ with $\partial \tilde{a} = a$. Then $\epsilon(a) = i_1^*(\tilde a) - i_0^*(\tilde a)$.
\end{lemma}
\begin{proof}
Consider the presheaf $\scr F$ of $H\Z$-modules on $\Sm_K$ given by $X \mapsto R\Gamma(X, M)$. Using the injective model structure on presheaves of chain complexes we obtain $F \in \Fun(\Sm_K^\op, Ch)$ modeling $\scr F$ such that for every open immersion $U \to X$ the pullback $F(X) \to F(U)$ is a fibration (of injective fibrant chain complexes). In particular for $Z \subset X$ closed let $F_Z(X) = ker(F(X) \to F(U))$; then $F_Z(X)$ is equivalent to the homotopy fiber of $F(X) \to F(U)$, and so computes $R\Gamma_Z(X, M)$. We can thus use $F$ to build a strict model $F^{(q)}_\bullet$ of the homotopy coniveau tower $M^{(q)}(K, \bullet)$. In other words, for each $q \ge 0$ and $n \in \Delta^\op$ we have the chain complex
\[ F^{(q)}_n = \colim_{Z \subset \Delta^n_K \qgood} F_Z( \Delta^n_K), \]
these chain complexes fit together into a simplicial chain complex $F^{(q)}_\bullet$, and there are canonical maps $F^{(q+1)}_\bullet \to F^{(q)}_\bullet$.

Note that if $n < q$ then 
$\Delta^n_K$ does not afford non-empty $q$-good closed subschemes, so $F^{(q)}_n = 0$. 
Let $a \in H^1_z(\A^1_K, M)$. Then $a$ defines an element of $h^1 F^{(1)}_1$ and so can be represented by an element $a' \in ker(F^{(1),1}_1 \to F^{(1),2}_1)$. By the previous remark, the vertical boundary of $a'$ is zero, and hence $a'$ represents a class in the cohomology of the total complex (i.e. geometric realization of $F^{(1)}_\bullet$), or in other words an element $\delta(a) \in \pi_0(M^{(1)}(K))$; this is an instance of the ``edge map'' in the spectral sequence corresponding to $F^{(1)}_\bullet$. Let $\alpha\colon F^{(1)}_\bullet \to F^{(0)}_\bullet$ be the canonical map; then $\epsilon(a) = |\alpha|(\delta(a))$. Since $\alpha$ is a map of simplicial chain complexes, $\alpha(a') \in F^{(0),1}_1$ is also a cycle in the total complex of $F^{(0)}_\bullet$; let us denote the cohomology class it represents by $\delta(\alpha(a'))$. Then $\epsilon(a) = \delta(\alpha(a'))$.

Since $h^1 F^{(0)}_1 = H^1(\A^1_K, M) = 0$, there exists $\bar{a} \in F_1^{(0),0}$ with horizontal boundary $d_h(\bar a) = \alpha(a')$. Then $\alpha(a')$ is cohomologous to the vertical boundary $d_v(\bar{a})$ in the total complex of $F^{(0)}_\bullet$, and the edge map $h^0 F^{(0)}_0 \to h^0 |F^{(0)}_\bullet| \wequi M(K)$ is the canonical equivalence. Hence $\epsilon(a) = d_v(\bar{a})$. We can restrict $\bar a$ to $U := \A^1_K \setminus z$. By functoriality, $d_h(\bar a|_U) = \alpha(a'|_U)$. But $a'|_U = 0$ (since it is supported on $\emptyset$), so $d_h(\bar a|_U) = 0$ and $\bar{a}|_U$ defines an element of $H^0(U, M)$. By construction, if $\partial\colon H^0(U, M) \to H^1_z(\A^1_K, M)$ denotes the boundary map in the long exact sequence of cohomology with support, then $\partial(\bar{a}|_U) = a$. Note that $d_v(\bar a) = i_1^*(\bar a) - i_0^*(\bar a)$, and this factors through the restriction to $U$. Hence $\epsilon(a) = i_1^*(\tilde a) - i_0^*(\tilde a)$ for some $\tilde a \in H^0(U, M)$ with $\partial \tilde a = a$, namely $\tilde a = \bar a|_U$.

It remains to observe that if $\tilde a \in H^0(U, M)$ with $\partial \tilde{a} = a$, then $\tilde a = \bar a|_U + b|_U$, for some $b \in H^0(\A^1_K, M) \wequi H^0(K, M)$. In particular $i_0^* b = i_1^* b$ and so $i_1^*(\tilde a) - i_0^*(\tilde a) = i_1^*(\bar a) - i_0^*(\bar a)$. This concludes the proof.
\end{proof}

If $M$ is a homotopy module, we can make the above recipe more concrete, generalizing \cite[Proposition 8.2]{levine-slice} to the inseparable case.
\begin{proposition} \label{prop:epsilon-computation}
Let $M \in \HI_0(k)$, $z \in \A^1_K \setminus \{0,1\}$ a closed point and $a \in H^1_z(\A^1_K, M)$. Then \[\epsilon(a) = \tau_z([1-z^{-1}] a),\] where we identify $H^1_z(\A^1_K, M) \wequi M_{-1}(z)$ by trivializing the normal bundle via the minimal polynomial.
\end{proposition}
\begin{proof}
Let $Z \subset \A^1_{K(z)}$ denote the reduced (and hence essentially smooth) preimage of $z$ and $z \in Z$ the canonical lift.  Consider the following diagram
\begin{equation*}
\begin{CD}
H^0(K(z), M(\omega_{K(z)/K})) @<i_s^*<< H^0(\A^1_{K(z)} \setminus z, M(\omega_{\A^1_{K(z)}/\A^1_K})) @>\partial>> H^1_z(\A^1_{K(z)}, M(\omega_{\A^1_{K(z)}/\A^1_K})) \\
@|                                               @V{j^*}VV                                                         @VVV \\
H^0(K(z), M(\omega_{K(z)/K})) @<i_s^*<< H^0(\A^1_{K(z)} \setminus Z, M(\omega_{\A^1_{K(z)}/\A^1_K})) @>\partial>> H^1_Z(\A^1_{K(z)}, M(\omega_{\A^1_{K(z)}/\A^1_K})) \\
@V{\tr}VV                                        @V{\tr}VV                                                         @V{\tr}VV \\
H^0(K, M)                     @<i_s^*<< H^0(\A^1_{K} \setminus z, M)                                 @>\partial>> H^1_z(\A^1_{K}, M).
\end{CD}
\end{equation*}
Here $s \in \{0,1\}$, $\partial$ denotes the boundary map in the long exact sequence of cohomology with support, $j: \A^1_{K(z)} \setminus Z \to \A^1_{K(z)} \setminus z$ is the inclusion, $\tr$ is the absolute transfer and the unlabelled arrow is extension of support. Note that $Spec(K(z))/Spec(K)$ is flat and hence so is $\A^1_{K(z)}/\A^1_K$. It follows that $i_s^* \omega_{\A^1_{K(z)}/\A^1_K} \wequi \omega_{K(z)/K}$, and the pullbacks $i_s^*$ make sense. It also follows that $\omega_{\A^1_{K(z)}/\A^1_K} \wequi p^* \omega_{K(z)/K}$, where $p\colon \A^1_{K(z)} \to K(z)$ is the canonical projection. 

Proposition \ref{prop:transfer-specialization} and Corollary \ref{corr:pullback-formula} imply that the lower left hand square commutes. The lower right hand square commutes since the absolute transfer is defined on the level of the Rost-Schmid complexes. The upper left hand square commutes since $i_s$ factors through $j$, and the upper right hand square commutes essentially by definition. We are given $a \in H^1_z(\A^1_{K}, M)$ and need to determine $\epsilon(a)$. By Lemma \ref{lemm:epsilon-computation}, we need to find $\tilde a \in H^0(\A^1_{K} \setminus z, M)$ and compute $i_1^*(\tilde a) - i_0^*(\tilde a)$. We shall find $\tilde a_1 \in H^0(\A^1_{K(z)} \setminus z, M(\omega_{\A^1_{K(z)}/\A^1_K}))$ with $\tr(j^* \tilde a_1) = \tilde a$. Consequently $\epsilon(a) = \tr(i_1^*(\tilde a_1) - i_0^*(\tilde a_1))$.

The right hand (composite) transfer $\tr\colon H^1_z(\A^1_{K(z)}, M(\omega_{\A^1_{K(z)}/\A^1_K})) \to H^1_z(\A^1_{K}, M)$ is an isomorphism, which can be seen as follows: \[H^1_z(\A^1_{K(z)}, M(\omega_{\A^1_{K(z)}/\A^1_K})) \wequi M_{-1}(z, \omega_{z/\A^1_{K(z)}} \otimes \omega_{\A^1_{K(z)}/\A^1_K}) \wequi M_{-1}(z, \omega_{z/\A^1_K}),\] and $H^1_z(\A^1_K, M) \wequi M_{-1}(z, \omega_{z/\A^1_K})$ as well. Hence we obtain a unique element $a_1 \in H^1_z(\A^1_{K(z)}, M(\omega_{\A^1_{K(z)}/\A^1_K}))$ with $\tr(a_1) = a$. Suppose we trivialize $\omega_{K(z)/K}$ by using the minimal polynomial. This induces a trivialization of $\omega_{\A^1_{K(z)}/\A^1_K}$, and consequently we can drop all the twists in the upper row of the diagram. The induced trivialization of $\omega_{K(z)/K}$ in the top left corner is the one we started with (i.e. by the minimal polynomial), and so the left hand transfer map turns into the geometric transfer $\tau_z$. We have $H^1_z(\A^1_{K(z)}, M) \wequi M_{-1}(z)$ canonically (since $z \in \A^1_{K(z)}$ is a rational point), and the above discussions shows that the isomorphism \[ M_{-1}(z) \wequi H^1_z(\A^1_{K(z)}, M) \wequi H^1_z(\A^1_{K(z)}, M(\omega_{\A^1_{K(z)}/\A^1_K})) \stackrel{\tr}{\wequi} H^1_z(\A^1_{K}, M) \] is the one  induced by the minimal polynomial.

Everything that remains to be done now only involves the top row of the diagram, without any twists: we need to produce $\tilde a_1 \in H^0(\A^1_{K(z)} \setminus z, M)$ with $\partial \tilde a_1 = a_1$ and show that $i_1^*(\tilde a_1) - i_0^*(\tilde a_1) = [1-z^{-1}]a_1$, where we use that $a_1 \in H^0(\A^1_{K(z)}) \wequi M_{-1}(z)$ canonically, so $[1-z]a_1 \in M(z)$ makes sense. In other words, we have reduced the problem to the situation where $z$ is a rational point. It is not difficult to produce a lift $\tilde a_1$ directly (it is given by $[t-z]a_1$) and verify the claim; but this case was also already dealt with in \cite[Proposition 7.1]{levine-slice}.
\end{proof}

\begin{proposition} \label{prop:virtual-transfers-correct}
Let $M \in \HI_0(k)$ and $K(x)/K$ a monogeneous field extension.
\begin{enumerate}
\item The map $\epsilon\colon C^0(K, M, 1) \to M(K)$ is a morphism of $GW(K)$-modules.
\item Consider the following diagram
\begin{equation*}
\begin{CD}
C^0(K(x), M, 1) @>{\tau_x}>> C^0(K, M, 1) \\
@V{\epsilon}VV            @V{\epsilon}VV \\
M(K(x)) @>{\tau_x}>> M(K),
\end{CD}
\end{equation*}
which does not in general commute. Let $y \in \A^1_{K(x)}$ be a closed point and $a \in H^1_y(\A^1_{K(x)}, M)$. Then $\tau_x(\epsilon(\lra{a_1} a)) = \epsilon(\tau_x(\lra{a_2} a))$ for appropriate $\lra{a_1}, \lra{a_2} \in GW(K(x, y))$ depending only on $x, y$.
\end{enumerate}
\end{proposition}
\begin{proof}
Let $z \in \A^1_K \setminus \{0,1\}$ be a closed point with coordinate $z$ defined over $K(z)$. By Proposition \ref{prop:epsilon-computation}, for $a \in H^1_z(\A^1_K, M)$ we have $\epsilon(a) = \tau_z([z']a)$, where $z' := 1-z^{-1}$.

(1) By the projection formula (Lemma \ref{lemm:untwisted-projection}(3)), we have for $b \in GW(K)$ that $\epsilon(ba) = \tau_z([z'] b a) = b \tau_z([z'] a) = b\epsilon(a)$. This was to be shown.

(2) Let $y \in \A^1_{K(x)} \setminus \{0,1\}$ be closed, with image $z$ in $\A^1_K$, and $a \in H^1_y(\A^1_{K(x)}, M)$. By Proposition \ref{prop:epsilon-computation} we have
\begin{align*}
 \tau_x(\epsilon(a)) &= \tau_{K(x)/K}(\tau_{K(x, y)/K(x)}([y'] a))\text{, and} \\
 \epsilon(\tau_x(a)) &= \tau_{K(z)/K}([z'] \tau_{K(x, y)/K(z)}(a)).
\end{align*}
We note that $[z']|_{K(x,y)} = [y']$.

Now we make use of Morel's \emph{absolute transfers}; see \cite[Section 5.1]{A1-alg-top} and Section \ref{subsec:twisted-transfers}. Briefly, for a homotopy module $M$ and $L/K$ a finite extension of fields, there exists the absolute transfer $\tr_{L/K}\colon M(L, \omega_L) \to M(K, \omega_K)$. These transfers satisfy the following properties:
\begin{enumerate}[(a)]
\item If $L=K(x)/K$ is monogeneous, then $\tr_{L/K}$ is a twist of $\tau_{K(x)/K}$ \cite[Remark 5.6(2)]{A1-alg-top}. In other words there exists a trivialization of $\omega_{L/K}$ (depending on $x$) such that $\tr_{L/K} = \tau_{K(x)/K} \otimes \omega_K$.
\item For $L'/L/K$ holds the following: $\tr_{L'/K} = \tr_{L/K} \circ \tr_{L'/L}$ \cite[Lemma 5.5]{A1-alg-top}.
\end{enumerate}

Let us now choose trivializations of $\omega_K$, $\omega_{K(x)}$, $\omega_{K(z)}$ and $\omega_{K(x,y)}$. Property (a) (together with the projection formula) implies that there exists $a_x \in K(x)^\times$ such that $\tau_{K(x)/K}(\ph) = \tr_{K(x)/K}(\lra{a_x} \ph)$, where on the right hand side we have identified $M(K, \omega_K)$ with $M(K)$ via our trivialization of $\omega_K$, and similarly for $M(K(x), \omega_{K(x)})$. Similarly for the other transfers. The result thus follows from (b) together with the projection formula (Corollary \ref{corr:twisted-projection}(1,3)).
\end{proof}

\subsection{Applications to conservativity of $\Gm$-stabilization}

Recall that by Lemma~\ref{lemm:SHSn-t}(2) the functor $\omega^{\infty-n} \colon \SH(k)^{\eff} \to \SHS(k)(n)$ restricts to the hearts of the homotopy $t$-structures on the corresponding categories. 

\begin{theorem} \label{thm:automatic-transfer-preservation}
Let $k$ be a perfect field. Then for $n \ge 1$ both of the functors
\[ \HI_0(k) \wequi \SH(k)^{\eff\heart} \xrightarrow{\omega^{\infty-n}} \SHS(k)(n)^\heart \xrightarrow{i_n^\heart} \SHS(k)^\heart \wequi \HI(k) \]
are fully faithful.
\end{theorem}
\begin{proof}
Since the second functor is fully faithful by Corollary \ref{corr:SHSn-heart}(2), it suffices to show that the composite is fully faithful. By Lemma \ref{lemm:SHSn-basics}(2) and \ref{lemm:SHSn-t}(2), the composite is also given by \[\SH(k)^{\eff\heart} \xrightarrow{(\ph \wedge \Gmp{n})_{\le 0}} \SH(k)^{\eff\heart} \xrightarrow{\omega^\infty} \SHS(k)^\heart.\] The functor $\SH(k)^{\eff\heart} \xrightarrow{(\ph \wedge \Gmp{n})_{\le 0}} \SH(k)^{\eff\heart}$ is fully faithful (with right adjoint/inverse $F \mapsto F_{-n}$). Denote its essential image by $\SH(k)^{\eff\heart}(n)$. Hence it suffices to show that the forgetful functor $\SH(k)^{\eff\heart}(n) \to \HI(k)$ is fully faithful. \NB{note that $\omega^\infty$ itself is not fully faithful}
In other words, by Corollary~\ref{corr:omega-infty-hat-ff}, given $F, G \in \SH(k)^{\eff\heart}(n)$ with underlying sheaves $F', G' \in \HI(k)$, we need to show that any morphism of sheaves $F' \to G'$ already preserves the transfers and $\ul{GW}$-module structure. It follows from Corollary \ref{corr:SHSn-heart}(2) and \cite[Proposition 3.2(2)]{levine-slice} that $\epsilon\colon C^0(K, F', 1) \to F'(K)$ is surjective, and similarly for $G'$. Hence the result follows from Lemma \ref{lemm:virtual-transfers-functorial} and Proposition \ref{prop:virtual-transfers-correct}.

Let us spell this out in a little bit more detail. The preservation of the $\ul{GW}$-module structure is easier, so we focus on the transfers. Let $\alpha: F' \to G'$ be a morphism of sheaves and $K(x)/K$ a monogeneous extension. Pick $t \in F'(K(x))$. We wish to show that \[ \alpha(\tau_x^{F'}(t)) = \tau_x^{G'}(\alpha(t)). \] By surjectivity of $\epsilon$, it suffices to prove this for $t = \epsilon^{F'}(t')$ for \[ t' \in C^0(K(x), F', 1) = \bigoplus_{y \in (\A^1_{K(x)})^{(1)} \setminus \{0,1\}} H^1_y(\A^1_{K(x)}, F'). \] Since all our maps are homomorphisms we may assume that $t' \in H^1_y(\A^1_{K(x)}, F')$ for some $y$. Now we compute
\begin{align*}
\alpha(\tau_x^{F'}(t)) &= \alpha(\tau_x^{F'}(\epsilon(t'))) \\
  &= \alpha(\tau_x^{F'}(\epsilon^{F'}(\lra{a_1}(\lra{a_1} t')))) \\
  &\stackrel{(i)}{=} \alpha(\epsilon^{F'}(\tau_x^{F'}(\lra{a_2}\lra{a_1} t'))) \\
  &\stackrel{(ii)}{=} \epsilon^{G'}(\alpha(\tau_x^{F'}(\lra{a_2}\lra{a_1} t'))) \\
  &\stackrel{(iii)}{=} \epsilon^{G'}(\tau_x^{G'}(\alpha(\lra{a_2}\lra{a_1} t'))) \\
  &\stackrel{(iv)}{=} \epsilon^{G'}(\tau_x^{G'}(\lra{a_2}\lra{a_1} \alpha( t'))) \\
  &\stackrel{(v)}= \tau_x^{G'}(\epsilon^{G'}(\lra{a_1}\lra{a_1} \alpha( t'))) \\
  &= \tau_x^{G'}(\epsilon^{G'}( \alpha( t'))) \\
  &\stackrel{(vi)}{=} \tau_x^{G'}(\alpha(\epsilon^{F'}(t'))) \\
  &= \tau_x^{G'}(\alpha(t)).
\end{align*}
Here $(i)$ is by Proposition \ref{prop:virtual-transfers-correct}(2), $(ii), (iii)$ and $(iv)$ are by Lemma \ref{lemm:virtual-transfers-functorial}(1), (3) and (2) respectively, $(v)$ undoes $(i)$ and $(vi)$ undoes $(ii)$.
\end{proof}

As explained in the introduction, the following strengthening of Theorem \ref{thm:automatic-transfer-preservation} would have very desirable consequences.
\begin{conjecture} \label{conj:equivalence}
Let $k$ be a perfect field. For $n \ge 1$ the functor 
\[\omega^{\infty - n}\colon \SH(k)^{\eff\heart} \to \SHS(k)(n)^\heart\] is an equivalence.
\end{conjecture}
For any $F \in \HI(k)$ there exists a map $\alpha\colon R \to F$ with $R \in \HI(k)$ a sum of sheaves of the form $\ul{\pi}_0(\Gmp{n} \wedge X_+)$ such that $\alpha_{-n}$ is surjective. Combining this with Corollary \ref{corr:SHSn-heart}, we find that $\SHS(k)(n)^\heart$ is generated under colimits by $\ul{\pi}_0(\Gmp{n} \wedge X_+)$ for $X \in \Sm_k$. Since $\omega^{\infty-n}$ is fully faithful (by Theorem \ref{thm:automatic-transfer-preservation}) and preserves colimits (by Lemma \ref{lemm:SHSn-basics}(1)), we find that Conjecture \ref{conj:equivalence} is true if and only if for every $X \in \Sm_k$ the sheaf $\ul{\pi}_0(\Gmp{n} \wedge X_+)$ is in the essential image of $\omega^{\infty-n}$; in other words if and only if $\ul{\pi}_0(\Gmp{n} \wedge X_+)$ is a homotopy module. We believe that this reformulation may be more amenable to prove.

\bibliographystyle{alphamod}
\bibliography{bibliography}

\newcommand{\etalchar}[1]{$^{#1}$}
\providecommand{\bysame}{\leavevmode\hbox to3em{\hrulefill}\thinspace}
\providecommand{\MR}{\relax\ifhmode\unskip\space\fi MR }
\providecommand{\MRhref}[2]{%
  \href{http://www.ams.org/mathscinet-getitem?mr=#1}{#2}
}
\providecommand{\href}[2]{#2}
\begin{thebibliography}{EHK{\etalchar{+}}19}
\providecommand{\url}[1]{\href{#1}{{\def~{\textasciitilde}\tt #1}}}

\bibitem[AF16]{asok2016comparing}
A.~Asok and J.~Fasel, \emph{Comparing Euler classes}, Q. J. Math. \textbf{67}
  (2016), no.~5, pp.~1--33

\bibitem[AN18]{anan-neshitov2017transfers}
A.~{Ananyevskiy} and A.~{Neshitov}, \emph{{Framed and MW-transfers for homotopy
  modules}}, 2018, \href{http://arxiv.org/abs/1710.07412v2}{{\sf
  arXiv:1710.07412v2}}

\bibitem[Bac17]{tom2017slices}
T.~Bachmann, \emph{The generalized slices of {H}ermitian {$K$}-theory}, J.
  Topol. \textbf{10} (2017), no.~4, pp.~1124--1144

\bibitem[Bac18a]{bachmann-tambara}
T.~Bachmann, \emph{Motivic Tambara Functors}, 2018,
  \href{http://arxiv.org/abs/1807.02981}{{\sf arXiv:1807.02981}}

\bibitem[Bac18b]{bachmann-hurewicz}
T.~Bachmann, \emph{On the conservativity of the functor assigning to a motivic
  spectrum its motive}, Duke Math. J. \textbf{167} (2018), no.~8,
  pp.~1525--1571

\bibitem[BF18]{bachmann-criterion}
T.~Bachmann and J.~Fasel, \emph{On the effectivity of spectra representing
  motivic cohomology theories}, 2018,
  \href{http://arxiv.org/abs/1710.00594v3}{{\sf arXiv:1710.00594v3}}

\bibitem[BH18]{bachmann-norms}
T.~Bachmann and M.~Hoyois, \emph{Norms in Motivic Homotopy Theory}, 2018,
  \href{http://arxiv.org/abs/1711.03061v4}{{\sf arXiv:1711.03061v4}}

\bibitem[Blo86]{bloch1986algebraic}
S.~Bloch, \emph{Algebraic cycles and higher K-theory}, Adv. Math. \textbf{61}
  (1986), no.~3, pp.~267--304

\bibitem[CF17]{calmes2014finite}
B.~Calm{\`e}s and J.~Fasel, \emph{The category of finite $MW$-correspondences},
  2017, \href{http://arxiv.org/abs/1412.2989v2}{{\sf arXiv:1412.2989v2}}

\bibitem[CTHK97]{bloch-ogus-gabber}
J.-L. Colliot-Thélène, R.~Hoobler, and B.~Kahn, \emph{The Bloch-Ogus-Gabber
  theorem}, Fields Inst. Commun. \textbf{16} (1997), pp.~31--94

\bibitem[DF17]{deglise-fasel}
F.~Déglise and J.~Fasel, \emph{$MW$-motivic complexes}, 2017,
  \href{http://arxiv.org/abs/1708.06095}{{\sf arXiv:1708.06095}}

\bibitem[DJK18]{DJK}
F.~Déglise, F.~Jin, and A.~A. Khan, \emph{Fundamental classes in motivic
  homotopy theory}, 2018, \href{http://arxiv.org/abs/1805.05920}{{\sf
  arXiv:1805.05920}}

\bibitem[DK19]{DruzhininKyllingOddChar}
A.~Druzhinin and J.~Kylling, \emph{Framed correspondences and the zeroth stable
  motivic homotopy group in odd characteristic}, 2019,
  \href{http://arxiv.org/abs/1809.03238v3}{{\sf arXiv:1809.03238v3}}

\bibitem[DP18]{DruzhininPaninChar2}
A.~Druzhinin and I.~Panin, \emph{Surjectivity of the \'etale excision map for
  homotopy invariant framed presheaves}, 2018,
  \href{http://arxiv.org/abs/1808.07765}{{\sf arXiv:1808.07765}}

\bibitem[Dé07]{deglise-regular-base}
F.~Déglise, \emph{Finite correspondences and transfers over a regular base},
  Algebraic cycles and motives. Volume 1, Lecture Note Ser., vol. 343, London
  Math. Soc., 2007

\bibitem[EHK{\etalchar{+}}18]{EHKSY2}
E.~Elmanto, M.~Hoyois, A.~A. Khan, V.~Sosnilo, and M.~Yakerson, \emph{Framed
  transfers and motivic fundamental classes}, to appear in J.~Topol., 2018,
  \href{http://arxiv.org/abs/1809.10666}{{\sf arXiv:1809.10666}}

\bibitem[EHK{\etalchar{+}}19]{EHKSY}
\bysame, \emph{Motivic infinite loop spaces}, 2019,
  \href{http://arxiv.org/abs/1711.05248v5}{{\sf arXiv:1711.05248v5}}

\bibitem[GP18]{GarkushaPaninStrictHomInv}
G.~Garkusha and I.~Panin, \emph{Homotopy invariant presheaves with framed
  transfers}, 2018, \href{http://arxiv.org/abs/1504.00884v3}{{\sf
  arXiv:1504.00884v3}}

\bibitem[{Gro}67]{EGAIV}
A.~{Grothendieck}, \emph{{\'El\'ements de g\'eom\'etrie alg\'ebrique. IV:
  \'Etude locale des sch\'emas et des morphismes de sch\'emas. R\'edig\'e avec
  la colloboration de Jean Dieudonn\'e.}}, {Publ. Math., Inst. Hautes \'Etud.
  Sci.} \textbf{32} (1967), pp.~1--361

\bibitem[Hoy15]{hoyois-algebraic-cobordism}
M.~Hoyois, \emph{From algebraic cobordism to motivic cohomology}, J. Reine
  Angew. Math. \textbf{702} (2015), pp.~173--226

\bibitem[Hoy17]{hoyois-equivariant}
\bysame, \emph{The six operations in equivariant motivic homotopy theory}, Adv.
  Math. \textbf{305} (2017), pp.~197--279

\bibitem[Ive12]{iversen2012cohomology}
B.~Iversen, \emph{Cohomology of sheaves}, Springer Science \& Business Media,
  2012

\bibitem[Kai18]{kai2015moving}
W.~Kai, \emph{A moving lemma for algebraic cycles with modulus and
  contravariance}, 2018, \href{http://arxiv.org/abs/1507.07619v4}{{\sf
  arXiv:1507.07619v4}}

\bibitem[Lev98]{levine1998mixed}
M.~Levine, \emph{Mixed motives}, vol.~57, American Mathematical Soc., 1998

\bibitem[Lev06]{levine2006chow}
\bysame, \emph{Chow’s moving lemma and the homotopy coniveau tower}, K-theory
  \textbf{37} (2006), no.~1, pp.~129--209

\bibitem[Lev08]{levine2008homotopy}
\bysame, \emph{The homotopy coniveau tower}, J. Topol. \textbf{1} (2008),
  no.~1, pp.~217--267

\bibitem[Lev10]{levine2010slices}
\bysame, \emph{Slices and transfers}, Doc. Math. \textbf{Extra Vol.} (2010),
  pp.~393--443

\bibitem[Lev11]{levine-slice}
M.~Levine, \emph{The slice filtration and Grothendieck-Witt groups}, Pure Appl.
  Math. Q. \textbf{7} (2011), no.~4, Special Issue: In memory of Eckart Viehweg

\bibitem[{Lev}18]{levine-enum}
M.~{Levine}, \emph{{Toward an enumerative geometry with quadratic forms}},
  2018, \href{http://arxiv.org/abs/1703.03049v3}{{\sf arXiv:1703.03049v3}}

\bibitem[Lur17a]{HA}
J.~Lurie, \emph{Higher Algebra}, September 2017,
  \url{http://www.math.harvard.edu/~lurie/papers/HA.pdf}

\bibitem[Lur17b]{HTT}
\bysame, \emph{Higher Topos Theory}, April 2017,
  \url{http://www.math.harvard.edu/~lurie/papers/HTT.pdf}

\bibitem[Mor03]{morel-trieste}
F.~Morel, \emph{An introduction to $\mathbb{A}^1$-homotopy theory}, ICTP
  Trieste Lecture Note Ser. 15 (2003), pp.~357--441

\bibitem[Mor12]{A1-alg-top}
\bysame, \emph{$\mathbb{A}^1$-Algebraic Topology over a Field}, Lecture Notes
  in Mathematics, Springer Berlin Heidelberg, 2012

\bibitem[MV99]{A1-homotopy-theory}
F.~Morel and V.~Voevodsky, \emph{$\mathbb{A}^1$-homotopy theory of schemes},
  Publ. Math. l.H.É.S. \textbf{90} (1999), no.~1, pp.~45--143

\bibitem[Ros96]{rost1996chow}
M.~Rost, \emph{Chow groups with coefficients}, Doc. Math. \textbf{1} (1996),
  no.~16, pp.~319--393

\bibitem[Stacks]{stacks}
{Stacks Project Authors}, \emph{The Stacks Project}, 2018,
  \url{http://stacks.math.columbia.edu}

\bibitem[Voe01]{voevodsky2001notes}
V.~Voevodsky, \emph{Notes on framed correspondences}, unpublished, 2001,
  \url{http://www.math.ias.edu/vladimir/files/framed.pdf}

\bibitem[Voe02]{voevodsky-slice-filtration}
V.~Voevodsky, \emph{Open Problems in the Motivic Stable Homotopy Theory , I},
  International Press Conference on Motives, Polylogarithms and Hodge Theory,
  International Press, 2002

\bibitem[WW17]{wickelgren2014simplicial}
K.~Wickelgren and B.~Williams, \emph{The Simplicial EHP Sequence in
  $\mathbb{A}^1$-Algebraic Topology}, 2017,
  \href{http://arxiv.org/abs/1411.5306v2}{{\sf arXiv:1411.5306v2}}

\end{thebibliography}
\end{document}